\documentclass[A4]{amsart}
\usepackage[square,sort,comma,numbers]{natbib}
\usepackage{amsmath, amssymb, amstext, amsfonts, textcomp, amsxtra, amsbsy, amsgen, amsopn, amscd, mathrsfs, amsthm, latexsym, array}
\usepackage[textwidth=16cm,textheight=22cm,centering]{geometry}

\usepackage{enumerate}

\usepackage{mathdots, bbm, dsfont}
\usepackage{color}
\usepackage{graphicx}

\usepackage{xcolor}
\definecolor{cite}{rgb}{0.00,0.60,1.00}
\definecolor{url}{rgb}{1.00,0.10,0.80}
\definecolor{link}{rgb}{0.00,0.00,1.00}
\usepackage[colorlinks,linkcolor=link,urlcolor=url,citecolor=cite,pagebackref,breaklinks]{hyperref}

\def\leq{\leqslant}

\def\geq{\geqslant}

\allowdisplaybreaks

\usepackage[all]{xy}

\newtheorem{theorem}             {Theorem}  [section]
\newtheorem{definition} [theorem] {Definition}
\newtheorem{lemma}      [theorem]{Lemma}
\newtheorem{corollary}  [theorem]{Corollary}
\newtheorem{proposition}[theorem]{Proposition}

\numberwithin{equation}{section} 

\theoremstyle{remark}
\newtheorem{remark}{\bf Remark}

\newcommand{\Cont}{{\rm C}}

\newcommand{\Sob}{{\rm S}}
\newcommand{\Sch}{\mathcal{S}}
\newcommand{\SSch}{\mathcal{S}_{\mathrm{sis}}}
\newcommand{\SMel}{\mathfrak{M}_{\mathrm{sis}}}

\newcommand{\intL}{{\rm L}}

\newcommand{\Nr}{{\rm Nr}}
\newcommand{\Tr}{{\rm Tr}}

\newcommand{\gp}[1]{\mathbf{#1}}
\newcommand{\GL}{{\rm GL}}
\newcommand{\PGL}{{\rm PGL}}
\newcommand{\SL}{{\rm SL}}
\newcommand{\SO}{{\rm SO}}
\newcommand{\SU}{{\rm SU}}

\newcommand{\ud}{\mathrm{d}}

\newcommand{\ag}[1]{\mathbb{#1}}

\newcommand{\Z}{\mathbb{Z}}

\newcommand{\Mat}{{\rm M}}
\newcommand{\id}{\mathbbm{1}}

\makeatletter
\def\legendre@dash#1#2{\hb@xt@#1{%
  \kern-#2\p@
  \cleaders\hbox{\kern.5\p@
    \vrule\@height.2\p@\@depth.2\p@\@width\p@
    \kern.5\p@}\hfil
  \kern-#2\p@
  }}
\def\@legendre#1#2#3#4#5{\mathopen{}\left(
  \sbox\z@{$\genfrac{}{}{0pt}{#1}{#3#4}{#3#5}$}%
  \dimen@=\wd\z@
  \kern-\p@\vcenter{\box0}\kern-\dimen@\vcenter{\legendre@dash\dimen@{#2}}\kern-\p@
  \right)\mathclose{}}
\newcommand\legendre[2]{\mathchoice
  {\@legendre{0}{1}{}{#1}{#2}}
  {\@legendre{1}{.5}{\vphantom{1}}{#1}{#2}}
  {\@legendre{2}{0}{\vphantom{1}}{#1}{#2}}
  {\@legendre{3}{0}{\vphantom{1}}{#1}{#2}}
}
\def\dlegendre{\@legendre{0}{1}{}}
\def\tlegendre{\@legendre{1}{0.5}{\vphantom{1}}}
\makeatother

\newcommand{\Q}{\mathbb{Q}}
\newcommand{\R}{\mathbb{R}}
\newcommand{\C}{\mathbb{C}}
\newcommand{\E}{\mathbf{E}}
\newcommand{\F}{\mathbf{F}}

\newcommand{\bF}{\mathbf{F}}

\newcommand{\A}{\mathbb{A}}

\newcommand{\vO}{\mathcal{O}}
\newcommand{\vo}{\mathfrak{o}}
\newcommand{\vP}{\mathcal{P}}
\newcommand{\vp}{\mathfrak{p}}
\newcommand{\idlJ}{\mathfrak{J}}

\newcommand{\norm}[1][\cdot]{\lvert #1 \rvert}
\newcommand{\extnorm}[1]{\left\lvert #1 \right\rvert}
\newcommand{\Norm}[1][\cdot]{\lVert #1 \rVert}
\newcommand{\extNorm}[1]{\left\lVert #1 \right\rVert}

\newcommand{\Pairing}[2]{\langle #1, #2 \rangle}

\newcommand{\OFour}{\mathfrak{F}}
\newcommand{\invOFour}{\overline{\mathfrak{F}}}

\newcommand{\Mellin}[2][]{\mathfrak{M}_{#1}(#2)}
\newcommand{\Rem}{\mathrm{r}}
\newcommand{\Ext}{\mathrm{e}}
\newcommand{\Inv}{\mathrm{i}}
\newcommand{\Mult}{\mathfrak{m}}
\newcommand{\Vor}{\mathcal{V}}
\newcommand{\VorH}{\mathcal{VH}}
\newcommand{\VHF}{\mathfrak{vh}}
\newcommand{\MS}{\mathcal{MS}}
\newcommand{\Trans}{\mathfrak{t}}

\newcommand{\rpL}{{\rm L}}
\newcommand{\rpR}{{\rm R}}
\newcommand{\Bas}{\mathcal{B}}

\newcommand{\Res}{{\rm Res}}
\newcommand{\Ind}{{\rm Ind}}
\newcommand{\cInd}{{\rm c-Ind}}

\newcommand{\Whi}{\mathcal{W}}
\newcommand{\Kir}{\mathcal{K}}

\newcommand{\fin}{{\rm fin}}
\newcommand{\eis}{{\rm E}}

\newcommand{\Zeta}{\mathrm{Z}}
\newcommand{\BesselD}{\mathrm{J}}
\newcommand{\BesselF}{\mathrm{j}}

\newcommand{\RamCst}{\vartheta}
\newcommand{\AutFam}{\mathcal{F}}

\newcommand{\Sph}{\mathbf{S}}

\newcommand{\Vol}{{\rm Vol}}
\makeatletter

\newcommand{\Rmnum}[1]{\expandafter\@slowromancap\romannumeral #1@}
\makeatother

\newcommand{\DBZ}[5]{\mathrm{Z} \left( \begin{matrix} #1 \\ #2 \end{matrix}, \begin{matrix} #3 \\ #4 \end{matrix}; #5 \right)}
\newcommand{\DBZLoc}[6][v]{\mathrm{Z}_{#1} \left( \begin{matrix} #2 \\ #3 \end{matrix}, \begin{matrix} #4 \\ #5 \end{matrix}; #6 \right)}

\title{On A Generalization of Motohashi's Formula}

\author{Han Wu}
\address{School of Mathematical Sciences, University of Scinece and Technology of China, 230026 Hefei, P. R. China}
\email{wuhan1121@yahoo.com}

\date{\today}

\begin{document}



\begin{abstract}
	We study a spectral reciprocity formula relating $\GL_3 \times \GL_2$ with $\GL_3 \times \GL_1$ and $\GL_1$ moments of $L$-functions discovered by Kwan. Globally we give an adelic and distributional treatment. Our test automorphic function is of general type. To achieve this generality we develop an extension of the generalized Godement sections. Locally we give the weight function transforms in both directions for the fixed tempered representation $\Pi$ of $\GL_3(\F)$. We obtain the transform by a theory of the Voronoi--Hankel transforms, which extends Miller--Schmid's local theory of the Voronoi formula for $\GL_n$.
\end{abstract}

	\maketitle
	
	\tableofcontents

\section{Introduction}

	\subsection{Motohashi's Formula and Generalizations}
	
	Let $\gp{G}_j$ with $j \in \{ 1,2 \}$ be reductive groups, say defined over $\Q$ with adele ring $\A$. Let $\AutFam_j$ be some family of automorphic representations of $\gp{G}_j(\A)$, and $\mathcal{L}_j(\pi_j)$ be some $L$-value parametrized by $\pi_j \in \AutFam_j$. The term \emph{spectral reciprocity} was introduced by Blomer, Li and Miller \cite{BLM19} to name identities roughly of the shape
\begin{equation} \label{eq: RoughSpecR}
	 \sum_{\pi_1 \in \AutFam_1} h_1(\pi_1) \mathcal{L}_1(\pi_1) = \sum_{\pi_2 \in \AutFam_2} h_2(\pi_2) \mathcal{L}_2(\pi_2), 
\end{equation}
where $h_j$ are \emph{weight functions}. Such identities have been playing important roles in the moment problem and the subconvexity problem for automorphic $L$-functions.

	For example, Motohashi's formula \cite{Mo93} relates the family with $\gp{G}_1=\PGL_2$, $\mathcal{L}_1(\pi) = L(\tfrac{1}{2},\pi)^3$ and the family with $\gp{G}_2=\GL_1$, $\mathcal{L}_2(\chi)=\extnorm{L(\tfrac{1}{2},\chi)}^4$. It is historically the first instance of spectral reciprocity. The exploitation of Motohashi's formula has led to a fine asymptotic formula of the fourth moement of the Riemann zeta function \cite[Theorem 5.2]{Mo97}, and possibly more strikingly the uniform Weyl-type subconvex bound for all Dirichlet characters due to Petrow--Young \cite{PY19_CF, PY19_All}, which non-trivially extends and generalizes the method of Conrey--Iwaniec \cite{CI00}.
	
	One direction of explanation and generalization is due to Reznikov \cite{Re08} and Michel--Venkatesh \cite{MV10}, who use the period approach and put the following instance of \emph{strong Gelfand configuration} as the underlying mechanism
\begin{equation} \label{eq: MVGraph}
\begin{matrix}
	& & \GL_2 \times \GL_2 & & \\
	& \nearrow & & \nwarrow & \\
	\GL_1 \times \GL_1 & & & & \GL_2 \\
	& \nwarrow & & \nearrow & \\
	& & \GL_1 & &
\end{matrix}.
\end{equation}
	Nelson \cite{Ne20} addresses the convergence issue in the above formalism over any number field with a theory of regularized integrals. We non-trivially improves Nelson's treatment by giving a full analysis of the degenerate terms with the method of meromorphic continuation (see \cite{Wu22} and \cite[Appendix (arXiv version 2)]{WX23+}). Another interpretation via the \emph{Godement--Jacquet pre-trace formula}, a type of pre-trace formula whose test functions are Schwartz functions on the $2$ by $2$ matrices $\Mat_2$ instead of $\GL_2$ is also given in \cite{Wu22}. This new interpretation has a different perspective of generalization from the period approach. We may regard these directions of generalization as ``balanced'' ones, since on both sides there are only $L$-functions of the same degree.
	
	Another direction of explanation and generalization, first appeared in Kwan's paper \cite{Kw24} (towards which Conrey--Iwaniec \cite[Introduction]{CI00} mentioned by simples words of ``harmonic analysis on $\GL_3$''), put another instance of strong Gelfand configuration
\begin{equation} \label{eq: KwanGraph}
\begin{matrix}
	& & \GL_3 & & \\
	& \nearrow & & \nwarrow & \\
	\GL_2 & & & & \gp{U}_3 \\
	& \nwarrow & & \nearrow & \\
	& & \gp{U}_2 & &
\end{matrix}, \qquad
\begin{matrix}
	& & \begin{pmatrix} 1 & \\ & g \end{pmatrix} & & \\
	& \nearrow & & \nwarrow & \\
	g & & & & * \\
	& \nwarrow & & \nearrow & \\
	& & * & &
\end{matrix}
\end{equation}
where $\gp{U}_k$ is the unipotent radical of the standard (upper triangular) minimal parabolic subgroup of $\GL_k$ (not a unitary group), underneath. We may refer to this direction as an ``unbalanced'' one. Intuitively, for a fixed automorphic representation $\Pi$ of $\GL_3$ the graph (\ref{eq: KwanGraph}) replaces the cubic moment side $L(s,\pi)^3$ with $L(s, \Pi \times \pi)$, and the fourth moment side $L(s,\chi)^4$ with a mixed moment $L(s,\Pi \times \chi)L(s,\chi)$, which justifies the adjective ``unbalanced''. Previous exploitation (without explicit spectral reciprocity formula for \eqref{eq: KwanGraph}) in this direction includes some good subconvex bounds for $\PGL_3$ $L$-functions in the $t$-aspect and for self-dual $\GL_3 \times \GL_2$ $L$-functions in the spectral aspect for the $\GL_2$ part by Li \cite{Li11}, their generalization over number fields by Qi \cite{Qi19, Qi24}. Some nice improvements over $\Q$ making use of the above spectral reciprocity \eqref{eq: KwanGraph} can be found in Lin--Nunes--Qi \cite{LNQ22} and Ganguly--Humphries--Lin--Nunes \cite{GHLN24+}.

	Note that the above unbalanced direction goes outside the world of reductive groups and uses the strong Gelfand-pair property offered by the uniqueness of Whittaker functionals. Note also that its relevance to the original Motohashi's formula is indicated by the identification
\begin{equation} \label{eq: UnbGenLId}
	L(s,\pi)^3 = L(s, (\id \boxplus \id \boxplus \id) \times \pi),
\end{equation}
	where $\id \boxplus \id \boxplus \id$ is the representation of $\GL_3$ induced from the trivial character of a Borel subgroup.

	\subsection{Main Results}
	
	This is the first of a series papers concerning an extensive study of the spectral reciprocity formula discovered by Kwan \cite{Kw24}, and its application to the subconvexity problem \emph{over number fields}. We focus on the global formula and the local weight transforms in the current paper.
	
		\subsubsection{Global Part}
		
	The graph (\ref{eq: KwanGraph}) appeared in Kwan's paper \cite{Kw24}, and was nicely explained in \cite[\S 4]{Kw24} in the case for $\SL_3(\Z)$. Kwan's test automorphic function is restricted to spherical ones at the real place. We follow Kwan's graph to treat general test automorphic functions in the adelic setting, and obtain a distributional version in the flavour of our previous work \cite{Wu22}.
	
	Let $\F$ be a number field with ring of adeles $\A$. Let $\psi$ be the additive character of $\F \backslash \A$ \`a la Tate. Fix a cuspidal automorphic representation $\Pi$ of $\GL_3(\A)$ with automorphic realization $V_{\Pi} \subset \intL_0^2(\GL_3,\omega_{\Pi})$, and a unitary Hecke character $\omega$ of $\F^{\times} \backslash \A^{\times}$. For every irreducible representation $*$ of $\GL_d(\F)$ we write $\omega_*$ for its central character. We assume, without loss of generality, that both $\omega_{\Pi}$ and $\omega$ are trivial on $\R_+$, which is embedded in $\A^{\times}$ via a fixed section map of the adelic norm map $\A^{\times} \to \R_+, x=(x_v)_v \mapsto \norm_{\A}=\sideset{}{_v} \prod \norm[x_v]_v$.

\begin{theorem}[$s_0=0$ of Theorem \ref{thm: MainId}] \label{thm: Main1}
	There is a distribution (continuous functional) $\Theta$ on $V_{\Pi}^{\infty}$, which admits two different decompositions:
\begin{align*}
	\Theta(F) &= \frac{1}{\zeta_{\F}^*} \sum_{\chi \in \widehat{\F^{\times} \R_+ \backslash \A^{\times}}} \int_{-\infty}^{\infty} \DBZ{1/2+i\tau}{1/2-i\tau}{\chi}{(\chi \omega \omega_{\Pi})^{-1}}{W_F} \frac{\ud \tau}{2\pi} \\
	&\quad + \frac{1}{\zeta_{\F}^*} \Res_{s_1=\frac{1}{2}} \DBZ{1/2+s_1}{1/2-s_1}{(\omega \omega_{\Pi})^{-1}}{\mathbbm{1}}{W_F} - \frac{1}{\zeta_{\F}^*} \Res_{s_1=-\frac{1}{2}} \DBZ{1/2+s_1}{1/2-s_1}{(\omega \omega_{\Pi})^{-1}}{\mathbbm{1}}{W_F},
\end{align*}
	where $W_F \in \Whi(\Pi^{\infty},\psi)$ is the Whittaker function associated with $F \in V_{\Pi}^{\infty}$ and $\Zeta(\cdots)$ is an integral representation of $L(1/2+i\tau, \Pi \times \chi) L(1/2-i\tau, (\chi \omega \omega_{\Pi})^{-1})$;
	$$ \Theta(F) = \sideset{}{_{\substack{\pi \text{ cuspidal of } \GL_2 \\ \omega_{\pi}=\omega^{-1}}}} \sum \Theta(F \mid \pi) + \sum_{\chi \in \widehat{\R_+ \F^{\times} \backslash \A^{\times}}} \int_{-\infty}^{\infty} \Theta(F \mid \pi_{i\tau}(\chi,\omega^{-1}\chi^{-1})) \frac{\ud \tau}{4\pi}, $$
	where $\Theta(F \mid \pi)$ (resp. $\Theta(F \mid \pi_{i\tau}(\chi,\omega^{-1}\chi^{-1}))$) is an integral representation of $L(1/2, \Pi \times \widetilde{\pi})$ (resp. $L(1/2-i\tau, \Pi \times \chi^{-1}) L(1/2+i\tau, \Pi \times \omega\chi)$).
\end{theorem}

\noindent Morally one should understand the above formula as (see \eqref{eq: Main1Bis} for the precise form)
\begin{align*}
	&\quad \sum_{\pi} L(1/2, \Pi \times \widetilde{\pi}) \cdot \mathfrak{L} \cdot \prod_{v \mid \infty} h_v(\pi_v) \prod_{\vp < \infty} h_{\vp}(\pi_{\vp}) \mathfrak{L}_{\vp}^{-1} + (\textrm{CSC}) \\
	&= \frac{1}{\zeta_{\F}^*} \int_{\widehat{\F^{\times} \backslash \A^{\times}}} L(1/2, \widetilde{\Pi} \times \chi^{-1}) L(1/2, \chi) \cdot \prod_{v \mid \infty} \widetilde{h}_v(\chi_v) \prod_{\vp < \infty} \widetilde{h}_{\vp}(\chi_{\vp}) \mathcal{L}_{\vp}^{-1} \ud \chi + (\textrm{DT}),
\end{align*}
	where the various $L$-factors $\mathfrak{L}, \mathfrak{L}_{\vp}^{-1}$ and $\mathcal{L}_{\vp}^{-1}$ are negligible in practice. Note that for every place $v$ of $\F$ we have introduced a pair of \emph{weight functions} $h_v(\pi_v)$ and $\widetilde{h}_v(\chi_v)$. They are given explicitly in terms of continuous functionals on the local Whittaker models $\Whi(\Pi_v^{\infty}, \psi_v)$ in \eqref{eq: WtGL2CuspDBis} and \eqref{eq: WtGL2EisDBis} below. The mutual determination of $h_v(\pi_v)$ and $\widetilde{h}_v(\chi_v)$ will be the major concern of three follow-up papers.
	
	We emphasize that the idea of using the graph \eqref{eq: KwanGraph} of strong Gelfand configuration to explain the underlying spectral reciprocity formulae is due to Kwan. Hence it is not a novelty of this paper. However extending Kwan's treatment to automorphic forms of general type (from the spherical ones) requires establishing the analytic properties of the \emph{double zeta-integrals}
	$$ \C^2 \ni (s_1, s_2) \mapsto \DBZ{s_1}{s_2}{\chi_1}{\chi_2}{W_F} $$
(defined by \eqref{DZInt} below) in the general case. This poses non-trivial difficulty. We overcome it by an \emph{extension} (see \S \ref{sec: GGS}) of some integral representation of $W_F$, called the \emph{generalized Godement sections} in \cite[\S 7.1]{J09}. This extension has its own interests. 

		\subsubsection{Local Part}
		
	We now turn to the mutual determination of the local weight functions $h_v(\pi_v)$ and $\widetilde{h}_v(\chi_v)$. We shall omit the subscript $v$ for simplicity. This problem turns out to be intimately related with the local \emph{Voronoi--Hankel transforms}, which we now define/recall as follows.
	
	Let $\F$ be a local field. Let $n \in \Z_{\geq 2}$. Let $\pi$ be a unitary irreducible representation of $\GL_n(\F)$ which is generic and $\RamCst$-tempered for some constant $0 \leq \RamCst < 1/2$. Let $W \in \Whi(\pi^{\infty},\psi)$ be a function in the Whittaker model of $\pi^{\infty}$. Let $w_n$ be the longest Weyl element of $\GL_n$. Then the function
	$$ \widetilde{W}(h) := W(w_n h^{\iota}), \quad \forall \ h \in \GL_n(\F) $$
is in $\Whi(\widetilde{\pi}^{\infty}, \psi^{-1})$, the Whittaker model of the smooth contragredient representation $\widetilde{\pi}^{\infty}$. We introduce some elementary operators on the space of functions on $\F^{\times}$:
\begin{itemize}
	\item For functions $\phi$ on $\F^{\times}$, its \emph{extension} by $0$ to $\F$ is denoted by $\Ext(\phi)$, and its \emph{inverse} is $\mathrm{Inv}(\phi)(t) := \phi(t^{-1})$; for functions $\phi$ on $\F$, its \emph{restriction} to $\F^{\times}$ is denoted by $\Rem(\phi)$, and the operator $\Inv$ is 
	$$ \Inv = \Ext \circ \mathrm{Inv} \circ \Rem. $$
	\item For $s \in \C$, $\mu \in \widehat{\F^{\times}}$ and functions $\phi$ on $\F$, we introduce the operator $\Mult_s(\mu)$ by
	$$ \Mult_s(\mu)(\phi)(t) = \phi(t) \mu(t) \norm[t]_{\F}^s. $$
	\item For $\delta \in \F^{\times}$ we introduce the operator $\Trans(\delta)$ by
	$$ \Trans(\delta)(\phi)(y) = \phi(y \delta). $$
\end{itemize}

\begin{definition} \label{def: VoronoiGLn}
	(0) If $n \geq 2$ and $0 \leq j \leq n-2$, the space of functions on $\F^{\times}$ (all containing $\Cont_c^{\infty}(\F^{\times})$)
\begin{equation} \label{eq: BasicVHS}
	\VorH(\pi) = \VorH(\pi; j) := \left\{ h(y) := \norm[y]^{-\frac{n-1}{2}} \int_{\F^j} W \begin{pmatrix} y & & \\ \vec{x} & \id_j & \\ & & \id_{n-1-j} \end{pmatrix} \ud \vec{x} \ \middle| \ W \in \Whi(\pi^{\infty}, \psi) \right\}
\end{equation}
	is independent of $j$ or $\psi$ (as long as $\psi$ is non-trivial).

\noindent (1) Let $n \in \Z_{\geq 2}$. The transform from the function $H$ to the function $H^*$ defined by
	$$ H(y) = \int_{\F^{n-2}} W \begin{pmatrix} y & & \\ \vec{x} & \id_{n-2} & \\ & & 1 \end{pmatrix} \ud \vec{x}, \quad H^*(y) = \widetilde{W} \begin{pmatrix} y & \\ & w_{n-1} \end{pmatrix} $$
	is the \emph{Voronoi transform} for $\pi$, written as $\Vor_{\pi}$, namely $H^*(y) := \Vor_{\pi}(H)(y)$. We also introduce the \emph{Voronoi--Hankel transform} as
	$$ \VorH_{\pi} := \Trans((-1)^{n-1}) \circ \Mult_{-\frac{n-3}{2}} \circ \Vor_{\pi} \circ \Mult_{\frac{n-1}{2}}, $$
	so that the local functional equation can be written as (independently of the rank $n$)
\begin{equation} \label{eq: LocFEGLnGL1}
	\int_{\F^{\times}} \VorH_{\pi}(H')(t) \chi^{-1}(t) \norm[t]^{-s} \ud^{\times} t = \gamma(s, \pi \times \chi, \psi) \int_{\F^{\times}} H'(t) \chi(t) \norm[t]^s \ud^{\times}t, 
\end{equation}
	where we have written $H'(t) := H(t) \norm[t]^{-\frac{n-1}{2}}$.
	
\noindent (2) Let $\Sch(\F)$ be the space of Schwartz--Bruhat functions. Let $\chi$ be a (quasi-character) of $\F^{\times}$. The Voronoi--Hankel transform $\VorH_{\chi}$ is the composition $\VorH_{\chi} = \Mult_1(\chi^{-1}) \circ \invOFour \circ \Mult_0(\chi^{-1})$ on
\begin{equation} \label{eq: LocFEGL1}
	\VorH_{\chi}: \chi \cdot \Sch(\F) \to \chi^{-1} \norm \cdot \Sch(\F). 
\end{equation}
	In this case we put $\VorH(\chi) := \chi \cdot \Sch(\F)$.
\end{definition}

\noindent Note that in the case $n=1$ we can rewrite the transform $\VorH_{\chi}$ in \eqref{eq: LocFEGL1} as
\begin{equation} \label{eq: LocFEGL1Bis}
	\Mult_{-1}(\chi) \circ \VorH_{\chi} = \invOFour \circ \Mult_0(\chi^{\iota}). 
\end{equation}

\noindent We extend the Voronoi--Hankel transform $\VorH_{\pi}$ to $\widetilde{\VorH}_{\pi}$ in terms of an analogue of \eqref{eq: LocFEGL1Bis}. This extension has its own interests: it shows that higher rank Voronoi--Hankel transforms are essentially Fourier transforms in higher dimensional affine spaces. In passing we complement Miller--Schmid's theory for non-archimedean local fields. In particular, we give in \S \ref{sec: MSTheory} an intermediate extension $\MS_{\pi}$ of $\VorH_{\pi}$ on $\Inv(\SSch(\F))$, the space of functions $\Inv(h)(t) := h(t^{-1})$ for $h \in \SSch(\F)$ (see Definition \ref{def: SSchNA}), and verify its consistency with $\widetilde{\VorH}_{\pi}$. 
\begin{theorem} \label{thm: ExtVorH}
	(1) Let $n \in \Z_{\geq 1}$. Let $\pi$ be an irreducible \emph{smooth} and generic representation of $\GL_n(\F)$. Let $\invOFour$ be the distributional inverse Fourier transform on $\Mat_n(\F)$. Let $I_n: \Cont(\F^{\times}) \to \Cont(\GL_n(\F))$ be given by $I_n(h)(g) := h(\det g)$. For any \emph{smooth matrix coefficient} $\beta$ of $\pi$ we have
	$$ \Mult_{-\frac{n+1}{2}}(\beta) \circ I_n \circ \MS_{\pi} = \invOFour \circ \Mult_{-\frac{n-1}{2}}(\beta^{\iota}) \circ I_n \mid_{\Inv(\SSch(\F))}. $$
	
\noindent (2) If in addition $\pi = \Pi^{\infty}$ is the subspace of smooth vectors in a \emph{tempered unitary} representation $\Pi$, then we have the equality on $\VorH(\pi)$
	$$ \Mult_{-\frac{n+1}{2}}(\beta) \circ I_n \circ \VorH_{\pi} = \invOFour \circ \Mult_{-\frac{n-1}{2}}(\beta^{\iota}) \circ I_n. $$
	Moreover, the following equation for a pair of functions $H, H^* \in \Cont(\F^{\times})$
	$$ \Mult_{-\frac{n+1}{2}}(\beta) \circ I_n(H^*) = \invOFour \circ \Mult_{-\frac{n-1}{2}}(\beta^{\iota}) \circ I_n(H), $$
	uniquely determines $\widetilde{\VorH}_{\pi}(H):=H^*$ as $\beta$ traverses the set $C(\pi)$ of smooth matrix coefficients of $\pi$.
\end{theorem}

	Just like Motohashi's formula in our former work \cite{Wu22}, there is a hidden pivot geometric side in our formula. Its local terms are realized as the \emph{relative orbital integrals} $H(y)$ for the \emph{Bessel distributions}. Namely we have (see \eqref{eq: WtIsBesselT} for a more precise version)
	$$ h(\pi) = \int_{\F^{\times}} H(y) \cdot \BesselF_{\widetilde{\pi},\psi^{-1}} \begin{pmatrix} & -y \\ 1 & \end{pmatrix} \frac{\ud^{\times}y}{\norm[y]}, $$
and one can find $H(y)$ in terms of $h(\pi)$ be the Bessel inversion formula. Consequently we regard the mutual determination of $\pi \mapsto h(\pi)$ and $y \mapsto H(y)$ as \emph{theoretically} well-understood. It remains to see the mutual determination of $y \mapsto H(y)$ and $\chi \mapsto \widetilde{h}(\chi)$.

\begin{theorem} \label{thm: DWtFTemp}
	Let $\widetilde{\Vor}_{\pi}$ be the extension of $\Vor_{\pi}$ corresponding to $\widetilde{\VorH}_{\pi}$. For tempered unitary $\Pi$ we have
	$$ \widetilde{h}(\chi) = \int_{\F^{\times}} \psi(-y) \chi^{-1}(y) \norm[y]^{-\frac{1}{2}} \cdot \widetilde{\Vor}_{\Pi}(H)(y) \ud^{\times} y, $$
where the integral on the right hand side is absolutely convergent.
\end{theorem}

\begin{remark}
	Our method leading to Theorem \ref{thm: DWtFTemp} lies in exploring the proof of Jacquet's conjecture on the local converse theorems via Chen \cite{Ch06} and Jacquet--Liu \cite{JL18}. This method is more general than the Voronoi formula approach. This will be made explicit in an upcoming paper on a further generalization replacing $\GL_3$ with $\GL_n$. In particular in the case of $n=4$ the local functional equations for twists by $\GL_2$-representations are necessary.
\end{remark}

\begin{remark}
	Clearly the transforms $H \mapsto \widetilde{\VorH}_{\Pi}(H)$ and $\widetilde{\Vor}_{\Pi}(H) \mapsto \widetilde{h}(\chi)$ given in Theorem \ref{thm: ExtVorH} (2) \& \ref{thm: DWtFTemp} are invertible. Hence our formulas give the weight transforms in both directions, at least for tempered $\Pi$.
\end{remark}

\begin{remark}
	Under the Ramanujan--Petersson conjecture (for $\GL_3$) every local component $\Pi$ of a generic automorphic representation is tempered. Therefore Theorem \ref{thm: ExtVorH} \& \ref{thm: DWtFTemp} cover all interesting cases related to automorphic setting. On the other hand, their extension to the non-tempered case has independent theoretic interests and poses non-trivial technical challenges. See Remark \ref{rmk: NonTempBd} for more details.
\end{remark}

\begin{remark}
	In a recent work Jiang--Luo defined $\VorH_{\pi}$ via the Godement--Jacquet theory through the graph \cite[3-16]{JiL23}. Note that their $\Sch_{\pi}(\F^{\times})$ is essentially our $\VorH(\pi)$ up to the consistency between the Godement--Jacquet theory and the Rankin--Selberg theory for $\GL_n \times \GL_1$. The formulation of Theorem \ref{thm: ExtVorH} is inspired by the theirs (as well as some earlier results of Jacquet et al. \cite[Proposition 4.5]{JL70} and \cite[\S (2.4)]{JS85} for supercuspidal representations). But Theorem \ref{thm: ExtVorH} does not seem to be included in their theory.
\end{remark}

\begin{remark} 
	The extension $\widetilde{\VorH}_{\pi}$ applies to test functions with \emph{non-simple singularities}. For example, for any non-trivial additive character $\psi: \F \to \Sph^1$ and for any $a \in \F^{\times}$ the function $H(x) := \chi(x) \psi(ax^2)$ lies in $\Cont(\F^{\times})$ for any $\chi \in \widehat{\F^{\times}}$. By Weil's classical work \cite{We64} we have
	$$ \widetilde{\VorH}_{\chi}(H)(y) = H^*(y) = \tfrac{\gamma_{\F}(a,\psi)}{\sqrt{2\norm[a]_{\F}}} \chi^{-1}(x) \norm[x] \psi(x^2/(4a)), $$
where $\gamma_{\F}(a,\psi) \in \Sph^1$ is the \emph{Weil index}. Obviously the function $H \notin \Inv(\SSch(\F)) + \chi \cdot \Sch(\F)$, and admits no properly defined Mellin transform. In a follow-up paper we will see that some natural test functions $H$ selecting \emph{short families} of $\PGL_2(\F)$-representations do have non-simple singularities at infinity, hence are beyond the applicability of Miller--Schmid's theory. 
\end{remark}

\begin{remark}
	In an earlier work \cite{BFW21+}  we have established a version of local weight transform expressed in terms of an integral transform with some hypergeometric kernel function at a real place. The space of admissible test functions is tricky in order to ensure the absolute convergence. The version in Theorem \ref{thm: DWtFTemp} is consistent with the version in \cite{BFW21+}, and has the advantage to be applicable to a larger class of test functions, which contains natural choices of test functions. 
\end{remark}

		\subsubsection{Kernel Function}
	
	Let $\F$ be a local field.
\begin{definition} \label{def: ConvTypeTrans}
	We call a transform $A$ on the space of functions on $\F^{\times}$ of \emph{convolution type} if its domain contains $\Cont_c^{\infty}(\F^{\times})$ and if there is a locally integrable function $a$ on $\F^{\times}$ so that
	$$ A(h)(y) = \int_{\F^{\times}} a(xy) h(x) \ud^{\times} x = a*(\Inv(h))(y), \quad \forall \ h \in \Cont_c^{\infty}(\F^{\times}). $$
	We call $a(y)$ the \emph{convolution kernel} of $A$.
\end{definition}

\begin{theorem}[Summary of Corollary \ref{cor: VorHKDi} and Lemma \ref{lem: VorHKInd}] \label{thm: VorHKerDih}
	Let $\pi$ be a \emph{split} or \emph{dihedral} representation of $\GL_2(\F)$, constructed from a quadratic extension $\E/\F$ and a \emph{regular} character $\eta$ of $\E^{\times}$. The Voronoi--Hankel transform $\VorH_{\pi}$ is of convolution type with kernel defined by
	$$ \VHF_{\pi}(t) := \zeta_{\E}(1)^{-1} \lambda(\E/\F,\psi) \id_{\Nr(\E^{\times})}(t) \cdot \norm[t]_{\F} \int_{\E^1} \psi(x \delta) \eta^{-1}(x \delta) \ud \delta, $$
where $\lambda(\E/\F,\psi)$ is the Weil index and $x \in \E$ is any element with $\Nr_{\E/\F}(x)=t$.
\end{theorem}

\begin{remark}
	The special cases (dihedral and unitary induced $\RamCst$-tempered cases) treated here already cover all cases that have been so far considered in literature. They already cover all possibilities over the archimedean fields. The extension of the integral representation to other unitary irreducible representations is interesting and would require deeper understanding of the local Langlands correspondences. We hope to come back to this problem, as well as the generalization to higher rank groups in the near future.
\end{remark}

\begin{remark}
	In the dihedral case, our integral representation should be regarded as a direct generalization of the one given by Baruch--Snitz \cite{BS11}.
\end{remark}

\begin{remark}
	We also note that a version of integral representation of $\VHF_{\pi}$ for a quite general class of $\pi$ is available by Jiang--Luo \cite[(3.15) \& (3.16)]{JiL22}. It would be nice to make that version useful for applications in our mind. We do not know how for the moment.
\end{remark}

\begin{remark}
	Note that for $\GL_2$ the Voronoi--Hankel kernel function $\VHF_{\pi}$ of $\VorH_{\pi}$ is intimately related to the Bessel function by the formula (we follow the convention made in \cite[Theorem 6.3 \& Corollary 7.3]{BM05})
	$$ \BesselF_{\pi,\psi} \begin{pmatrix} & -y \\ 1 & \end{pmatrix} = \omega_{\pi}(-y) \norm[y]^{-\frac{1}{2}} \VHF_{\pi}(y). $$
	So our integral representation of $\VHF_{\pi}$ automatically gives an integral representation of $\BesselF_{\pi,\psi}$, which will be the starting point of our choice of test functions for the relevant moment/subconvexity problem.
\end{remark}

\begin{remark}
	It can be checked that the integral representation of $\BesselF_{\pi,\psi}$ obtained here is consistent with the formulas of $\BesselF_{\pi,\psi}$ given by Baruch--Mao \cite[Theorem 6.4]{BM05} in the real case and summarized by Chai-Qi \cite[(4.3)]{CQ20} (due to Qi) in the complex case.
\end{remark}

\subsection{Notation and Convention}
	
	Below the \emph{general notation} applies to all parts of this paper. Every section/subsection takes either the \emph{global notation} or the \emph{local notation}, and will be specified at the beginning of each section/subsection.
	
		\subsubsection{General Notation}	
	
	For a locally compact group $G$, we write $\widehat{G}$ for the topological dual of continuous unitary irreducible representations. For $\pi \in \widehat{G}$, we write $V_{\pi}$ for the underlying Hilbert space, and write $V_{\pi}^{\infty} \subset V_{\pi}$ for the subspace of smooth vectors if $G$ carries extra structure to make sense of the notion.
	
	For a ring $R$, we write $\id_n$ for the identity matrix in $\GL_n(R)$, $w_n$ for the Weyl element with $1$'s on the anti-diagonal, and define $w_{n,t}$ for $t < n$ as
	$$ w_n = \begin{pmatrix} & & 1 \\ & \iddots & \\ 1 & & \end{pmatrix}, \quad w_{n,t} = \begin{pmatrix} \id_t & \\ & w_{n-t} \end{pmatrix}. $$
	We introduce the standard involution of \emph{inverse-transpose} on $\GL_n(R)$ as $g^{\iota} := {}^tg^{-1}$. We sometimes also write the transpose as $g^T$. For every function $F$ on $\GL_n(R)$ we define
	$$ \widetilde{F}(g) := F(w_n g^{\iota}). $$
	For a (unitary irreducible) representation $\Pi$ of $\GL_n(R)$, the central character is denoted by $\omega_{\Pi}$ and the contragredient is denoted by $\widetilde{\Pi}$.
	
	We introduce the following subgroups of $\GL_n(R)$:
	$$ \gp{A}_n(R) = \left\{ \begin{pmatrix} t_1 & & \\ & \ddots & \\ & & t_n \end{pmatrix} \ \middle| \ t_j \in R^{\times} \right\}, \quad \gp{N}_n(R) = \left\{ \begin{pmatrix} 1 & x_{1,2} & \cdots & x_{1,n} \\ 0 & 1 & \cdots & x_{2,n} \\ \vdots & \ddots & \ddots & \vdots \\ 0 & \cdots & 0 & 1 \end{pmatrix} \ \middle| \ x_{i,j} \in R \quad \forall i < j \right\}, $$
	$$ \gp{B}_n(R) = \gp{A}_n(R) \gp{N}_n(R), \quad \gp{Z}_n(R) = \left\{ z \id_n \ \middle| \ z \in R^{\times} \right\}. $$
	We usually write $z$ for $z \id_n$ if the context is clear. In the case $n=2$, we omit the subscript $2$ and write
	$$ a(y) = \begin{pmatrix} y & \\ & 1 \end{pmatrix}, \quad n(x) = \begin{pmatrix} 1 & x \\ & 1 \end{pmatrix}. $$
	In the case $n=3$ we have a Weyl element $w := w_3 w_{3,1}$.
	
		\subsubsection{Global Setting}
		\label{sec: GlobNotation}

	Let $\bF$ be a number field with ring of adeles $\A$, and group of ideles $\A^{\times}$. Write $\A^{(1)}$ for the subgroup of ideles with adelic norm $1$. We identify $\R_{>0}$ with the image of a fixed section map of the adelic norm map $\bF^{\times} \backslash \A^{\times} \to \R_{>0}$, so that $\bF^{\times} \backslash \A^{\times} \simeq \bF^{\times} \backslash \A^{(1)} \times \R_{>0}$ is identified as the direct product of a compact abelian group and $\R_{>0}$. Let $V_{\bF}$ be the set of all places of $\bF$. We fix the non-trivial additive character $\psi: \bF \backslash \A \to \C^1$ \`a la Tate, and choose the Haar measure $\ud x =\prod_v\ud x_v$ on $\A$ to be self-dual with respect to $\psi$. The Haar measure $\ud^{\times} x =\prod_v\ud^{\times}x_v$ on $\A^{\times}$ is taken to be the Tamagawa measure with factors of convergences $\zeta_v(1)$, namely
	\begin{align*}
	\ud^{\times}x_v = \zeta_v(1) \frac{\ud x_v}{\norm[x_v]_v}, \quad 
	\zeta_v(s) := \begin{cases} \pi^{-s/2} \Gamma(s/2),\ \ & \text{if } \bF_v = \R,\\
	(2\pi)^{1-s} \Gamma(s), & \text{if } \bF_v = \C,\\
	(1-\Nr(\vp)^{-s})^{-1}, & \text{if } v=\vp < \infty.\end{cases}
	\end{align*}
	
	We write $[\PGL_n] = \gp{Z}_n(\A) \GL_n(\F) \backslash \GL_n(\A)$. If $\omega$ is a unitary character of $\F^{\times} \backslash \A^{\times}$, called a \emph{Hecke character}, we denote by $\intL^2(\GL_n, \omega)$ the (Hilbert) space of Borel measurable functions $\varphi$ satisfying
\begin{align*}
\begin{cases}
\varphi(z \gamma g) = \varphi(g), \quad \text{for all } \gamma \in \GL_n(\F), ~z \in \gp{Z}_n(\A),~ g \in \GL_2(\A), \\ 
\text{The Petersson norm } \Pairing{\varphi}{\varphi} := \displaystyle\int_{[\PGL_n]} \norm[\varphi(g)]^2 \ud\bar{g} < \infty. 
\end{cases}
\end{align*}
	The subspace $\intL_0^2(\GL_n, \omega) \subset \intL^2(\GL_n, \omega)$ consists of those $\varphi$ satisfying
	$$ \int_{\gp{N}(\F) \backslash \gp{N}(\A)} \varphi(ng) \ud n = 0, \quad \text{a.e. } g $$
	for \emph{every} unipotent radical $\gp{N}$ of a \emph{parabolic} subgroup of $\GL_n$. It suffices to consider maximal parabolic subgroups. In particular, for $n=3$ it is equivalent to
\begin{equation} \label{eq: CuspCondGL3}
	\int_{(\F \backslash \A)^3} \varphi \left( \begin{pmatrix} 1 & & x_1 \\ & 1 & x_2 \\ & & 1 \end{pmatrix} g \right) \ud x_1 \ud x_2 = 0, \quad \text{a.e. } g. 
\end{equation}
	It can be shown that $\intL_0^2(\GL_n, \omega)$ is a closed subspace, and is completely decomposable as a direct sum of unitary irreducible components $(\Pi,V_{\Pi})$, called \emph{cuspidal (automorphic) representations} of $\GL_n(\A)$. 
	
	In the case $n=2$, the ortho-complement of $\intL_0^2(\GL_2, \omega)$ in $\intL^2(\GL_2,\omega)$ is the orthogonal sum of the one-dimensional spaces
	$$ \ag{C} \left( \xi \circ \det \right) : \quad \xi \text{ a Hecke character such that } \xi^2 = \omega $$
and a direct integral representation over the unitary dual of $\bF^{\times} \backslash \ag{A}^{\times} \simeq \ag{R}_+ \times (\bF^{\times} \backslash \ag{A}^{(1)} )$. Precisely, for $\tau \in \ag{R}$ and a unitary character $\chi$ of $\bF^{\times} \backslash \ag{A}^{(1)}$ which is regarded as a unitary character of $\bF^{\times} \backslash \ag{A}^{\times}$ via trivial extension, we associate a unitary representation $\pi_{i\tau}(\chi,\omega\chi^{-1})$ of $\GL_2(\ag{A})$ on the following Hilbert space $V_{i\tau}(\chi,\omega\chi^{-1})$ of functions via the right regular translation
\begin{equation} \label{eq: PSRepsGlob}
\begin{cases}
f\Big(\Big(\begin{matrix} t_1 & x \\ 0 & t_2 \end{matrix}\Big) g \Big) = \chi(t_1/t_2) \extnorm{\frac{t_1}{t_2}}_{\ag{A}}^{\frac{1}{2}+i\tau} f(g), \quad \text{for all }t_1,t_2 \in \ag{A}^{\times}, ~x \in \ag{A}, ~g \in \GL_2(\ag{A}) ; \\ 
\text{The induced norm } \Pairing{f}{f} := \displaystyle\int_{\gp{K}} \norm[f(\kappa)]^2 \ud\kappa < \infty . 
\end{cases}
\end{equation}	
	If $\tau=0$ we may omit it by writing $\pi(\chi,\omega\chi^{-1}) = \pi_0(\chi, \omega\chi^{-1})$. If $f_{i\tau} \in V_{i\tau}(\chi,\omega\chi^{-1})$ s.t. $f_{i\tau} \mid_{\gp{K}} =: h$ is independent of $\tau$, we call it a {\it flat section}. It extends to a holomorphic section $f_s \in \pi_s(\chi,\omega\chi^{-1})$ for $s \in \C$. Then $\pi_{i\tau}(\chi,\omega\chi^{-1})$ is realized as a subspace of functions on $\GL_2(\F) \backslash \GL_2(\A)$ via the Eisenstein series
	$$ \eis(s,h)(g) = \eis(f_s)(g) := \sum_{\gamma \in \gp{B}_2(\F) \backslash \GL_2(\F)} f_s(\gamma g), $$
which is absolutely convergent for $\Re s > 1/2$ and admits a meromorphic continuation regular at $s=i\tau$.
	
		\subsubsection{Local Setting}
		\label{sec: LocNotation}
		
	By local setting we fix a place $v \in V_{\F}$ and omit it from the relevant notation.
	
	For a local field $\F$ we let $\F^1$ to be the subgroup of elements in $\F^{\times}$ with norm $1$. If $\F$ is archimedean, then we have $\F^{\times} = \F^1 \times \R_{>0}$ and $\widehat{\F^{\times}} = \widehat{\F^1} \times i \R(\F)$ with $\R(\F) := \R$. If $\F$ is a non-archimedean, we write normalized valuation $v_{\F}$, ring of integers $\vO_{\F}$, a chosen uniformizer $\varpi_{\F}$ s.t. $v_{\F}(\varpi_{\F})=1$, valuation ideal $\vP_{\F} = \varpi_{\F} \vO_{\F}$ and $q=q_{\F}:= \norm[\vO_{\F}/\vP_{\F}]$. We identify $\F^1 = \vO_{\F}^{\times}$ with the quotient group $\F^{\times}/\varpi_{\F}^{\Z}$, as well as their characters. We also identify $\widehat{\F^{\times}}$ as $\widehat{\F^1} \times i\R(\F)$ where $\R(\F) := \R/ \left( 2\pi \log q \right) \Z$ is a torus. The group $\widetilde{\F^{\times}}$ of quasi-characters of $\F^{\times}$ is identified with 
	$$ \widehat{\vO_{\F}^{\times}} \times \C(\F) := \C/ \left( 2\pi i \log q \right) \Z = \R + i\R(\F) \to \widetilde{\F^{\times}}, \quad (\xi,s) \mapsto (t \mapsto \xi(t)\norm[t]^s). $$
For any $\sigma \in \R$ write $(\sigma)_{\F} \subset \C(\F)$ for the subset of elements with real part $\sigma$. The transported Plancherel measure on $(\sigma)_{\F}$ is denoted as $\ud_{\F} s$.
	
	For integers $n,m \geq 1$ we write $\Sch(n \times m, \F)$ for the space of Schwartz--Bruhat functions on $\Mat(n \times m, \F)$ the $n \times m$ matrices with entries in $\F$. It is naturally acted by $\GL_n(\F) \times \GL_m(\F)$ via the formula
	$$ g.\Psi.h(X) := \Psi(h X g), \quad \forall h \in \GL_n(\F), g \in \GL_m(\F), X \in \Mat(n \times m, \F). $$
	The (inverse) $\psi$-Fourier transform is denoted and defined by (see Remark \ref{rmk: FourTransConv} for more details)
	$$ \widehat{\Psi}(X) = \OFour_{\psi}(\Psi)(-X) = \int_{\Mat(n \times m, \F)} \Psi(Y) \psi \left( \Tr(XY^T) \right) \ud Y. $$
	If no confusion occurs, we may omit $\psi$ from the notation. If this is the case, then the inverse Fourier transform is denoted by $\invOFour = \OFour_{\overline{\psi}} = \OFour_{\psi^{-1}}$. We may also add the subscript $\F$ to emphasize the base field. Note that the Fourier transform has the property
	$$ \widehat{g.\Psi.h} = \norm[\det g]^{-n} \cdot \norm[\det h]^{-m} \cdot g^{\iota}.\widehat{\Psi}.h^{\iota}. $$
	For indices $1 \leq i \leq n$ and $1 \leq j \leq m$, let $\OFour_{i,j}$ be the partial Fourier transform with respect to the variable at the $i$th row and $j$th column in $\Mat(n \times m, \F)$. Let $\OFour_{\vec{j}}$ be the composition of all $\OFour_{i,j}$ with $1 \leq i \leq n$.
	
	We introduce the (connected) maximal compact subgroup $\gp{K}_n$ of $\GL_n(\F)$ as
	$$ \gp{K}_n = 
	\begin{cases}  
		\SO_n(\R) & \text{if } \F = \R \\
		\SU_n(\C) & \text{if } \F = \C \\
		\GL_n(\vo) & \text{if } \F \text{ is non-archimedean with valuation ring } \vo
	\end{cases}, $$
	and equip it with the probability Haar measure $\ud \kappa$.
	
	The principal series representations given by (\ref{eq: PSRepsGlob}) have obvious local versions at every place $v \in V_{\F}$.

	\subsection*{Acknowledgement}
	
	The author thanks Jingsong Chai, Dihua Jiang, Emmanuel Kowalski, Yongxiao Lin, Hengfei Lv, Xinchen Miao, Zhi Qi and Ping Xi for discussions related to the topics of this paper.

\section{Technicality with Schwartz-Bruhat Functions}
\label{sec: SchTechBd}
	
	We take the local setting (see \S \ref{sec: LocNotation}) in this section. We recollect some estimation related to Schwartz-Bruhat functions. Only non-obvious ones will require full proofs. We begin with some general results.
	
\begin{proposition}[Bounds of Restrictions] \label{prop: SchRestriction}
	Let $\Phi \in \Sch(1 \times n, \F)$ and integer $0 \leq m \leq n$. Let $y_j, 1 \leq j \leq n-m$ be functions of the variables $x_k, 1 \leq k \leq m$. For any constants $\sigma_j \geq 0$ with $1 \leq j \leq n-m$, we can find a non-negative $\phi \in \Sch(1 \times m, \F)$ such that
	$$ \Phi(x_1, \cdots, x_m, y_1, \cdots, y_{n-m}) \leq \phi(x_1,\cdots, x_m) \prod_{j=1}^{n-m} \norm[y_j]^{-\sigma_j}. $$
\end{proposition}

\begin{proposition}[Bounds of Compact Translations] \label{prop: SchCompTrans}
	Let $\Phi \in \Sch(1 \times n, \F)$ with variables $\vec{x} = (x_1, \cdots, x_n)$. Suppose $C \subset \GL_n(\F)$ is a compact subset. For $I \subset \{ 1, \cdots, n \}$, let $\OFour_I$ be the composition of partial Fourier transforms with respect to the variables $x_i, i \in I$. Then we can find a positive $\widetilde{\Phi} \in \Sch(1 \times n, \F)$ s.t.
	$$ \extnorm{(\OFour_I(g.\Phi))(\vec{x})} \leq \widetilde{\Phi}(\vec{x}), \quad \forall g \in C, X \in \F^n. $$
\end{proposition}
\begin{corollary} \label{cor: SchCompTrans}
	Let $\Phi \in \Sch(n \times m, \F)$. We can find positive $\widetilde{\Phi}, \Phi_1, \Phi_2 \in \Sch(n \times m, \F)$ such that
	$$ \max_{\kappa_1 \in \gp{K}_n, \kappa_2 \in \gp{K}_m} \norm[(\kappa_2.\Phi.\kappa_1)(X)] \leq \widetilde{\Phi}(X), \quad \max_{\kappa \in \gp{K}_n} \norm[(\Phi.\kappa)(X)] \leq \Phi_1(X), \quad \max_{\kappa \in \gp{K}_m} \norm[(\kappa.\Phi)(X)] \leq \Phi_2(X). $$
\end{corollary}
\begin{proof}
	Since the action of $\GL_n(\F) \times \GL_m(\F)$ is naturally embedded in the action of $\GL_{nm}(\F)$ on $\Sch(n \times m, \F) \simeq \Sch(1 \times nm, \F)$, the corollary follows readily from Proposition \ref{prop: SchCompTrans}.
\end{proof}

	Proposition \ref{prop: SchRestriction} \& \ref{prop: SchCompTrans} were frequently used in the literature on the Godement--Jacquet theory and the Rankin--Selberg theory by Jacquet and his collaborators. We only sketch the proofs. These results are easy when $\F$ is non-archimedean, since any function in $\Sch(\F^n) = \Cont_c^{\infty}(\F)$ is a finite linear combination of characteristic functions $\id_C$, where $C \in \F^n$ is a compact subset of product type. Then note that these results trivially hold for such $\id_C$, and that $\Cont_c^{\infty}(\F^n)$ is stable by taking maximum (or sum) of two functions, i.e., $\Phi_j \in \Cont_c^{\infty}(\F) \Rightarrow \max(\Phi_1, \Phi_2) \in \Cont_c^{\infty}(\F)$. For the archimedean case, we only need to consider the real case $\F=\R$, since $\Sch(\C^n) \simeq \Sch(\R^{2n})$ and $\GL_n(\C) < \GL_{2n}(\R)$. The real case can be reduced to the following classical lemma, whose proof can be found in Garrett's note \cite{Ga04} on his webpage.
\begin{lemma}[Weil-Schwartz envelopes] \label{lem: WeilSchwartzEnv}
	Let $f: \R^n \to \R$ be \emph{rapidly decreasing} in the sense that for every $m \in \Z_{\geq 0}$ we have
	$$ \sup_{\vec{x} \in \R^n} \Norm[\vec{x}]_2^m \norm[f(\vec{x})] < \infty. $$
	Then we can find a positive $\phi \in \Sch(\R^n)$, which is spherical (i.e., depends only on the Euclidean norm $\Norm_2$) and monotone decreasing in $\Norm_2$, such that
	$$ \phi(\vec{x}) \geq \extnorm{f(\vec{x})}, \quad \forall \vec{x} \in \R^n. $$
\end{lemma}

\noindent For example, to prove Proposition \ref{prop: SchCompTrans}, we shall apply Lemma \ref{lem: WeilSchwartzEnv} to
	$$ f(\vec{x}) := \sup_{g \in C} \extnorm{(\OFour_I(g.\Phi))(\vec{x})}. $$
	To verify the rapid decay of $f$, it suffices to show that for every (unitary) monomial $m(\vec{x})$, the function $m(\vec{x}) \OFour_I(g.\Phi))(\vec{x})$ is uniformly bounded for $g \in C$ and $\vec{x} \in \R^n$. Writing $m(\vec{x}) = m_{I^c}(\vec{x}) m_I(\vec{x})$, where $I^c$ is the complementary of $I$ and $m_J(\vec{x})$ is a monomial whose variables have indices only in $J$, and taking into account that the Fourier transform essentially exchanges multiplication by $x_j$ and the partial differentiation $\partial_j$ with respect to $x_j$, it suffices to show a uniform bound of the $\intL^1$-norms of
	$$ m_{I^c}(\vec{\partial}) \OFour_{I^c} \OFour_I (m_I(\vec{\partial}) (g.\Phi)) =  m_{I^c}(\vec{\partial}) \OFour (m_I(\vec{\partial}) (g.\Phi)). $$
	By induction on the degree $d$ of a monomial $P$, it is easy to show that
	$$ P(\vec{\partial})(g.\Phi) = \sum_{Q \text{ monomial of degree } d} R_{P,Q}(g) (g.(Q(\vec{\partial}) \Phi)) $$
	where $R_{P,Q}$ is a monomial on the matrix entries of $g$ of degree $d$. Hence $R_{m_I,Q}(g)$ is uniformly bounded as $g \in C$. Consequently we are reduced to showing a uniform bound of the $\intL^1$-norms of
	$$ m_{I^c}(\vec{\partial}) \OFour(g.(Q_1(\vec{\partial})\Phi)) = \norm[\det g]^{-1} m_{I^c}(\vec{\partial}) \left( g^{\iota}.\OFour(Q_1(\vec{\partial})\Phi)) \right) $$
	for all (unitary) monomial $Q_1$ with the same degree as $m_I$. Since $\norm[\det g]^{-1}$ and the matrix entries of $g^{\iota}$ are also uniformly bounded, the same reasoning reduces to showing a uniform bound of the $\intL^1$-norms of
	$$ g^{\iota}.Q_2(\vec{\partial}) \left( \OFour(Q_1(\vec{\partial})\Phi)) \right) $$
	for all (unitary) monomial $Q_1$ and $Q_2$ with the same degrees as $m_I$ and $m_{I^c}$ respectively. We can ignore $g^{\iota}$ since the resulting $\intL^1$-norms will be affected by a positive factor which is uniformly bounded from above and below. Then the desired uniform bound exists by definition of a Schwartz function.

\section{Generalized Godement Sections}
\label{sec: GGS}

	\subsection{Whittaker-Valued Schwartz Functions}

	For simplicity of notation we write $\gp{G}_n := \GL_n(\F)$ for any integer $n \in \Z_{\geq 1}$. All representations are assumed to be smooth, and of \emph{moderate growth} if $\F$ is archimedean (see \cite[\S 3.2]{J09}).
\begin{definition} \label{def: W-ValSchF}
	Let $n_i \in \Z_{\geq 1}$ and $\pi_i$ be a generic irreducible representation of $\gp{G}_{n_i}$. Consider the tensor product representation $\pi = \otimes_{i=1}^r \pi_i$ of the direct product group $\gp{M} := \sideset{}{_{i=1}^r} \prod \gp{G}_{n_i}$.
\begin{itemize} 
	\item[(1)] The $\psi$-Whittaker model, resp. smooth $\psi$-Whittaker model, of $\pi$ is defined to be 
	$$ \Whi(\pi,\psi) := \sideset{}{_{i=1}^r} \bigotimes \Whi(\pi, \psi), \quad \Whi(\pi^{\infty},\psi) := \sideset{}{_{i=1}^r} \bigotimes \Whi(\pi^{\infty}, \psi), $$
where the tensor product is taken as the completion of the algebraic one.
	\item[(2)] Let $\mathcal{A}$ be a finite dimensional $\F$-vector space. We write 
	$$ \Cont_c^{\infty}(\mathcal{A}, \Whi(\pi^{\infty},\psi)) = \cInd \left( \mathcal{A} \times \gp{M}, \gp{M}; \Whi(\pi^{\infty},\psi) \right). $$
	Let $\Sch(\mathcal{A}, \Whi(\pi^{\infty},\psi)) = \Sch(\mathcal{A}) \widehat{\otimes}_{\pi} \Whi(\pi^{\infty},\psi)$ be the \emph{projective tensor product} of Fr\'echet spaces, viewed as a subspace of functions/vectors $f$ in the smoothly induced representation
	$$ \Ind \left( \mathcal{A} \times \gp{M}, \gp{M}; \Whi(\pi^{\infty},\psi) \right) \subset \Cont^{\infty}(\mathcal{A} \times \gp{G}, \C). $$
 	\item[(3)] For simplicity of notation we write
	$$ \Cont_c^{\infty}(m \times l, \Whi(\pi^{\infty},\psi)) := \Cont_c^{\infty}(\Mat(m \times l, \F), \Whi(\pi^{\infty},\psi)), $$
	$$ \Sch(m \times l, \Whi(\pi^{\infty},\psi)) := \Sch(\Mat(m \times l, \F), \Whi(\pi^{\infty},\psi)). $$
	For $\Phi \in \Sch(m \times l, \Whi(\pi^{\infty},\psi))$ and $g \in \gp{G}_m$, $h \in \gp{G}_l$ we write
	$$ (g.\Phi.h)(X; \tau) := \Phi(hXg; \tau), \quad \forall \ X \in \Mat(m \times l, \F), \tau \in \gp{M}. $$
\end{itemize}
	Elements in $\Sch(\mathcal{A}, \Whi(\pi^{\infty},\psi))$ are called $\Whi(\pi^{\infty},\psi)$-valued Schwartz functions. Note that $\Sch(\mathcal{A}, \Whi(\pi^{\infty},\psi)) = \Cont_c^{\infty}(\mathcal{A}, \Whi(\pi^{\infty},\psi))$ for non-archimedean $\F$.
\end{definition}
\begin{definition} \label{def: W-ValSchParF}
	For $\Phi \in \Sch(m \times l, \Whi(\pi^{\infty},\psi))$ its partial $\psi$-Fourier transform in the $j$-th column of $\Mat(m \times l, \F)$ is denoted and defined by 
	$$ \OFour_{\vec{j}}(\Phi)(\vec{v}_1, \dots, \vec{v}_l; g) := \int_{\F^m} \Phi(\vec{v}_1, \dots, \vec{v}_{j-1}, \vec{u}, \vec{v}_{j+1}, \dots, \vec{v}_l; g) \psi(- \vec{v}_j^T \vec{u}) \ud \vec{u}. $$
\end{definition}
\begin{remark}
	We have $\OFour_{\vec{j}}(\Phi) \in \Sch(\Mat(m \times l, \F), \Whi(\pi^{\infty},\psi))$. More generally for any $\Phi \in \Sch(\mathcal{A}, \Whi(\pi^{\infty},\psi))$ the partial Fourier transform with respect to any coordinate of $\mathcal{A}$ lies in $\Sch(\mathcal{A}, \Whi(\pi^{\infty},\psi))$.
\end{remark}
\begin{definition} \label{def: Gauge}
	(1) Let $n \in \Z_{\geq 1}$. A \emph{gauge} $\xi$ on $\gp{G}_n$ is a (positive) function satisfying
	$$ \xi(nak) = \xi(a) = \norm[\det a]^M \sideset{}{_{i=1}^{n-1}} \prod \extnorm{\alpha_i(a)}^{-N} \cdot \phi(\vec{\alpha}(a)), \quad \vec{\alpha}(a):=(\alpha_1(a), \dots, \alpha_{n-1}(a)) $$
	for some constants $M \in \R,N > 0$ and Schwartz-Bruhat function $\phi \in \Sch(\F^{n-1})$, where $nak$ is the Iwasawa decomposition of an element in $\gp{G}_n$ and for $a = \mathrm{diag}(d_1,\cdots,d_n)$ we have $\alpha_i(a)=d_i/d_{i+1}$.

\noindent (2) Let $\mathcal{A}$ be a finite dimensional $\F$-vector space. Let $n_i \in \Z_{\geq 1}$. A \emph{gauge} $\xi$ on $\mathcal{A} \times \sideset{}{_{i=1}^r} \prod \gp{G}_{n_i}$ is a (positive) function satisfying
	$$ \xi(a; n_1a_1k_1, \dots, n_ra_rk_r) = \xi(a; a_1, \dots, a_r) = \sideset{}{_{i=1}^r} \prod \norm[\det a_i]^{M_i} \sideset{}{_{j=1}^{n_i-1}} \prod \extnorm{\alpha_j(a_i)}^{-N_i} \cdot \phi(a; \vec{\alpha}(a_1), \dots, \vec{\alpha}(a_r)) $$
	for some constants $M_i \in \R,N_i > 0$ and Schwartz-Bruhat function $\phi \in \Sch(\mathcal{A} \times \F^{n_1-1} \times \cdots \times \F^{n_r-1})$.
\end{definition}

\begin{proposition} \label{prop: W-ValSchBd}
	Any $\Phi \in \Sch(\mathcal{A}, \Whi(\pi^{\infty},\psi))$, defined as in Definition \ref{def: W-ValSchF}, is bounded by some gauge.
\end{proposition}
\begin{proof}
	If $\F$ is non-archimedean, any element in $\Whi(\pi^{\infty},\psi)$ is bounded by some gauge by \cite[Proposition (2.3.6)]{JPS79}. This suffices to conclude because any element in $\Sch(\mathcal{A}, \Whi(\pi^{\infty},\psi))$ is a finite sum of elements of the form $\Phi(a;m) = \phi(a) W(m)$ for some $\phi \in \Sch(\mathcal{A})$ and $W \in \Whi(\pi^{\infty},\psi)$. If $\F$ is archimedean we apply \cite[Proposition 3.1]{J09} to get the bound for some $M_i \in \R, N_i > 0$ and any $N \in \Z_{\geq 1}$
\begin{multline*} 
	\extnorm{\Phi(a; n_1a_1k_1, \dots, n_r a_r k_r)} \leq \sideset{}{_{i=1}^r} \prod \norm[\det a_i]^{M_i} \sideset{}{_{j=1}^{n_i-1}} \prod \extnorm{\alpha_j(a_i)}^{-N_i} \cdot \\
	\sideset{}{_{i=1}^r} \prod \sideset{}{_{j=1}^{n_i-1}} \prod \left( 1 + \extnorm{\alpha_j(a_i)} \right)^{-N} \cdot \nu_N(\Phi(a; \cdot)),
\end{multline*}
where $\nu_N$ is a semi-norm on $\pi^{\infty}$ independent of $\Phi$. Since $a \mapsto \nu_N(\Phi(a; \cdot))$ is in $\Sch(\mathcal{A})$, the function
	$$ \sup_{n_i,k_i, d_{n_i}(a_i)} \extnorm{\Phi(a; n_1a_1k_1, \dots, n_r a_r k_r)} \cdot \left( \sideset{}{_{i=1}^r} \prod \norm[\det a_i]^{M_i} \sideset{}{_{j=1}^{n_i-1}} \prod \extnorm{\alpha_j(a_i)}^{-N_i} \right)^{-1} $$
	is a rapidly decreasing function on $\mathcal{A} \times \F^{n_1-1} \times \cdots \times \F^{n_r-1}$, hence is bounded by a Schwartz function by Lemma \ref{lem: WeilSchwartzEnv}.
\end{proof}

	\subsection{Whittaker Functions for Induced Representations}
	
	Take the case $r=2$ in Definition \ref{def: W-ValSchF} with $n=n_1+n_2$. Let $\gp{Q}_{n_1,n_2} < \gp{G}_n$ be the (block \emph{lower} triangular) parabolic subgroup whose Levi component is isomorphic to $\gp{M} = \gp{G}_{n_1} \times \gp{G}_{n_2}$. Then we have a surjective group homomorphism $\gp{Q}_{n_1,n_2} \to \gp{M}$, so that $\pi$ can be inflated to a representation of $\gp{Q}_{n_1,n_2}$ still denoted by $\pi$. Let $\Pi = \pi_1 \boxplus \pi_2$ be the representation of $\gp{G}_n$ induced from $\pi$ of $\gp{Q}_{n_1,n_2}$. We may realize the underlying vector space of $\Pi$ as
	$$ V_{\Pi} = \Ind(\gp{G}_{n}, \gp{Q}_{n_1,n_2}; \Whi(\pi_1,\psi) \otimes \Whi(\pi_2,\psi)). $$
	The subspace of smooth vectors $V_{\Pi}^{\infty}$ consists of $f \in \Cont^{\infty}(\gp{G}_{n} \times \gp{G}_{n_1} \times \gp{G}_{n_2}, \C)$ satisfying (see \cite[\S (4.6)]{JPS83}):
\begin{itemize}
	\item[(1)] For any $g \in \gp{G}_n$, $a_j$ \& $h_j \in \gp{G}_{n_j}$ and $X \in \Mat(n_2 \times n_1, \F)$ we have 
	$$ f \left[ \begin{pmatrix} a_1 & 0 \\ X & a_2 \end{pmatrix}g; h_1, h_2 \right] = \frac{\norm[\det a_2]^{\frac{n_1}{2}}}{\norm[\det a_1]^{\frac{n_2}{2}}} \cdot f\left[ g; h_1a_1, h_2a_2 \right]; $$
	\item[(2)] For any fixed $g \in \gp{G}_n$ and $h_2 \in \gp{G}_{n_2}$, the following function lies in $\Whi(\pi_1^{\infty}, \psi)$
	$$ h_1 \mapsto f\left[ g; h_1, h_2 \right]; $$
	\item[(3)] For any fixed $g \in \gp{G}_n$ and $h_1 \in \gp{G}_{n_1}$, the following function lies in $\Whi(\pi_2^{\infty}, \psi)$
	$$ h_2 \mapsto f\left[ g; h_1, h_2 \right]. $$
\end{itemize}
	We are interested in the special case $n_2=1$. Then $\pi_2 = \mu$ is a (quasi-)character of $\F^{\times}$. In this case the Grassmannian $\gp{Q}_{n_1,1} \backslash \gp{G}_n$ is identified with the projective space $\mathbb{P}^{n_1}(\F)$ via
	$$ \gp{G}_n \to \mathbb{P}^{n_1}(\F), \quad g \mapsto [g^{-1}.\vec{e}_n]. $$ 
Note that $\mathbb{P}^{n_1}(\F)$ admits an open covering of $n$ copies of $\F^{n_1}$. Denote by $\gp{N}_{n_1,1}$ the transpose/opposite of the unipotent radical of $\gp{Q}_{n_1,1}$. We have the corresponding open affine covering of $\gp{G}_n$
\begin{equation} \label{eq: AffCovGn} 
	\gp{G}_n = \sideset{}{_{i=1}^n} \bigcup \gp{Q}_{n_1,1} \gp{N}_{n_1,1} \alpha_n^i, \quad \alpha_n := \begin{pmatrix} & \id_{n_1} \\ 1 & \end{pmatrix}. 
\end{equation}
\begin{lemma} \label{lem: EleSmoothInd}
	The space $V_{\Pi}^{\infty}$ is generated by smooth vectors $f$ satisfying the following conditions:
\begin{itemize}
	\item[(1)] There is $1 \leq i \leq n$ such that $f$ has support contained in $\gp{Q}_{n_1,1} \gp{N}_{n_1,1} \alpha_n^i$.
	\item[(2)] There is $\phi \in \Cont_c^{\infty}(\F^{n_1}, \Whi(\pi_1^{\infty},\psi))$ such that
	$$ f\left[ \begin{pmatrix} \id_{n_1} & \vec{u} \\ 0 & 1 \end{pmatrix} \alpha_n^i; h_1, h_2 \right] = \phi(\vec{u}, h_1) \cdot \mu(h_2), \quad \forall \vec{u} \in \F^{n_1}, h_i \in \gp{G}_{n_i}. $$
\end{itemize}
\end{lemma}
\begin{proof}
	A smooth partition of unity subordinate to the standard affine covering of $\mathbb{P}^{n_1}(\F)$ can be lifted $\gp{G}_n$, so that for any $f \in V_{\Pi}^{\infty}$ we can write $f = \sideset{}{_{i=1}^r} \sum f_i$, where each $f_i \in V_{\Pi}^{\infty}$ has support contained in $\gp{Q}_{n_1,1} \gp{N}_{n_1,1} \alpha_n^i$ and compact modulo $\gp{Q}_{n_1,1}$. Equivalently there is $\phi_i \in \Cont_c^{\infty}(\F^{n_1}, \Whi(\pi_1^{\infty},\psi))$ such that
	$$ f_i\left[ \begin{pmatrix} \id_{n_1} & \vec{u} \\ 0 & 1 \end{pmatrix} \alpha_n^i; h_1, h_2 \right] = \phi_i(\vec{u}, h_1) \cdot \mu(h_2), \quad \forall \vec{u} \in \F^{n_1}, h_i \in \gp{G}_{n_i}. $$
	The assertion follows readily.
\end{proof}

\noindent Let $\Phi \in \Sch(n_1 \times n, \Whi(\pi^{\infty},\psi))$. Consider the following function on $\C \times \gp{G}_n \times \gp{G}_{n_1} \times \gp{G}_1$
\begin{equation} \label{eq: GenGSInd}
	f_{\Phi}(s; g; h, t) := \mu(\det g) \norm[\det g]^{n_1 \left( \frac{1}{2}+s \right)} \int_{\gp{G}_{n_1}} \left( g.\Phi \right)(h_1, \vec{0}; hh_1^{-1}, t\det(h_1h^{-1})) \norm[\det h_1]^{n \left( \frac{1}{2}+s \right)} \ud h_1.
\end{equation}
\begin{lemma} \label{lem: GenGSIndConv}
	(1) The right hand side of the equation \eqref{eq: GenGSInd} is absolutely convergent for $\Re(s) \gg 1$, where the implicit constant depends only on $\pi_1$ and $\mu$ (independent of $\Phi$).
	
\noindent (2) Write $\pi(s) := \pi \otimes \norm[\det(\cdot)]^s$ and let $\Pi_s := \pi_1(-s) \boxplus \mu(n_1s)$. In the absolute convergent region we have $f_{\Phi}(s; \cdot) \in V_{\Pi_s}^{\infty}$. Moreover any element in $V_{\Pi_s}^{\infty}$ is a such $f_{\Phi}(s; \cdot)$.
\end{lemma}
\begin{proof}
	(1) By the change of variables $h_1 \mapsto h_1h$ it suffices to prove the absolute convergence of
	$$ \int_{\gp{G}_{n_1}} \Phi(h_1, \vec{0}; h_1^{-1}, \det h_1) \norm[\det h_1]^{n \left( \frac{1}{2}+s \right)} \ud h_1 $$
for $\Re(s) \gg 1$ and any $\Phi \in \Sch(n_1 \times n, \Whi(\pi_1^{\infty},\psi))$. Bounding $\Phi$ by a gauge by Proposition \ref{prop: W-ValSchBd} and applying the change of variables $h_1 = kan$ we see that the above integral is bounded by
	$$ \int_{(\F^{\times})^{n_1}} \left( \int_{\F^{n_1(n_1-1)/2}} \phi(t_1,\dots,t_{n_1}, \vec{x}) \ud \vec{x} \right) \sideset{}{_{i=1}^{n_1}} \prod \norm[t_i]^{\sigma_i} \ud^{\times} t_i $$
for some $\phi \in \Sch(\F^{n_1(n_1+1)/2})$ and $\sigma_i \gg 1$, which is absolutely convergent.

\noindent (2) Assuming the absolute convergence we easily check the relation
	$$ f_{\Phi} \left( s; \begin{pmatrix} a_1 & 0 \\ X & a_2 \end{pmatrix} g; h, t \right) = \left( \frac{\norm[a_2]^{n_1}}{\norm[\det a_1]} \right)^{\frac{1}{2}-s} f_{\Phi}(s; g; ha_1, ta_2), $$
	showing that $f_{\Phi}(s; \cdot) \in V_{\Pi_s}^{\infty}$. We turn to the ``moreover'' part. Note that the case $n_1=1$ is classical. A proof can be found in \cite[Lemma 3.5 \& 3.8 \& 3.14]{Wu17}. Assume $n_1 \geq 2$ from now on. Taking Lemma \ref{lem: EleSmoothInd} into account, it suffices to show for any $\phi \in \Cont_c^{\infty}(\F^{n_1}, \Whi(\pi^{\infty},\psi))$ there exists $\Phi \in \Sch(n_1 \times n, \Whi(\pi^{\infty},\psi))$ such that
	$$ f_{\Phi} \left( s; \begin{pmatrix} \id_{n_1} & \vec{u} \\ 0 & 1 \end{pmatrix} \alpha_n^i; h, t \right) = \begin{cases} 
		\phi(\vec{u}, h) \cdot \mu(t) & \text{if } i = 0 \\
		0 & \text{if } 0 < i < n
	\end{cases}. $$
	Note that $\Cont_c^{\infty}(\F^{n_1}, \Whi(\pi^{\infty},\psi))$ is a smooth Fr\'echet representation of $\gp{G}_{n_1}$ for the (right-)action
	$$ \phi^{h_1}(\vec{u}, h) := \phi(h_1 \vec{u}, hh_1^{-1}). $$
	In fact it a smooth Fr\'echet representation of $\gp{G}_{n_1} \times \gp{G}_{n_1}$ inherited from the natural actions on $\Cont_c^{\infty}(\F^{n_1})$ and $\Whi(\pi^{\infty},\psi)$. The above action is just the restriction of $\gp{G}_{n_1} \times \gp{G}_{n_1}$ to 
	$$ \gp{G}_{n_1} \hookrightarrow \gp{G}_{n_1} \times \gp{G}_{n_1}, \quad h \mapsto (h, h^{-1}). $$
	By Diximier-Malliavin's theorem (see \cite[Proposition 6.1]{CaB16_D} in the archimedean case and trivial in the non-archimedean case) there exist functions $\ell_j \in \Cont_c^{\infty}(\gp{G}_{n_1})$ and $\phi_j \in \Cont_c^{\infty}(\F^{n_1}, \Whi(\pi^{\infty},\psi))$ such that
	$$ \phi(\vec{u}, h) = \sideset{}{_{j=1}^d} \sum \int_{\gp{G}_{n_1}} \ell_j(h_1) \cdot \phi_j(h_1\vec{u}, hh_1^{-1}) \ud h_1. $$
	Moreover, the support of $\ell_j$ can be taken as close as possible to $\id_{n_1}$. Take a small neighborhood $\Omega$ of $\id_{n_1}$ so that for the natural projections $P_k$ from $\Mat(n_1 \times n_1, \F)$ to the $k$-th column vector space the sets $P_k(\Omega)$ are disjoint for $1 \leq k \leq n_1$. We easily check that
	$$ \Phi_j(h_1,\vec{u}; h, 1) = \begin{cases}
		\ell_j(h_1) \phi_j(\vec{u},h) \mu(\det h) \norm[\det h_1]^{-n \left( \frac{1}{2}+s \right)} & \text{if } h_1 \in \gp{G}_{n_1} \\
		0 & \text{if } \det h_1 = 0
	\end{cases} $$
defines a $\Phi_j \in \Sch(n_1 \times n, \Whi(\pi^{\infty},\psi))$, and $\Phi := \sideset{}{_{j=1}^r} \sum \Phi_j$ satisfies the desired properties.
\end{proof}
	
\noindent By \cite[\S (4.6)]{JPS83}, the $\psi$-Whittaker function of $f_{\Phi}(s; \cdot)$ defined in \eqref{eq: GenGSInd} is given by
\begin{equation} \label{eq: GenGSWhiDef}
	W_{\Phi}(s,g) := \int_{\F^{n_1}} f_{\Phi} \left( s; \begin{pmatrix} \id_{n_1} & \vec{u} \\ 0 & 1 \end{pmatrix} g; \id_{n_1}, 1 \right) \psi(-u_{n_1}) \ud \vec{u}, \quad \vec{u} = (u_1, \dots, u_{n_1})^T.
\end{equation}
	Inserting \eqref{eq: GenGSInd} into \eqref{eq: GenGSWhiDef} and changing the order of integrations we get
\begin{equation} \label{eq: GenGSWhi}
	W_{\Phi}(s,g) = \mu(\det g) \norm[\det g]^{n_1 \left( \frac{1}{2}+s \right)} \int_{\gp{G}_{n_1}} \OFour_{\vec{n}}(g.\Phi.h_1)(\id_{n_1}, \vec{e}_{n_1}; h_1^{-1}, \det h_1) \norm[\det h_1]^{n \left( \frac{1}{2}+s \right)} \ud h_1.
\end{equation}
	
\begin{proposition} \label{prop: GenGSWhi}
	(1) The integral \eqref{eq: GenGSWhi} is absolutely convergent for all $s \in \C$. In particular, $W_{\Phi}(\cdot) := W_{\Phi}(0, \cdot) \in \Whi(\Pi^{\infty},\psi)$. 
	
\noindent (2) Every element of $W \in \Whi(\Pi^{\infty},\psi)$ is equal to $W_{\Phi}(\cdot)$ for some $\Phi \in \Sch(n_1 \times n, \Whi(\pi^{\infty},\psi))$.
\end{proposition}
\begin{proof}
	For (1) it suffices to prove the absolute convergence of
	$$ \int_{\gp{G}_{n_1}} \Phi(h_1, h_1^{\iota}.\vec{e}_{n_1}; h_1^{-1}, \det h_1) \norm[\det h_1]^{A} \ud h_1 $$
for any real number $A$ and $\Phi \in \Sch(n_1 \times n, \Whi(\pi^{\infty},\psi))$. Applying the change of variables $h_1=kan$ and bounding $\Phi$ by a gauge we see that the above integral is bounded by
	$$ \int_{\F^{n_1}} \phi \left( a_1, \cdots, a_{n_1}, a_{n_1}^{-1}; \vec{\alpha}^{-1} \right) \frac{\norm[a_1]^N}{\norm[a_n]^N} \cdot \sideset{}{_{j=1}^{n_1}} \prod \norm[a_j]^{B-j} \ud a_j $$
for some $\phi \in \Sch(\F^{2n_1})$, $N \in \Z_{\geq 1}$, $B \in \R$ and $\vec{\alpha}^{-1} = (\alpha_1^{-1}, \dots, \alpha_{n_1-1}^{-1})$ with $\alpha_j =a_ja_{j+1}^{-1}$. If $B-j < 0$ for some $j$, we rewrite $\norm[a_j]^{B-j} = \extnorm{\alpha_j \alpha_{j+1} \cdots \alpha_{n_1-1}}^{B-j} \cdot \norm[a_n]^{B-j}$. So the integrand is
	$$ \phi \left( a_1, \cdots, a_{n_1}, a_{n_1}^{-1}; \vec{\alpha}^{-1} \right) \cdot \sideset{}{_{j=1}^{n_1}} \prod \norm[a_j]^{\sigma_j} \sideset{}{_{j=1}^{n_1-1}} \prod \norm[\alpha_j]^{-\tau_j} $$
	for some real numbers $\sigma_j ,\tau_j \geq 0$. It is therefore convergent. (2) is a direct consequence of the ``moreover'' part of Lemma \ref{lem: GenGSIndConv} (2).
\end{proof}

\begin{remark}
	The integral representation \eqref{eq: GenGSWhiDef} of Whittaker functions is a generalization of the one given in \cite[(7.6) \& (7.7)]{J09}, where $\pi$ is assumed to be induced from the Borel subgroup.
\end{remark}

\section{Double Zeta Integrals}

	\subsection{Complements on Rankin-Selberg Integrals: Global Theory}

	We take the global setting in this subsection (see \S \ref{sec: GlobNotation}). 
	
	Let $\Pi$ (resp. $\pi$) be a cuspidal (resp. automorphic) representation of $\GL_3(\A)$ (resp. $\GL_2(\A)$) with central character $\omega_{\Pi}$ (resp. $\omega$). Without loss of generality, we assume $\omega_{\Pi}$ and $\omega$ are trivial on $\R_+$. We take a smooth vector $F \in V_{\Pi}^{\infty}$ (resp. $\varphi \in V_{\pi}^{\infty}$). If $\pi = \pi_{s_1}(\chi, \omega\chi^{-1}\norm_{\A})$ lies in the continuous spectrum, we assume $\chi$ to be trivial on $\R_+$ and further require $\varphi = \eis(s_1,f)$ to be an Eisenstein series associated with a flat section $f(s_1,\cdot)$ defined by $f \in \pi(\chi, \omega\chi^{-1})$. The global Rankin-Selberg integral for $\Pi \times \pi$, when $\pi$ is cuspidal, was introduced by Jacquet--Shalika in \cite[\S 3.3]{JS81_Eur2} as
\begin{equation} 
	\Psi(s, F, \varphi) = \int_{[\GL_2]} F \begin{pmatrix} g & \\ & 1 \end{pmatrix} \varphi(g) \norm[\det g]_{\A}^{s-\frac{1}{2}} \ud g. 
\label{RSGlobal}
\end{equation}
	By the rapid decay of $F$ (see \cite{MS12} for example), the above integral is absolutely convergent for any $s$ even when $\varphi$ is an Eisenstein series, defining an entire function in $s$. We need to study (\ref{RSGlobal}) for $\varphi = \eis(s_1,f)$ as a function in $s_1$.
	
	The Fourier-Whittaker expansion of $F$
	$$ F(g) = \sideset{}{_{\gp{N}_2(\F) \backslash \GL_2(\F)}} \sum W_F \left( \begin{pmatrix} \gamma & \\ & 1 \end{pmatrix} g \right) $$
	readily implies the decomposition for $\Re s \gg 1$ (just like in the cuspidal case of $\varphi$)
	$$ \Psi(s,F,\eis(s_1,f)) = \sideset{}{_v} \prod \Psi_v(s,W_{F,v},W_{f_v}(s_1)), $$
	$$ \Psi_v(s,W_{F,v},W_{f_v}(s_1)) := \int_{\gp{N}_2(\F_v) \backslash \GL_2(\F_v)} W_{F,v} \begin{pmatrix} g & \\ & 1 \end{pmatrix} W_{f_v}(s_1,g) \norm[\det g]_v^{s-\frac{1}{2}} \ud g, $$
	where $W_{F,v}$ (resp. $W_{f_v}(s_1)$) is the Whittaker function of $F_v$ (resp. the flat section $f_v(s_1,\cdot)$) with respect to the additive character $\psi_v$ (resp. $\psi_v^{-1}$). At an unramified place $\vp < \infty$, the value of $W_{f_v}(s_1,g)$ is given by \cite[Proposition 4.6.5]{Bu98}, which is the product of $L_{\vp}(1+2s_1, \omega_{\vp}^{-1}\chi_{\vp}^2)^{-1}$ and its cuspidal counterpart with normalization $W_v(\mathbbm{1})=1$. Hence
\begin{align} 
	\Psi_{\vp}(s,W_{F,\vp},W_{f_{\vp}}(s_1)) &= L_{\vp}(1+2s_1, \omega_{\vp}^{-1}\chi_{\vp}^2)^{-1} L_{\vp}(s, \Pi_{\vp} \times \pi_{s_1}(\chi_{\vp}, \omega_{\vp} \chi_{\vp}^{-1})) \nonumber \\
	&= L_{\vp}(1+2s_1, \omega_{\vp}^{-1}\chi_{\vp}^2)^{-1} L_{\vp}(s+s_1, \Pi_{\vp} \times \chi_{\vp}) L_{\vp}(s-s_1, \Pi_{\vp} \times \omega_{\vp}\chi_{\vp}^{-1}).
\label{RSLocUnram}
\end{align}
\begin{proposition} \label{prop: RSEisProp}
	Let $S$ be a finite set of places including the archimedean ones such that at any $\vp \notin S$ the section $f_{\vp}$, the Whittaker function $W_{F,\vp}$ are spherical, and $\psi_{\vp}$ has conductor $\vo_{\vp}$. Then
	$$ \Psi(s,F,\eis(s_1,f)) \cdot \frac{L^{(S)}(1+2s_1, \omega^{-1}\chi^2)}{\Lambda(s+s_1, \Pi \times \chi) \Lambda(s-s_1, \Pi \times \omega\chi^{-1})} $$
	is entire in $s,s_1$. In particular, the poles of $s_1 \mapsto \Psi(s,F,\eis(s_1,f))$ are included in the zeroes of $L^{(S)}(1+2s_1, \omega^{-1}\chi^2)$, independent of $s$.
\end{proposition}
\begin{proof}
	We have the unramified computation at $\vp \notin S$ given by (\ref{RSLocUnram}). At other places $v$, $W_{f_v}(s_1)$ is entire in $s_1$. Hence $\Psi_v(\cdots)$ share the same properties as its counterpart in the case of cuspidal $\varphi$. We get
\begin{align*}
	&\quad \Psi(s,F,\eis(s_1,f)) \cdot \frac{L^{(S)}(1+2s_1, \omega^{-1}\chi^2)}{\Lambda(s+s_1, \Pi \times \chi) \Lambda(s-s_1, \Pi \times \omega\chi^{-1})} \\
	&= \prod_{v \in S} \frac{\Psi_v(s,W_{F,v},W_{f_v}(s_1))}{L_v(s+s_1, \Pi_v \times \chi_v) L_v(s-s_1, \Pi_v \times \omega_v\chi_v^{-1})},
\end{align*}
	which is entire since every factor on the right hand side is.
\end{proof}

	\subsection{Complements on Rankin-Selberg Integrals: Local Theory}
	
	We take the local setting in this subsection (see \S \ref{sec: LocNotation}).
	
	Let $\Pi$ be a generic irreducible admissible representation of $\GL_3(\F)$, whose subspace of smooth vectors is denoted by $\Pi^{\infty}$. Let $\widetilde{\Pi}$ (resp. $\widetilde{\Pi}^{\infty}$) be the contragredient of $\Pi$ (resp. $\Pi^{\infty}$). Denote by $\Whi(\Pi^{\infty};\psi)$ the Whittaker model of $\Pi^{\infty}$ with respect to $\psi$. Then for every $W \in \Whi(\Pi^{\infty};\psi)$, we have $\widetilde{W} \in \Whi(\widetilde{\Pi}^{\infty};\psi^{-1})$. For every (unitary) character $\chi$ of $\F^{\times}$ and $s \in \C$, the following two integrals are integral representations of the Rankin-Selberg $L$-functions $L(s,\Pi \times \chi)$ introduced in \cite{JPS83}:
\begin{equation}
	\Psi(s,W,\chi;0) = \int_{\F^{\times}} W \begin{pmatrix} t & & \\ & 1 & \\ & & 1 \end{pmatrix} \chi(t) \norm[t]^{s-1} \ud^{\times}t,
\label{RSIntGL3GL1Type0}
\end{equation}
\begin{equation}
	\Psi(s,W,\chi;1) = \int_{\F^{\times}} \left( \int_{\F} W \begin{pmatrix} t & & \\ x & 1 & \\ & & 1 \end{pmatrix} \ud x \right) \chi(t) \norm[t]^{s-1} \ud^{\times}t,
\label{RSIntGL3GL1Type1}
\end{equation}
	where $W \in \Whi(\Pi^{\infty};\psi)$. Both integrals are absolutely convergent for $\Re s \gg 1$, admit meromorphic continuation to $s \in \C$ and satisfy the following functional equation
\begin{equation}
	\Psi(1-s, \widetilde{\Pi}(w_{3,1}).\widetilde{W},\chi^{-1};1) = \gamma(s, \Pi \times \chi; \psi) \Psi(s,W,\chi;0).
\label{RSIntFEGL3GL1}
\end{equation}
\begin{remark}
	For non-archimedean $\F$, the above results are contained in \cite[Theorem (2.7)]{JPS83}, while for archimedean $\F$, they are contained in \cite[Theorem 2.1]{JS90}.
\end{remark}

	We shall need the above results for a special type of $W \in \Whi(\Pi^{\infty};\psi)$, namely
\begin{equation}
	W(g) = \int_{\F} W_0 \left( g \begin{pmatrix} 1 & & \\ x & 1 & \\ & & 1 \end{pmatrix} \right) \Phi(x) \ud x, \quad W_0 \in \Whi(\Pi^{\infty};\psi), \Phi \in \Sch(\F).
\label{TestWhiF}
\end{equation}
	The following result is an easy extension of \cite[Theorem (7.4)]{JPS79}.

\begin{lemma} \label{lem: RSIntFEGL3GL1Var}
	(1) The function defined by (\ref{TestWhiF}) satisfies $W \in \Whi(\Pi^{\infty};\psi)$.
	
\noindent (2) We write for any $W_0 \in \Whi(\Pi^{\infty};\psi)$
	$$ \Psi(s,W_0,\chi;\Phi) = \int_{\F^{\times}} \left( \int_{\F} W_0 \begin{pmatrix} t & & \\ x & 1 & \\ & & 1 \end{pmatrix} \Phi(x) \ud x \right) \chi(t) \norm[t]^{s-1} \ud^{\times}t. $$
	Then the above integral is absolutely convergent for $\Re s \gg 1$, admits meromorphic continuation to $s \in \C$ and satisfies the functional equation
	$$ \Psi(1-s,\widetilde{\Pi}(w_{3,1}).\widetilde{W_0},\chi^{-1};\invOFour(\Phi)) = \gamma(s, \Pi \times \chi; \psi) \Psi(s,W_0,\chi;\Phi). $$
\end{lemma}
\begin{proof}
	(1) This is obvious for non-archimedean $\F$. For $\F \in \{ \R,\C \}$, let $X$ be any element in the enveloping algebra of the Lie algebra of $\GL_3(\F)$. It suffices to prove the convergence of the following integral
	$$ \int_{\F} \Phi(x) \cdot \Pi(X) \Pi(n_1^-(x)) W_0 \ud x, \quad n_1^-(x) := \begin{pmatrix} 1 & & \\ x & 1 & \\ & & 1 \end{pmatrix} $$
	in the underlying Hilbert space $V_{\Pi}$ of $\Pi$. If we write $n_1^-(-x) X n_1^-(x)$ as a linear combination of elements of a basis in the enveloping algebra of the Lie algebra of $\GL_3(\F)$, then the length of the sum depends only on the degree of $X$, and the coefficients are at most polynomial in $x$. It follows that the dominant integral
	$$ \int_{\F} \extnorm{\Phi(x)} \cdot \extNorm{ \Pi(X) \Pi(n_1^-(x)) W_0 } \ud x = \int_{\F} \extnorm{\Phi(x)} \cdot \extNorm{ \Pi(n_1^-(-x) X n_1^-(x)).W_0 } \ud x < +\infty $$
is convergent. Hence $W \in \Whi(\Pi^{\infty};\psi)$.

\noindent (2) It suffices to prove the functional equation, as other assertions follow easily from (1). Note that, with $W$ defined by (\ref{TestWhiF}), we have
	$$ \Psi(s,W_0,\chi;\Phi) = \Psi(s,W,\chi;0). $$
	The desired functional equation follows from (\ref{RSIntFEGL3GL1}) if we can justify
\begin{equation}
	\Psi(1-s, \widetilde{\Pi}(w_{3,1}).W,\chi^{-1};1) = \Psi(1-s,\widetilde{\Pi}(w_{3,1}).\widetilde{W_0},\chi^{-1};\widehat{\Phi}).
\label{DualZetaRel}
\end{equation}
	To this end, we compute
\begin{align*}
	\widetilde{W}(g) &= W(w_3{}^tg^{-1}) = \int_{\F} \Phi(x) W_0(w_3{}^tg^{-1}n_1^-(x)) \ud x \\
	&= \int_{\F} \Phi(y) \widetilde{W_0}(gn_1^+(-y)) \ud y, \quad n_1^+(x) := {}^t \left( n_1^-(x) \right) = \begin{pmatrix} 1 & x & \\ & 1 & \\ & & 1 \end{pmatrix}.
\end{align*}
	It follows that
	$$ \widetilde{W}(gw_{3,1}) = \int_{\F} \Phi(y) \widetilde{\Pi}(w_{3,1}).\widetilde{W_0}(gn_2^+(-y)) \ud y, \quad n_2^+(x) := w_{3,1}n_1^+(x)w_{3,1} = \begin{pmatrix} 1 & & x \\ & 1 & \\ & & 1 \end{pmatrix}. $$
	In particular, we have
\begin{align*}
	\widetilde{\Pi}(w_{3,1}).\widetilde{W} \begin{pmatrix} t & & \\ x & 1 & \\ & & 1 \end{pmatrix} &= \int_{\F} \Phi(y) \widetilde{\Pi}(w_{3,1}).\widetilde{W_0} \left( \begin{pmatrix} t & & \\ x & 1 & \\ & & 1 \end{pmatrix} n_2^+(-y) \right) \ud y \\
	&= \int_{\F} \Phi(y) \widetilde{\Pi}(w_{3,1}).\widetilde{W_0} \left( \begin{pmatrix} 1 & & -ty \\ & 1 & -xy \\ & & 1 \end{pmatrix} \begin{pmatrix} t & & \\ x & 1 & \\ & & 1 \end{pmatrix} \right) \ud y \\
	&= \invOFour(\Phi)(x) \cdot \widetilde{\Pi}(w_{3,1}).\widetilde{W_0} \begin{pmatrix} t & & \\ x & 1 & \\ & & 1 \end{pmatrix},
\end{align*}
	which justifies well (\ref{DualZetaRel}).
\end{proof}

	\subsection{A Double Zeta Integral}
	
	We regard $\Phi \mapsto \Psi(s,W_0,\chi; \Phi)$ as a tempered distribution, and would like to study its Mellin transform. Precisely, for unitary characters $\chi_j$ of $\F^{\times}$, $s_j \in \C$ and $W \in \Whi(\Pi^{\infty};\psi)$ we introduce the following double zeta integral
\begin{equation} \label{DZInt}
	\DBZ{s_1}{s_2}{\chi_1}{\chi_2}{W} = \int_{(\F^{\times})^2} W \begin{pmatrix} t_1 & & \\ t_2 & 1 & \\ & & 1 \end{pmatrix} \chi_1(t_1) \chi_2(t_2) \norm[t_1]^{s_1-1} \norm[t_2]^{s_2} \ud^{\times}t_1 \ud^{\times}t_2.
\end{equation}

\begin{proposition} \label{prop: DZIntLocProp}
	(1) The integral in (\ref{DZInt}) is absolutely convergent for $\Re s_2 > 0$ and $\Re s_1 \gg 1$. Moreover, if $\Pi$ is unitary and $\RamCst$-tempered for some $0 < \RamCst < 1/2$, then the absolute convergence holds for $\Re s_2 > 0$ and $\Re(s_1) > \RamCst$.
	
\noindent (2) The integral in (\ref{DZInt}) has meromorphic continuation to $s_1,s_2 \in \C$ so that the ratio
	$$ \DBZ{s_1}{s_2}{\chi_1}{\chi_2}{W} / \left( L(s_1, \Pi \times \chi_1) L(s_2,\chi_2) \right) $$
	is holomorphic in $(s_1,s_2) \in \C^2$. Moreover, if $\F$ is archimedean, then $\DBZ{s_1}{s_2}{\chi_1}{\chi_2}{W}$ has rapid decay in any vertical region of the shape $a_j \leq \Re s_j \leq b_j$ with $a_j,b_j \in \R$ for $j=1,2$.
	
\noindent (3) Let $\Pi = \pi \boxplus \mu$, where $\pi$ is a unitary irreducible (not necessarily square-integrable) representation of $\GL_2(\F)$ and is $\vartheta$-tempered for some $0< \vartheta < 1/2$. Then the integral in (\ref{DZInt}) is absolutely convergent for $\Re s_2 > 0$ and $\Re s_1 > \vartheta$.
	
\noindent (4) The double zeta integral satisfies the following functional equation
	$$ \DBZ{1-s_1}{1-s_2}{\chi_1^{-1}}{\chi_2^{-1}}{\widetilde{\Pi}(w_{3,1}).\widetilde{W}} = \gamma(s_1,\Pi \times \chi_1,\psi) \gamma(s_2,\chi_2,\psi) \DBZ{s_1}{s_2}{\chi_1}{\chi_2}{W}. $$
\end{proposition}
\begin{proof}
	(1) \& (2) We first note that for non-archimedean $\F$, the two assertions are easy. In fact, by \cite[Lemma (4.1.5)]{JPS79} we know that the function on $\F$
	$$ x \mapsto W \begin{pmatrix} t_1 & & \\ x & 1 & \\ & & 1 \end{pmatrix} $$
	has support contained in a compact subset independent of $t_1 \in \F^{\times}$. By smoothness, there exist $n,N \in \Z_{\geq 0}$ so that the above function has support in $\vp^{-N}$, and is invariant by additive translation by $\vp^n$. Taking a system of representatives $\alpha_j$ of $\vp^{-N}/\vp^n$, we readily see that
	$$ W \begin{pmatrix} t_1 & & \\ t_2 & 1 & \\ & & 1 \end{pmatrix} = \sum_j W\left( \begin{pmatrix} t_1 & & \\ & 1 & \\ & & 1 \end{pmatrix} \begin{pmatrix} 1 & & \\ \alpha_j & 1 & \\ & & 1 \end{pmatrix} \right) \mathbbm{1}_{\alpha_j + \vp^n}(t_2). $$
	Thus the double zeta integral
\begin{align*}
	\DBZ{s_1}{s_2}{\chi_1}{\chi_2}{W} &= \int_{(\F^{\times})^2} W \begin{pmatrix} t_1 & & \\ t_2 & 1 & \\ & & 1 \end{pmatrix} \chi_1(t_1) \norm[t_1]^{s_1-1} \chi(t_2) \norm[t_2]^{s_2} \ud^{\times}t_1 \ud^{\times}t_2 \\
	&= \sum_j \int_{\F^{\times}} W\left( \begin{pmatrix} t_1 & & \\ & 1 & \\ & & 1 \end{pmatrix} \begin{pmatrix} 1 & & \\ \alpha_j & 1 & \\ & & 1 \end{pmatrix} \right) \chi_1(t_1) \norm[t_1]^{s_1-1} \ud^{\times}t_1 \cdot \\
	&\qquad \int_{\F^{\times}} \mathbbm{1}_{\alpha_j + \vp^n}(t_2) \chi_2(t_2) \norm[t_2]^{s_2} \ud^{\times} t_2
\end{align*}
	is a finite sum of products of standard integrals representing $L(s_1, \Pi \times \chi_1)$ and $L(s_2,\chi_2)$. The required properties follow from \cite[Theorem (2.7)]{JPS83} and Tate's thesis. Moreover, if $\Pi$ is unitary and $\RamCst$-tempered, then the function represented by $\DBZ{s_1}{s_2}{\chi_1}{\chi_2}{W}$ is a Laurent series in $q^{-s_1}$ and holomorphic in $\Re(s_1) > \RamCst$ (where $L(s_1, \Pi \times \chi_1)$ is holomorphic). Its radius of absolute convergence must be $> q^{-\RamCst}$, i.e., it is absolutely convergent for $\Re(s_1) > \RamCst$. We then note that the case of archimedean $\F$ satisfies the condition in (3) by Langlands's classification. Hence we leave this case to the next part.

\noindent (3) We first give another treatment of meromorphic continuation. Let $\Phi \in \Sch(2 \times 3, \Whi(\pi^{\infty}; \psi))$ (see Definition \ref{def: W-ValSchF}). The equation \eqref{eq: GenGSWhi} associates a $W \in \Whi(\Pi^{\infty},\psi)$ by the formula
\begin{equation} \label{GodementWhi}
	W(g) := \mu(\det g) \norm[\det g] \int_{\GL_2(\F)} \OFour_{\vec{3}}(g.\Phi)(h,h^{\iota}\vec{e}_2; h^{-1}) \mu(\det h) \norm[\det h]^{\frac{1}{2}} \ud h,
\end{equation}
	where $\vec{e}_2 = (0,1)^T$. Moreover, every element in $\Whi(\Pi^{\infty},\psi)$ is a such $W$ by Proposition \ref{prop: GenGSWhi}. From (\ref{GodementWhi}) we deduce, by the change of variables $h \mapsto ha(t_1)^{-1}$,
\begin{align*}
	W \begin{pmatrix} t_1 & & \\ t_2 & 1 & \\ & & 1 \end{pmatrix} &= \mu(t_1) \norm[t_1] \int_{\GL_2(\F)} \OFour_{\vec{3}}(\Phi) \left( h \begin{pmatrix} t_1 & 0 \\ t_2 & 1 \end{pmatrix}, h^{\iota} \begin{pmatrix} 0 \\ 1 \end{pmatrix}; h^{-1} \right) \mu(\det h) \norm[\det h]^{\frac{1}{2}} \ud h \\
	&= \norm[t_1]^{\frac{1}{2}} \int_{\GL_2(\F)} \OFour_{\vec{3}}(\Phi) \left( h \begin{pmatrix} 1 & 0 \\ t_2 & 1 \end{pmatrix}, h^{\iota} \begin{pmatrix} 0 \\ 1 \end{pmatrix}; a(t_1)h^{-1} \right) \mu(\det h) \norm[\det h]^{\frac{1}{2}} \ud h.
\end{align*}
	We introduce $\gp{A}^1 = \gp{A}^1(\F) = \left\{ a(t) \ \middle| \ t \in \F^{\times} \right\}$ and rewrite
\begin{align*} 
	W \begin{pmatrix} t_1 & & \\ t_2 & 1 & \\ & & 1 \end{pmatrix} &=  \norm[t_1]^{\frac{1}{2}} \int_{\GL_2(\F) / \gp{A}^1(\F)} \int_{\F^{\times}} \OFour_{\vec{3}}(\Phi) \left( h \begin{pmatrix} t & 0 \\ t_2 & 1 \end{pmatrix}, h^{\iota} \begin{pmatrix} 0 \\ 1 \end{pmatrix}; a(t_1 t^{-1})h^{-1} \right) \cdot \\
	&\qquad \mu(t) \norm[t]^{\frac{1}{2}} \mu(\det h) \norm[\det h]^{\frac{1}{2}} \ud^{\times}t \ud h.
\end{align*}
	Consequently, the double zeta integral (\ref{DZInt}) can be rewriten, at first formally and with the change of variables $t_1 \mapsto t_1t$, as
\begin{align}
	\DBZ{s_1}{s_2}{\chi_1}{\chi_2}{W} &= \int_{\GL_2(\F) / \gp{A}^1(\F)} \mu(\det h) \norm[\det h]^{\tfrac{1}{2}} \cdot \left\{ \int_{(\F^{\times})^3} \OFour_{\vec{3}}(\Phi) \left( h \begin{pmatrix} t & 0 \\ t_2 & 1 \end{pmatrix}, h^{\iota} \begin{pmatrix} 0 \\ 1 \end{pmatrix}; a(t_1)h^{-1} \right) \right. \nonumber \\
	&\quad \left. \chi_1(t_1) \norm[t_1]^{s_1-\frac{1}{2}} \mu\chi_1(t) \norm[t]^{s_1} \chi_2(t_2) \norm[t_2]^{s_2} \ud^{\times}t_1 \ud^{\times}t \ud^{\times}t_2 \right\} \ud h. \label{DZetaAC}
\end{align}
	The inner integrals over $\ud^{\times} t_1$ and $\ud^{\times} t \ud^{\times} t_2$ are standard integral representations of $L(s_1, \pi \times \chi_1)$ and $L(s_1, \mu\chi_1) L(s_2,\chi_2)$ respectively for any fixed $h$, hence admit meromorphic continuation to $(s,s_0) \in \C^2$, which become holomorphic after dividing by $L(s_1, \pi \times \chi_1) L(s_1, \mu\chi_1) L(s_2,\chi_2)$.
	
	It remains to justify the absolute convergence, and the rapid decay in the case of archimedean $\F$. To this end, we use the following Iwasawa decomposition
	$$ \GL_2(\F) = \gp{K} \gp{N} \gp{A}^1 \gp{Z}; \qquad h = \kappa n(u) a(t) z, \quad u \in \F \ \& \ t,z \in \F^{\times}. $$
	Therefore the measure on $\GL_2(\F)/\gp{A}^1$ is identified with $\ud \kappa \ud u \ud^{\times}z$. Now that for $h = \kappa n(u) z$ we have
	$$ \mu(\det h) \norm[\det h]^{\tfrac{1}{2}} = \mu(\det \kappa) \mu^2(z) \norm[z], $$
	$$ \OFour_{\vec{3}}(\Phi) \left( h \begin{pmatrix} t & 0 \\ t_2 & 1 \end{pmatrix}, h^{\iota} \begin{pmatrix} 0 \\ 1 \end{pmatrix}; a(t_1)h^{-1} \right) = \psi(-t_1u) \omega_{\pi}^{-1}(z) \OFour_{\vec{3}}(\Phi.\kappa) \left( \begin{pmatrix} z(t+ut_2) & zu & 0 \\ zt_2 & z & z^{-1} \end{pmatrix}; a(t_1)\kappa^{-1} \right), $$
	we can rewrite (\ref{DZetaAC}) as
\begin{align}
	\DBZ{s_1}{s_2}{\chi_1}{\chi_2}{W} &= \int_{\gp{K}} \int_{\F} \int_{\F^{\times}} \mu(\det \kappa) \mu^2(z) \norm[z] \cdot \nonumber \\
	&\quad \omega_{\pi}^{-1}(z) \left\{ \int_{(\F^{\times})^3} \OFour_{\vec{3}}(\Phi.\kappa n(u)) \left( \begin{pmatrix} zt & 0 & 0 \\ zt_2 & z & z^{-1} \end{pmatrix}; a(t_1)\kappa^{-1} \right) \right. \cdot \nonumber \\
	&\quad \left. \chi_1(t_1) \norm[t_1]^{s_1-\frac{1}{2}} \mu\chi_1(t) \norm[t]^{s_1} \chi_2(t_2) \norm[t_2]^{s_2} \ud^{\times}t_1 \ud^{\times}t \ud^{\times}t_2 \right\} \ud \kappa \ud u \ud^{\times}z. \label{DZetaACbis}
\end{align}
	If $\Re s_1 > \RamCst$ and $\Re s_2 > 0$, then the inner integrals are absolutely convergent. More generally, if $\Re s_1$ and $\Re s_2$ vary in compact intervals and if $s_1$ and $s_2$ are away from the possible poles, then there are Sobolev, resp. Schwartz norms so that we have uniformly
	$$ \extnorm{\int_{\F^{\times}} W_1(a(t_1)) \chi_1(t_1) \norm[t_1]^{s_1-\frac{1}{2}} \ud^{\times}t_1} \ll \Sob_1(W_1), \quad \forall \ W_1 \in \Whi(\pi^{\infty}, \psi); $$
	$$ \extnorm{\int_{(\F^{\times})^2} \phi(t,t_2) \mu\chi_1(t) \norm[t]^{s_1} \chi_2(t_2) \norm[t_2]^{s_2} \ud^{\times}t \ud^{\times}t_2 } \ll \Sob_2(\phi), \quad \forall \ \phi \in \Sch(\F^2). $$
	If $\F$ is archimedean, we even have the rapid decays for any $A \gg 1$ (with different norms)
	$$ \extnorm{\int_{\F^{\times}} W_1(a(t_1)) \chi_1(t_1) \norm[t_1]^{s_1-\frac{1}{2}} \ud^{\times}t_1} \ll (1+\norm[\Im s_1])^{-A} \Sob_1(W_1), \quad \forall \ W_1 \in \Whi(\pi^{\infty}, \psi); $$
	$$ \extnorm{\int_{(\F^{\times})^2} \phi(t,t_2) \mu\chi_1(t) \norm[t]^{s_1} \chi_2(t_2) \norm[t_2]^{s_2} \ud^{\times}t \ud^{\times}t_2 } \ll (1+\norm[\Im s_1])^{-A}(1+\norm[\Im s_2])^{-A} \Sob_2(\phi), \quad \forall \ \phi \in \Sch(\F^2). $$
	Changing $W_1$ (resp. $\phi$) with the translates $(\kappa n(u))^{-1}.W_1$ (resp. $(\kappa n(u))^T.\phi$) only increases the right hand side by a polynomial in $\norm[u]$. By Proposition \ref{prop: SchRestriction} \& \ref{prop: SchCompTrans} and Definition \ref{def: W-ValSchF} we can dominate the outer triple integrals in (\ref{DZetaACbis}) by the convergent integral
	$$ (1+\norm[\Im s_1])^{-A}(1+\norm[\Im s_2])^{-A} \int_{\gp{K}} \int_{\F} \int_{\F^{\times}} \norm[z]^{1-\Re s_1 - \Re s_2} (1+\norm[u])^B \phi \begin{pmatrix} zu & 0 \\ z & z^{-1} \end{pmatrix} \ud \kappa \ud u \ud^{\times}z, $$
	for some positive Schwartz function $\phi \in \Sch(2 \times 2, \F)$. Consequently, the new integral representations (\ref{DZetaAC}) and (\ref{DZetaACbis}) are absolutely convergent for all such $s_0,s$, which have rapid decay in the case of archimedean $\F$. This also justifies the formal computation leading to these integral representations by Fubini.
	
\noindent (4) Let $\widehat{\F^{\times}}$ be the unitary dual group of $\F^{\times}$. For $c \in \R$, write $\widehat{\F^{\times}}(c)$ for the set of quasi-characters $\chi$ of $\F^{\times}$ such that $\norm[\chi(t)] = \norm[t]^c$. This is a principal homogeneous space of $\widehat{\F^{\times}}$, which inherits the Plancherel measure of $\widehat{\F^{\times}}$ dual to $\ud^{\times}t$ on $\F^{\times}$, denoted by $\ud \mu$, so that we have the Mellin inversion formula
	$$ f(t) = \int_{\widehat{\F^{\times}}(c)} \left( \int_{\F^{\times}} f(t_1) \chi(t_1) \norm[t_1]^s \ud^{\times}t_1 \right) \ud \mu(\chi \norm^s), \quad \forall f \in \Sch(\F^{\times}). $$
For $\Re s_1 \gg 1$ large, we can change the order of integrations, apply the Mellin inversion over $\F^{\times}$ and get
\begin{align}
	\Psi(s_1,W,\chi_1;\Phi) &= \int_{\F} \left( \int_{\F^{\times}} W \begin{pmatrix} t_1 & & \\ x & 1 & \\ & & 1 \end{pmatrix} \chi_1(t_1) \norm[t_1]^{s_1-1} \ud^{\times}t_1 \right) \Phi(x) \ud x \nonumber \\
	&= \int_{\widehat{\F^{\times}}(c)} \left( \int_{(\F^{\times})^2} W \begin{pmatrix} t_1 & & \\ t_2 & 1 & \\ & & 1 \end{pmatrix} \chi_1(t_1) \norm[t_1]^{s_1-1} \chi_2(t_2) \norm[t_2]^{s_2} \ud^{\times}t_1 \ud^{\times}t_2 \right) \cdot \nonumber \\
	&\qquad \left( \int_{\F^{\times}} \Phi(t_2) \chi_2(t_2)^{-1} \norm[t_2]^{1-s_2} \ud^{\times}t_2 \right) \ud \mu(\chi_2 \norm^{s_2}) \nonumber \\
	&= \int_{\widehat{\F^{\times}}(c)} \DBZ{s_1}{s_2}{\chi_1}{\chi_2}{W} \cdot \Zeta(1-s_2,\chi_2^{-1},\Phi) \ud \mu(\chi_2 \norm^{s_2}), \label{RSDZMellinRel}
\end{align}
	where $0<c=\Re s_2<1$. Note that both sides of (\ref{RSDZMellinRel}) have meromorphic continuation to $s_1 \in \C$ with absolutely convergent integral on the right hand side. Similarly, we have
\begin{equation}
	\Psi(1-s_1,\widetilde{\Pi}(w_{3,1})\widetilde{W},\chi_1^{-1};\invOFour(\Phi)) = \int_{\widehat{\F^{\times}}(c)} \DBZ{1-s_1}{1-s_2}{\chi_1^{-1}}{\chi_2^{-1}}{\widetilde{\Pi}(w_{3,1})\widetilde{W}} \cdot \Zeta(s_2,\chi_2,\invOFour(\Phi)) \ud \mu(\chi_2 \norm^{s_2}).
\label{RSDZMellinRelBis}
\end{equation}
	Applying the functional equation in Lemma \ref{lem: RSIntFEGL3GL1Var} (2) and Tate's local functional equation, we get
\begin{align*} 
	&\quad \int_{\widehat{\F^{\times}}(c)} \DBZ{1-s_1}{1-s_2}{\chi_1^{-1}}{\chi_2^{-1}}{\widetilde{\Pi}(w_{3,1})\widetilde{W}} \cdot \Zeta(s_2,\chi_2,\invOFour(\Phi)) \ud \mu(\chi_2 \norm^{s_2}) \\
	&= \gamma(s_1,\Pi \times \chi_1, \psi) \int_{\widehat{\F^{\times}}(c)} \DBZ{s_1}{s_2}{\chi_1}{\chi_2}{W} \cdot \gamma(s_2,\chi_2,\psi) \Zeta(s_2,\chi_2,\invOFour(\Phi)) \ud \mu(\chi_2 \norm^{s_2}).
\end{align*}
	The desired functional equation for $0<\Re s_2<1$ follows by the denseness of the Mellin transform for $\Phi \in \Sch(\F)$, and by meromorphic continuation of both sides to $s_2 \in \C$.
\end{proof}

\section{Global Distributions}

	We take the global setting in this section (see \S \ref{sec: GlobNotation}).

	Fix a cuspidal automorphic representation $\Pi$ of $\GL_3(\A)$ and a unitary Hecke character $\omega$ of $\F^{\times} \backslash \A^{\times}$. The main distribution is defined on $V_{\Pi}^{\infty}$, the space of smooth vectors in (the automorphic realization of) $\Pi$, by the formula
\begin{equation} \label{eq: MainDis}
	\Theta(F) := \int_{\F \backslash \A} \left( \int_{\F^{\times} \backslash \A^{\times}} F \begin{pmatrix} t & & \\ & 1 & x \\ & & 1 \end{pmatrix} (\omega\omega_{\Pi})^{-1}(t) \ud^{\times}t \right) \psi(-x) \ud x.
\end{equation}

	Although the above integral defining $\Theta(F)$ is absolutely convergent, it will turn out to be convenient to introduce a holomorphic variant
\begin{equation} \label{eq: MainDisVar}
	\Theta(s_0, F) := \int_{\F \backslash \A} \left( \int_{\F^{\times} \backslash \A^{\times}} F \begin{pmatrix} t & & \\ & 1 & x \\ & & 1 \end{pmatrix} (\omega\omega_{\Pi})^{-1}(t) \norm[t]_{\A}^{s_0} \ud^{\times}t \right) \psi(-x) \ud x.
\end{equation}
	The above integral is still absolutely convergent for any $s_0 \in \C$ thanks to the rapid decay of $F$ at the cusp varieties. We obviously have
	$$ \Theta(F) = \Theta(0, F). $$
	We are going to decompose $\Theta(s_0, F)$ for $\Re s_0 \gg 1$ in two different ways with meromorphic continuation to $s_0 \in \C$, and get our main identity as the equality of the two decompositions at $s_0=0$.
	
\begin{theorem} \label{thm: MainId}
	(1) The distribution $\Theta(s_0,F)$ has a meromorphic continuation to $\norm[\Re s_0] < 1/2$ given by
\begin{align*}
	\Theta(s_0,F) &= \frac{1}{\zeta_{\F}^*} \sum_{\chi \in \widehat{\F^{\times} \R_+ \backslash \A^{\times}}} \int_{-\infty}^{\infty} \DBZ{1/2+i\tau}{s_0+1/2-i\tau}{\chi}{(\chi \omega \omega_{\Pi})^{-1}}{W_F} \frac{\ud \tau}{2\pi} \\
	&\quad + \frac{1}{\zeta_{\F}^*} \Res_{s_1=s_0} \DBZ{s_1+1}{s_0-s_1}{(\omega \omega_{\Pi})^{-1}}{\mathbbm{1}}{W_F} - \frac{1}{\zeta_{\F}^*} \Res_{s_1=s_0-1} \DBZ{s_1+1}{s_0-s_1}{(\omega \omega_{\Pi})^{-1}}{\mathbbm{1}}{W_F},
\end{align*}
	where $\Zeta(\cdots)$ is an integral representation of $L(1/2+i\tau, \Pi \times \chi) L(s_0+1/2-i\tau, (\chi \omega \omega_{\Pi})^{-1})$ (see (\ref{DZIntGlobal})).
	
\noindent (2) The distribution $\Theta(s_0,F)$ for $\norm[\Re s_0] < 1/2$ has another expression as
	$$ \Theta(s_0,F) = \sideset{}{_{\substack{\pi \text{ cuspidal} \\ \omega_{\pi}=\omega^{-1}}}} \sum \Theta(s_0,F \mid \pi) + \sum_{\chi \in \widehat{\R_+ \F^{\times} \backslash \A^{\times}}} \int_{-\infty}^{\infty} \Theta(s_0,F \mid \pi(\chi,\omega^{-1}\chi^{-1})) \frac{\ud \tau}{4\pi}, $$
	where $\Theta(s_0,F \mid \pi)$ (resp. $\Theta(s_0,F \mid \pi(\chi,\omega^{-1}\chi^{-1}))$) is a distribution representing $L((1-s_0)/2, \Pi \times \widetilde{\pi})$ (resp. $L((1-s_0)/2-i\tau, \Pi \times \chi^{-1}) L((1-s_0)/2+i\tau, \Pi \times \omega\chi)$) (see (\ref{2ndDecompCusp}) and (\ref{2ndDecompEis})).
\end{theorem}

	\subsection{First Decomposition}
	
	Recall the Fourier coefficients of $F$ for $\beta_1, \beta_2 \in \F$ associated with a unipotent subgroup
	$$ F_{(\beta_1,\beta_2)}(g) := \int_{(\F \backslash \A)^2} F \left( \begin{pmatrix} 1 & & x_1 \\ & 1 & x_2 \\ & & 1 \end{pmatrix} g \right) \psi(-\beta_1 x_1 - \beta_2 x_2) \ud x_1 \ud x_2. $$
	We have
	$$ \int_{\F \backslash \A} F \left( \begin{pmatrix} 1 & & \\ & 1 & x \\ & & 1 \end{pmatrix} g \right) \psi(-x) \ud x = \sum_{\beta_1 \in \F} F_{(\beta_1,1)}(g) = \sum_{\beta_1 \in \F} F_{(0,1)} \left( \begin{pmatrix} 1 & & \\ \beta_1 & 1 & \\ & & 1 \end{pmatrix} g \right). $$
	Also recall the Whittaker function of $F$ defined by
	$$ W_F(g) := \int_{(\F \backslash \A)^3} F \left( \begin{pmatrix} 1 & x_3 & x_1 \\ & 1 & x_2 \\ & & 1 \end{pmatrix} g \right) \psi(-x_2-x_3) \ud x_1 \ud x_2 \ud x_3, $$
	we have the relation (by cuspidality of $F$, see (\ref{eq: CuspCondGL3}))
	$$ F_{(0,1)}(g) = \sum_{\alpha \in \F^{\times}} W_F \left( \begin{pmatrix} \alpha & & \\ & 1 & \\ & & 1 \end{pmatrix} g \right). $$
	Hence we can rewrite, at first formally
\begin{align}
	\Theta(s_0,F) &= \int_{\F^{\times} \backslash \A^{\times}} \left( \sum_{\alpha \in \F^{\times}} \sum_{\beta_1 \in \F} W_F \left( \begin{pmatrix} \alpha & & \\ \beta_1 & 1 & \\ & & 1 \end{pmatrix} \begin{pmatrix} t & & \\ & 1 & \\ & & 1 \end{pmatrix} \right) \right) (\omega\omega_{\Pi})^{-1}(t) \norm[t]_{\A}^{s_0} \ud^{\times}t \nonumber \\
	&= \int_{\F^{\times} \backslash \A^{\times}} \left( \sum_{\alpha \in \F^{\times}} W_F \begin{pmatrix} \alpha t & & \\ & 1 & \\ & & 1 \end{pmatrix} \right) (\omega\omega_{\Pi})^{-1}(t) \norm[t]_{\A}^{s_0} \ud^{\times}t \nonumber \\
	&\quad + \int_{\F^{\times} \backslash \A^{\times}} \left( \sum_{\alpha \in \F^{\times}} \sum_{\beta \in \F^{\times}} W_F \begin{pmatrix} \alpha t & & \\ \beta t & 1 & \\ & & 1 \end{pmatrix} \right) (\omega\omega_{\Pi})^{-1}(t) \norm[t]_{\A}^{s_0} \ud^{\times}t. \label{1stDecompInitial}
\end{align}

	The first integral on the right hand side of (\ref{1stDecompInitial}) is simply a global Rankin-Selberg integral for $\GL_3 \times \GL_1$, namely
\begin{align*} 
	\Psi \left( s_0+1, F, (\omega\omega_{\Pi})^{-1} \right) &= \int_{\F^{\times} \backslash \A^{\times}} \left( \int_{(\F \backslash \A)^2} F \begin{pmatrix} t & & x_1 \\ & 1 & x_2 \\ & & 1 \end{pmatrix} \psi(-x_2) \ud x_1 \ud x_2 \right) (\omega\omega_{\Pi})^{-1}(t) \norm[t]_{\A}^{s_0} \ud^{\times}t \\
	&= \int_{\F^{\times} \backslash \A^{\times}} \left( \sum_{\alpha \in \F^{\times}} W_F \begin{pmatrix} \alpha t & & \\ & 1 & \\ & & 1 \end{pmatrix} \right) (\omega\omega_{\Pi})^{-1}(t) \norm[t]_{\A}^{s_0} \ud^{\times}t
\end{align*}
	in the notation of \cite{JPS83}, which is an integral representation of $L(s_0+1,\Pi \times (\omega\omega_{\Pi})^{-1})$. Hence the absolute convergence for $\Re s_0 \gg 1$ and meromorphic continuation to $s_0 \in \C$ of this term follow easily. 
\begin{remark}
	We recall the other global Rankin-Selberg integral for $\GL_3 \times \GL_1$
\begin{align*} 
	\widetilde{\Psi} \left( s, F, \chi \right) &= \int_{\F^{\times} \backslash \A^{\times}} \left( \int_{(\F \backslash \A)^2} F \begin{pmatrix} t & & \\ x_1 t & 1 & x_2 \\ & & 1 \end{pmatrix} \psi(-x_2) \ud x_1 \ud x_2 \right) \chi(t) \norm[t]_{\A}^s \ud^{\times}t \\
	&= \int_{\A^{\times}} \left( \int_{\A} W_F \begin{pmatrix} t & & \\ x & 1 & \\ & & 1 \end{pmatrix} \ud x \right) \chi(t) \norm[t]_{\A}^{s-1} \ud^{\times}t,
\end{align*}
	and the functional equation $\widetilde{\Psi}(1-s, \widetilde{\Pi}(w_{3,1})\widetilde{F}, \chi^{-1}) = \Psi(s,F,\chi)$. Obviously, $\widetilde{\Psi}(s,F,\chi)$ has an infinite product decomposition, which we write as
	$$ \widetilde{\Psi}(s,F,\chi) = \sideset{}{_v} \prod \widetilde{\Psi}_v(s,W_{F,v},\chi_v). $$
\label{RSFEGlobal}
\end{remark}
	
	To prove the absolute convergence of the second integral on the right hand side of (\ref{1stDecompInitial}), we first recall a fundamental estimation, which was implicitly used in \cite[\S 2.6.2]{Wu14} to control the dominant of the Fourier-Whittaker expansion of an automorphic form for $\GL_2$. Note that a refined version also appeared as \cite[Lemma 5.37]{Wu19_TF} which implies the following result needed here. 

\begin{lemma}
	Let $f: \A^{\times} \to \C$ be a function on the ideles, which is decomposable in the sense that
	$$ f((t_v)_v) = \sideset{}{_v} \prod f_v(t_v), \quad \forall (t_v)_v \in \A^{\times}, $$
	where $f_{\vp} \mid_{\vo_{\vp}^{\times}} = 1$ at any $\vp \nmid \idlJ$ for some integral ideal $\idlJ$ but is \emph{not necessarily} equal to the characteristic function of $\vo_{\vp}^{\times}$. Suppose there is a constant $c \in \R$ such that for any $N > 0$ we have
	$$ \left\{ \begin{matrix} f_v(t) \ll_{c,N} \min(\norm[t]_v^c, \norm[t]_v^{-N}) & \forall v \mid \infty \\ f_{\vp}(t) \ll \norm[t]_{\vp}^c \mathbbm{1}_{\mathrm{ord}_{\vp}(t) \geq - \mathrm{ord}_{\vp}(\idlJ)} & \forall \vp < \infty \ \& \ \vp \mid \idlJ \\ f_{\vp}(t) \leq \norm[t]_{\vp}^c \mathbbm{1}_{\mathrm{ord}_{\vp}(t) \geq 0} & \forall \vp < \infty \ \& \ \vp \nmid \idlJ\end{matrix} \right. . $$
	Then we have for any $N \gg 1$ the estimation of
	$$ \sideset{}{_{\alpha \in \F^{\times}}} \sum \norm[f(\alpha t)] \ll_{c,N,\idlJ} \min (\norm[t]_{\A}^{c-1}, \norm[t]_{\A}^{-N}). $$
\label{FundSumEst}
\end{lemma}

\noindent Next, we recall some uniform bound of the Whittaker function $W_F$, which was established in the Rankin-Selberg theory by \cite{JPS83, J09}. For our purpose, we only need a special case for $\GL_3$ stated as follows.

\begin{lemma}
	There is a constant $M \in \R$ depending only on $\Pi$ and an integral ideal $\idlJ$ depending only on $F$, such that for any $N_1,N_2 > 1$ we have
\begin{align*} 
	\extnorm{ W_F \begin{pmatrix} t_1 & & \\ t_2 & 1 & \\ & & 1 \end{pmatrix} } &\ll_{N_1,N_2,F} \sideset{}{_{v \mid \infty}} \prod \min (\norm[t_{1,v}]^M, \norm[t_{1,v}]_v^{-N_1}) \min (1, \norm[t_{2,v}]_v^{-N_2}) \\
	&\quad \cdot \sideset{}{_{\vp < \infty}} \prod \norm[t_{1,\vp}]_{\vp}^M \mathbbm{1}_{\mathrm{ord}_{\vp}(t_{1,\vp}), \mathrm{ord}_{\vp}(t_{2,\vp}) \geq - \mathrm{ord}_{\vp}(\idlJ)}.
\end{align*}
\label{WhiUBd}
\end{lemma}
\begin{proof}
	Without loss of generality, we may assume $W_F = \otimes_v W_{F,v}$ is decomposable. At an archimedean place $v \mid \infty$, the argument in \cite[\S 5.2]{J09}, which is based on \cite[Proposition 3.3]{J09}, shows the existence of some constant $M_v \in \R$ depending only on $\Pi_v$ so that for any $N_1,N_2 > 1$
	$$ \extnorm{ W_{F,v} \begin{pmatrix} t_{1,v} & & \\ t_{2,v} & 1 & \\ & & 1 \end{pmatrix} } \ll_{N_1,N_2,F_v} \norm[t_{1,v}]_v^{M_v} (1+\norm[t_{1,v}]_v^2)^{-N_1} (1+\norm[t_{2,v}]_v^2)^{-N_2}. $$
	At a non-archimedean $\vp < \infty$ such that $W_{F,\vp}$ is not spherical, the proof of \cite[Lemma (2.6)]{JPS83} implies the existence of some constant $M_{\vp} \in \R$ depending only on $\Pi_{\vp}$, some integral ideal $\idlJ_{\vp}' \subset \vo_{\vp}$ depending only on $F_{\vp}$ and some Schwartz function $\Phi_{\vp} \in \Sch(\F_{\vp})$ so that
	$$ \extnorm{ W_{F,\vp} \begin{pmatrix} t_{1,\vp} & & \\ t_{2,\vp} & 1 & \\ & & 1 \end{pmatrix} } \leq \norm[t_{1,\vp}]_{\vp}^{M_{\vp}} \Phi_{\vp}(t_{1.\vp}) \mathbbm{1}_{(\idlJ_{\vp}')^{-1}}(t_{2,\vp}) \ll \norm[t_{1,\vp}]_{\vp}^{M_{\vp}} \mathbbm{1}_{\mathrm{ord}_{\vp}(t_{1,\vp}), \mathrm{ord}_{\vp}(t_{2,\vp}) \geq - \mathrm{ord}_{\vp}(\idlJ_{\vp})}, $$
	where $\idlJ_{\vp}^{-1}$ contains $(\idlJ_{\vp}')^{-1}$ and the support of $\Phi_{\vp}$. At a non-archimedean $\vp < \infty$ such that $W_{F,\vp}$ is spherical, we first have by the proof of \cite[Lemma (4.1.5)]{JPS79} that
	$$ W_{F,\vp} \begin{pmatrix} t_{1,\vp} & & \\ t_{2,\vp} & 1 & \\ & & 1 \end{pmatrix} = W_{F,\vp} \begin{pmatrix} t_{1,\vp} & & \\ & 1 & \\ & & 1 \end{pmatrix} \mathbbm{1}_{\vo_{\vp}}(t_{2,\vp}). $$
	The value of the spherical Whittaker function is due to Shintani \cite{Sh76}. If $(\alpha_1,\alpha_2,\alpha_3)$ is the Satake parameter of $\Pi_{\vp}$, and if we normalize $W_{F,\vp}(\mathbbm{1})=1$, then
	$$ W_{F,\vp} \begin{pmatrix} \varpi^n & & \\ & 1 & \\ & & 1 \end{pmatrix} = \mathbbm{1}_{n \geq 0} \cdot q^{-n} \cdot \frac{\begin{vmatrix} \alpha_1^{n+2} & \alpha_2^{n+2} & \alpha_3^{n+2} \\ \alpha_1 & \alpha_2 & \alpha_3 \\ 1 & 1 & 1 \end{vmatrix}}{(\alpha_1 - \alpha_2) (\alpha_1 - \alpha_3) (\alpha_2 - \alpha_3)}. $$
	Since $\Pi$ is unitary, it is $\theta$-tempered for any constant towards the Ramanujan-Petersson conjecture for $\GL_3$ (current record $\theta = 5/14$). Hence $\norm[\alpha_j] < q^{\theta}$. We readily deduce
	$$ \extnorm{ W_{F,\vp} \begin{pmatrix} t_{1,\vp} & & \\ t_{2,\vp} & 1 & \\ & & 1 \end{pmatrix} } \leq C_{\vp} \norm[t_{1,\vp}]_{\vp}^{1-\theta} \mathbbm{1}_{\vo_{\vp}}(t_{1,\vp}) \mathbbm{1}_{\vo_{\vp}}(t_{2,\vp}), $$
	and up to a finite number of exceptional places, we can take $C_{\vp} = 1$. The desired estimation follows by taking $M$ to be the smallest one among the $M_v$ and $M_{\vp}$, $\idlJ$ to be the product of $\idlJ_{\vp}$.
\end{proof}
\begin{remark}
	It should be possible to take $M=1-\theta$ in the above lemma. But this seems to be a difficult problem, and is not yet available in the literature so far.
\end{remark}

\begin{corollary}
	There is $M \in \R$ depending only on $\Pi$ such that for any $N_1,N_2 > 1$
	$$ \sum_{\alpha,\beta \in \F^{\times}} \extnorm{ W_F \begin{pmatrix} \alpha t_1 & & \\ \beta t_2 & 1 & \\ & & 1 \end{pmatrix} } \ll_{N_1,N_2} \min(\norm[t_1]_{\A}^{M-1}, \norm[t_1]_{\A}^{-N_1}) \cdot \min(\norm[t_2]_{\A}^{-1}, \norm[t_2]_{\A}^{-N_2}). $$
\label{SumWhiUBd}
\end{corollary}

\noindent This corollary is an immediate consequence of the previous two lemmas. It implies readily
	$$ \sum_{\alpha,\beta \in \F^{\times}} \extnorm{ W_F \begin{pmatrix} \alpha t & & \\ \beta t & 1 & \\ & & 1 \end{pmatrix} } \ll_N \min(\norm[t]_{\A}^{M-2}, \norm[t]_{\A}^{-N}), $$
	and justifies the absolute convergence of (\ref{1stDecompInitial}) for $\Re s_0 > 2-M$.
	
	In order to get the meromorphic continuation, we introduce the following (global) double zeta integral
\begin{equation}
	\DBZ{s_1}{s_2}{\chi_1}{\chi_2}{W_{F}} = \int_{(\A^{\times})^2} W_F \begin{pmatrix} t_1 & & \\ t_2 & 1 & \\ & & 1 \end{pmatrix} \chi_1(t_1) \chi_2(t_2) \norm[t_1]^{s_1-1} \norm[t_2]^{s_2} \ud^{\times}t_1 \ud^{\times}t_2,
\label{DZIntGlobal}
\end{equation}
	where $\chi_j$'s are unitary Hecke characters of $\F^{\times} \backslash \A^{\times}$.
	
\begin{proposition}
	(1) The integral in (\ref{DZIntGlobal}) is absolutely convergent for $\Re s_2 > 1$ and $\Re s_1 \gg 1$.
	
\noindent (2) The integral in (\ref{DZIntGlobal}) has meromorphic continuation to $s_1,s_2 \in \C$ so that the ratio
	$$ \DBZ{s_1}{s_2}{\chi_1}{\chi_2}{W_F} / \left( L(s_1, \Pi \times \chi_1) L(s_2,\chi_2) \right) $$
	is holomorphic in $(s_1,s_2) \in \C^2$. Moreover, it has rapid decay in any vertical region of the shape $a_j \leq \Re s_j \leq b_j$ with $a_j,b_j \in \R$ for $j=1,2$.
	
\noindent (3) The double zeta integral satisfies the following functional equation
	$$ \DBZ{1-s_1}{1-s_2}{\chi_1^{-1}}{\chi_2^{-1}}{\widetilde{\Pi}(w_{3,1}).\widetilde{W_F}} = \DBZ{s_1}{s_2}{\chi_1}{\chi_2}{W_F}. $$
	
\noindent (4) The double zeta integral has possible simple poles at $s_2 \in \{ 0,1 \}$ for $\chi_2 = \mathbbm{1}$ with residues
	$$ \Res_{s_2=1} \DBZ{s_1}{s_2}{\chi_1}{\mathbbm{1}}{W_F} = \zeta_{\F}^* \int_{\A^{\times}} \left( \int_{\A} W_F \begin{pmatrix} t_1 & & \\ x & 1 & \\ & & 1 \end{pmatrix} \ud x \right) \chi_1(t_1) \norm[t_1]^{s_1-1} \ud^{\times}t_1 = \zeta_{\F}^* \widetilde{\Psi}(s_1, F, \chi_1), $$
	$$ \Res_{s_2=0} \DBZ{s_1}{s_2}{\chi_1}{\mathbbm{1}}{W_F} = -\zeta_{\F}^* \int_{\A^{\times}} W_F \begin{pmatrix} t_1 & & \\ & 1 & \\ & & 1 \end{pmatrix} \chi_1(t_1) \norm[t_1]^{s_1-1} \ud^{\times}t_1 = -\zeta_{\F}^* \Psi(s_1, F, \chi_1). $$
\label{DZIntGlobalProp}
\end{proposition}
\begin{proof}
	(1)-(3) follow from their local counterpart given in Proposition \ref{prop: DZIntLocProp}. To prove (4), we first notice that the second formula follows from the first one via the functional equation in (3) and the functional equation of the global Rankin-Selberg integrals recalled in Remark \ref{RSFEGlobal}. Thus it suffices to prove the first formula. We first suppose $\Re s_1 \gg 1$ is large and take $s_2 > 1$, so that the double zeta integral is absolutely convergent. Then we have an obvious infinite product decomposition
	$$ \DBZ{s_1}{s_2}{\chi_1}{\mathbbm{1}}{W_F} = \sideset{}{_v} \prod \DBZLoc{s_1}{s_2}{\chi_{1,v}}{\mathbbm{1}}{W_{F,v}}. $$
	For a finite number of places $S$ containing the archimedean ones we have for any $\vp \notin S$
	$$ W_{F,\vp} \begin{pmatrix} t_1 & & \\ t_2 & 1 & \\ & & 1 \end{pmatrix} = W_{F,\vp} \begin{pmatrix} t_1 & & \\ & 1 & \\ & & 1 \end{pmatrix} \cdot \mathbbm{1}_{\vo_{\vp}}(t_2) \quad \Rightarrow $$
	$$ \DBZLoc[\vp]{s_1}{s_2}{\chi_{1,{\vp}}}{\mathbbm{1}}{W_{F,\vp}} = \widetilde{\Psi}_{\vp}(s_1, W_{F,\vp}, \chi_{1,\vp}) \cdot \zeta_{\vp}(s_2). $$
	In other words, we have
	$$ \DBZ{s_1}{s_2}{\chi_1}{\mathbbm{1}}{W_F} = \zeta^{S}(s_2) \cdot \sideset{}{_{v \in S}} \prod \DBZLoc{s_1}{s_2}{\chi_{1,v}}{\mathbbm{1}}{W_{F,v}} \cdot \sideset{}{_{\vp \notin S}} \prod \widetilde{\Psi}_{\vp}(s_1, W_{F,\vp}, \chi_{1,\vp}). $$
	The pole at $s_2=1$ is given by $\zeta^{S}(s_2)$ while the other terms are regular at $s_2 = 1$ with
	$$ \DBZLoc{s_1}{1}{\chi_{1,v}}{\mathbbm{1}}{W_{F,v}} = \zeta_v(1) \cdot \widetilde{\Psi}_{\vp}(s_1, W_{F,v}, \chi_{1,v}). $$
	The desired formula then follows readily in this case. The general case follows by the uniqueness of meromorphic continuation on both sides.
\end{proof}

\noindent Let $M$ be the constant determined by $\Pi$ in Corollary \ref{SumWhiUBd}. We fix a $s_0$ with $\Re s_0 > 1-M$, and consider the following function on $\F^{\times} \backslash \A^{\times}$
\begin{equation} 
	f(t_1) := \int_{\F^{\times} \backslash \A^{\times}} \left( \sum_{\alpha \in \F^{\times}} \sum_{\beta \in \F^{\times}} W_F \begin{pmatrix} \alpha t t_1 & & \\ \beta t & 1 & \\ & & 1 \end{pmatrix} \right) (\omega\omega_{\Pi})^{-1}(t) \norm[t]_{\A}^{s_0} \ud^{\times}t. 
\label{AuxFtoMInv}
\end{equation}

\begin{lemma}
	(1) The defining integral of $f(t_1)$ is absolutely convergent integral for every $t_1 \in \A^{\times}$.
	
\noindent (2) The function $f(t_1)$ is smooth for the action of $\A^{\times}$.

\noindent (3) The function $f(t_1)$ admits the Mellin inversion
	$$ f(1) = \frac{1}{\Vol(\F^{\times} \R_+ \backslash \A^{\times})} \sum_{\chi \in \widehat{\F^{\times} \R_+ \backslash \A^{\times}}} \int_{(c)} \left( \int_{\F^{\times} \backslash \A^{\times}} f(t_1) \chi(t_1) \norm[t_1]_{\A}^{s_1} \ud^{\times}t_1 \right) \frac{\ud s_1}{2\pi i} $$
	for any $c$ satisfying $1-M < c < \Re s_0-1$.
\end{lemma}
\begin{proof}
	(1) Applying Corollary \ref{SumWhiUBd}, we get
\begin{multline*} 
	\sum_{\alpha,\beta \in \F^{\times}} \extnorm{ W_F \begin{pmatrix} \alpha t t_1 & & \\ \beta t & 1 & \\ & & 1 \end{pmatrix} } \ll_{N_1,N_2} \min(\norm[tt_1]_{\A}^{M-1}, \norm[tt_1]_{\A}^{-N_1}) \cdot \min(\norm[t]_{\A}^{-1}, \norm[t]_{\A}^{-N_2}) \\
	= \min \left( \norm[t_1]_{\A}^{M-1} \cdot \min (\norm[t]_{\A}^{M-2}, \norm[t]_{\A}^{M-N_2}), \norm[t_1]_{\A}^{-N_1} \cdot \min( \norm[t]_{\A}^{-N_1-1}, \norm[t]_{\A}^{-N_1-N_2} ) \right).
\end{multline*}
	We choose $1-M < N_1 < \Re s_0-1$ arbitrarily, and $N_2$ large to get the absolute convergence of the integral of right hand side against $\norm[t]_{\A}^{s_0} \ud^{\times}t$, together with the estimation
\begin{equation} 
	\norm[f(t_1)] \leq \int_{\F^{\times} \backslash \A^{\times}} \sum_{\alpha,\beta \in \F^{\times}} \extnorm{ W_F \begin{pmatrix} \alpha t t_1 & & \\ \beta t & 1 & \\ & & 1 \end{pmatrix} } \norm[t]_{\A}^{\Re s_0} \ud^{\times}t \ll_{N_1} \min (\norm[t_1]_{\A}^{M-1}, \norm[t_1]_{\A}^{-N_1}). 
\label{ParIntBd}
\end{equation}

\noindent (2) Having proved the absolute convergence in (1), the smoothness for $\A_{\fin}^{\times}$ obviously follows from the smoothness of $W_F$. The smoothness for $\A_{\infty}^{\times}$ is defined by the differentials $D_v = t_{1,v} \partial_{t_{1,v}}$ at $\F_v = \R$, resp. $D_v = \rho_v \partial_{\rho_v}$ and $D_v' = \partial_{\theta_v}$ at $\F_v = \C$ for the polar coordinates $t_{1,v} = \rho_v e^{i\theta_v}$. They correspond to the left differential operators defined by $E_v = \mathrm{diag}(1,0,0)$, resp. $E_v$ and $E_v' = \mathrm{diag}(i,0,0)$  of the Lie algebra $\mathfrak{gl}_3(\F_v)$. Let $H_v$ resp. $H_v$ and $H_v'$ be elements in $\mathfrak{gl}_3(\F_v)$ given by
	$$ H_v = \begin{bmatrix} 0 & 0 & 0 \\ 1 & 0 & 0 \\ 0 & 0 & 0 \end{bmatrix}, \quad H_v' = \begin{bmatrix} 0 & 0 & 0 \\ i & 0 & 0 \\ 0 & 0 & 0 \end{bmatrix}. $$
	It is easy to write the left differentials by the right differentials via
	$$ \begin{pmatrix} t_v t_{1,v} & & \\ t_v & 1 & \\ & & 1 \end{pmatrix}^{-1} E_v \begin{pmatrix} t_v t_{1,v} & & \\ t_v & 1 & \\ & & 1 \end{pmatrix} = E_v - t_v H_v, $$
	$$ \begin{pmatrix} t_v t_{1,v} & & \\ t_v & 1 & \\ & & 1 \end{pmatrix}^{-1} E_v' \begin{pmatrix} t_v t_{1,v} & & \\ t_v & 1 & \\ & & 1 \end{pmatrix} = E_v' - t_v H_v'. $$
	In other words, if we introduce
	$$ g(t_1,t) := W_{\F} \begin{pmatrix} t_v t_{1,v} & & \\ t_v & 1 & \\ & & 1 \end{pmatrix}, $$
	then we have the relations
	$$ D_vg = W_{E_v.F} - t_v W_{H_v.F}, \quad D_v'g = W_{E_v'.F} - t_v W_{H_v'.F}. $$
	Hence for any polynomial $P$ in $D_v,D_v'$, there is a polynomial $Q$ in $E_v,E_v',H_v,H_v',t_v$ such that $P.g = W_{Q.F}$. It follows that $P.g(t_1,t)$ satisfies the same bound offered by Lemma \ref{WhiUBd} and Corollary \ref{SumWhiUBd} as $g(t_1,t)$, because for any integer $k$ we have
	$$ \norm[t_v]_v^k \min(1,\norm[t_v]_v^{-N_2}) \leq \min(1,\norm[t_v]_v^{k-N_2}). $$
	Hence we can pass the differentials in the sum and integral defining $f(t_1)$, proving its smoothness.
	
\noindent (3) By (1) and (2), $f(t_1)$ is a smooth function on $\F^{\times} \backslash \A^{\times}$, whose derivatives with respect to any polynomial in $D_v,D_v'$ at $v \mid \infty$ satisfy the bound (\ref{ParIntBd}). Hence $f(t_1)$ is Mellin invertible over $\F^{\times} \backslash \A^{\times}$, and the stated formula holds for $1-M < c < N_1$. Since $N_1$ can be arbitrarily close to $\Re s_0 - 1$, we conclude.
\end{proof}

\begin{remark}
	The above argument is quite similar to that in \cite[Lemma 3.5]{Wu22}. The only difference is that we replace \cite[Proposition 2.8]{Wu22} with its generalization Lemma \ref{FundSumEst}.
\end{remark}

	We can finally rewrite the equation (\ref{1stDecompInitial}) as follows.
\begin{align*}
	&\quad \Theta(s_0,F) - \Psi \left( s_0+1, F, (\omega\omega_{\Pi})^{-1} \right) \\
	&= \int_{\F^{\times} \backslash \A^{\times}} \left( \sum_{\alpha \in \F^{\times}} \sum_{\beta \in \F^{\times}} W_F \begin{pmatrix} \alpha t & & \\ \beta t & 1 & \\ & & 1 \end{pmatrix} \right) (\omega\omega_{\Pi})^{-1}(t) \norm[t]_{\A}^{s_0} \ud^{\times}t \\
	&= \frac{1}{\Vol(\F^{\times} \R_+ \backslash \A^{\times})} \sum_{\chi \in \widehat{\F^{\times} \R_+ \backslash \A^{\times}}} \int_{(c)} \left( \int_{(\A^{\times})^2} W_F \begin{pmatrix} tt_1 & & \\ t & 1 & \\ & & 1 \end{pmatrix} \chi(t_1) \norm[t_1]_{\A}^{s_1} (\omega\omega_{\Pi})^{-1}(t) \norm[t]_{\A}^{s_0} \ud^{\times}t_1 \ud^{\times}t \right) \frac{\ud s_1}{2\pi i} \\
	&= \frac{1}{\Vol(\F^{\times} \R_+ \backslash \A^{\times})} \sum_{\chi \in \widehat{\F^{\times} \R_+ \backslash \A^{\times}}} \int_{(c)} \DBZ{s_1+1}{s_0-s_1}{\chi}{(\chi \omega \omega_{\Pi})^{-1}}{W_F} \frac{\ud s_1}{2\pi i}.
\end{align*}

\noindent To obtain the meromorphic continuation for $\Re s_0$ small, we first shift the contour of integration in $s_1$ to $\Re s_1 = c > \Re s_0$, pick up two poles at $s_1 = s_0-1, s_0$ for $\chi = (\omega\omega_{\Pi})^{-1}$ and obtain by Proposition \ref{DZIntGlobalProp} (4)
\begin{align*}
	\Theta(s_0,F) - \Psi \left( s_0+1, F, (\omega\omega_{\Pi})^{-1} \right) &= \frac{1}{\zeta_{\F}^*} \sum_{\chi \in \widehat{\F^{\times} \R_+ \backslash \A^{\times}}} \int_{(c)} \DBZ{s_1+1}{s_0-s_1}{\chi}{(\chi \omega \omega_{\Pi})^{-1}}{W_F} \frac{\ud s_1}{2\pi i} \\
	&\quad + \widetilde{\Psi}(s_0, F, (\omega\omega_{\Pi})^{-1}) - \Psi(s_0+1,F,(\omega\omega_{\Pi})^{-1}).
\end{align*}
	We assume $\norm[\Re s_0] < 1/2$ and shift the contour back to $\Re s_1 = -1/2$, pick up the pole at $s_1 = s_0$ for $\chi = (\omega\omega_{\Pi})^{-1}$ again and obtain Theorem \ref{thm: MainId} (1) by
\begin{align*}
	\Theta(s_0,F) &= \frac{1}{\zeta_{\F}^*} \sum_{\chi \in \widehat{\F^{\times} \R_+ \backslash \A^{\times}}} \int_{(-1/2)} \DBZ{s_1+1}{s_0-s_1}{\chi}{(\chi \omega \omega_{\Pi})^{-1}}{W_F} \frac{\ud s_1}{2\pi i} \\
	&\quad + \widetilde{\Psi}(s_0, F, (\omega\omega_{\Pi})^{-1}) + \Psi(s_0+1,F,(\omega\omega_{\Pi})^{-1}).
\end{align*}

	\subsection{Second Decomposition}
	
	Recall the Weyl element $w$ and its action
	$$ w = \begin{pmatrix} & 1 & \\ & & 1 \\ 1 & & \end{pmatrix} = w_3 w_{3,1}; \quad w^{-1} \begin{pmatrix} g & \\ & t \end{pmatrix} w = \begin{pmatrix} t & \\ & g \end{pmatrix}, \ \forall g \in \GL_2, t \in \GL_1. $$
	Taking into account the left invariance by $w$ of $F$, we can rewrite $\Theta(s_0,F)$ as
\begin{align} \label{eq: 2ndDecompInitial}
	\Theta(s_0,F) &= \int_{\F \backslash \A} \left( \int_{\F^{\times} \backslash \A^{\times}} \Pi(w).F \begin{pmatrix} 1 & x & \\ & 1 & \\ & & t \end{pmatrix} (\omega\omega_{\Pi})^{-1}(t) \norm[t]_{\A}^{s_0} \ud^{\times}t \right) \psi(-x) \ud x \nonumber \\
	&= \int_{\F \backslash \A} \left( \int_{\F^{\times} \backslash \A^{\times}} \Pi(w).F \begin{pmatrix} t & tx & \\ & t & \\ & & 1 \end{pmatrix} \omega(t) \norm[t]_{\A}^{-s_0} \ud^{\times}t \right) \psi(-x) \ud x. 
\end{align}
	Note that the function $F_1$ on $\GL_2(\A)$ defined by
	$$ F_1(g) := \int_{\F^{\times} \backslash \A^{\times}} \Pi(w).F \begin{pmatrix} tg & \\ & 1 \end{pmatrix} \omega(t) \norm[t]_{\A}^{-s_0} \ud^{\times}t \cdot \norm[\det g]_{\A}^{-\frac{s_0}{2}} $$
is smooth, invariant by $\GL_2(\F)$ on the left, transforms as the character $\omega^{-1}$ on the center of $\GL_2$, and has rapid decay as $g \to \infty$ in any Siegel domain of $\GL_2$ since $\Pi(w).F$ is a cusp form on $\GL_3$. We can apply the automorphic Fourier inversion (see \cite[Theorem 2.3]{Wu22}) and get a normally convergent expansion
\begin{align*}
	F_1(g) &= \sideset{}{_{\substack{\pi \text{ cuspidal} \\ \omega_{\pi}=\omega^{-1}}}} \sum \sideset{}{_{e \in \Bas(\pi)}} \sum \Pairing{F_1}{e} e(g) \\
	&\quad + \sum_{\chi \in \widehat{\R_+ \F^{\times} \backslash \A^{\times}}} \sum_{f \in \Bas(\chi,\omega^{-1}\chi^{-1})} \int_{-\infty}^{\infty} \Pairing{F_1}{\eis(i\tau,f)} \eis(i\tau,f)(g) \frac{\ud \tau}{4\pi} \\
	&\quad + \frac{1}{\Vol([\PGL_2])} \sideset{}{_{\substack{ \chi \in \widehat{\F^{\times} \backslash \A^{\times}} \\ \chi^2 = \omega^{-1} }}} \sum \int_{[\PGL_2]} F_1(x) \overline{\chi(\det x)} \ud x \cdot \chi(\det g).
\end{align*}
	Since $\F\backslash \A$ is compact, we can insert the above expansion of $F_1$ to (\ref{eq: 2ndDecompInitial}) and get
\begin{equation} \label{eq: 2ndDecomp}
	\Theta(s_0,F) = \sideset{}{_{\substack{\pi \text{ cuspidal} \\ \omega_{\pi}=\omega^{-1}}}} \sum \Theta(s_0,F \mid \pi) + \sum_{\chi \in \widehat{\R_+ \F^{\times} \backslash \A^{\times}}} \int_{-\infty}^{\infty} \Theta(s_0,F \mid \chi,\omega^{-1}\chi^{-1}; i\tau) \frac{\ud \tau}{4\pi}, 
\end{equation}
	where we have written
\begin{align} \label{2ndDecompCusp}
	\Theta(s_0,F \mid \pi) &= \sideset{}{_{e \in \Bas(\pi)}} \sum \Pairing{F_1}{e} W_e(\mathbbm{1}) \\
	&= \sideset{}{_{e \in \Bas(\pi)}} \sum \int_{\GL_2(\F) \backslash \GL_2(\A)} \Pi(w).F \begin{pmatrix} g & \\ & 1 \end{pmatrix} \overline{e(g)} \norm[\det g]_{\A}^{-\frac{s_0}{2}} \ud g \cdot W_e(\mathbbm{1}) \nonumber \\
	&= \sideset{}{_{e \in \Bas(\pi)}} \sum \Psi \left( \frac{1-s_0}{2}, \Pi(w).F, e^{\vee} \right) W_e(\mathbbm{1}), \nonumber
\end{align}
\begin{multline} \label{2ndDecompEis}
	\Theta(s_0,F \mid \chi,\omega^{-1}\chi^{-1}; i\tau) = \sum_{f \in \Bas(\chi,\omega^{-1}\chi^{-1})} \Pairing{F_1}{\eis(i\tau,f)} W_{f,i\tau}(\mathbbm{1}) \\
	= \sum_{f \in \Bas(\chi,\omega^{-1}\chi^{-1})} \int_{\GL_2(\F) \backslash \GL_2(\A)} \Pi(w).F \begin{pmatrix} g & \\ & 1 \end{pmatrix} \overline{\eis(i\tau,f)(g)} \norm[\det g]_{\A}^{-\frac{s_0}{2}} \ud g \cdot W_{f,i\tau}(\mathbbm{1}) \\
	= \sum_{f \in \Bas(\chi,\omega^{-1}\chi^{-1})} \Psi \left( \frac{1-s_0}{2}, \Pi(w).F, \eis(-i\tau, f^{\vee}) \right) W_{f,i\tau}(\mathbbm{1}). 
\end{multline}

\noindent The decomposition (\ref{eq: 2ndDecomp}) of $\Theta(s_0,F)$ is done for $\Re s_0 \gg 1$. Moving $s_0$ continually, we come across no poles of any summand/integrand, by the Rankin-Selberg theory and Proposition \ref{prop: RSEisProp}. Hence (\ref{eq: 2ndDecomp}) is valid for all $s_0 \in \C$, proving Theorem \ref{thm: MainId} (2).

	\subsection{Euler Product Factorisation}
	\label{sec: EulerProdD}
	
	The local versions of (\ref{2ndDecompCusp}) and (\ref{2ndDecompEis}) are given by:
	
\begin{equation} \label{eq: 2ndDecompCuspLoc}
	\Theta_v(s_0,W_{F,v} \mid \pi_v) = \sideset{}{_{e \in \Bas(\pi_v)}} \sum \Psi_v \left( \frac{1-s_0}{2}, \Pi(w).W_{F,v}, W_e^{\vee} \right) W_e(\mathbbm{1}),
\end{equation}

\noindent where the dual basis $W_e^{\vee} \in \Whi(\pi_v, \psi_v^{-1})$ is taken in terms of the norm in the Kirillov model;

\begin{equation} \label{eq: 2ndDecompEisLoc}
	\Theta(s_0,W_{F,v} \mid \chi_v,\omega_v^{-1}\chi_v^{-1}; i\tau) = \sum_{f \in \Bas(\chi_v,\omega_v^{-1}\chi_v^{-1})} \Psi_v \left( \frac{1-s_0}{2}, \Pi(w).W_{F,v}, W_{f^{\vee},-i\tau} \right) W_{f,i\tau}(\mathbbm{1}),
\end{equation}

\noindent where the dual basis $f^{\vee} \in \pi(\chi_v^{-1}, \omega_v \chi_v)$ is taken in terms of the norm in the induced model. The norm identification \cite[Proposition 2.13]{Wu14} or \cite[(2.3)]{BFW21+} implies

\begin{multline} \label{eq: WtGL2CuspD}
	\Theta(s_0, F \mid \pi) = \frac{1}{2 \Lambda_{\F}(2) L(1,\pi,\mathrm{Ad})} \cdot \prod_{v \mid \infty} \Theta_v(s_0,W_{F,v} \mid \pi_v) \cdot \prod_{\vp < \infty} \Theta_{\vp}(s_0,W_{F,\vp} \mid \pi_{\vp}) L(1, \pi_{\vp} \times \widetilde{\pi}_{\vp}) \\
	= \frac{L((1-s_0)/2, \Pi \times \widetilde{\pi})}{2 \Lambda_{\F}(2) L(1,\pi,\mathrm{Ad})} \cdot \prod_{v \mid \infty} \Theta_v(s_0,W_{F,v} \mid \pi_v) \cdot \prod_{\vp < \infty} \Theta_{\vp}(s_0,W_{F,\vp} \mid \pi_{\vp}) \frac{L(1, \pi_{\vp} \times \widetilde{\pi}_{\vp})}{L((1-s_0)/2, \Pi_{\vp} \times \widetilde{\pi}_{\vp})},
\end{multline}
	whose specialization to $s_0=0$ is written as
\begin{equation} \label{eq: WtGL2CuspDBis}
	\Theta(F \mid \pi) = \frac{L(1/2, \Pi \times \widetilde{\pi})}{2 \Lambda_{\F}(2) L(1,\pi,\mathrm{Ad})} \cdot \prod_{v \mid \infty} \Theta_v(W_{F,v} \mid \pi_v) \cdot \prod_{\vp < \infty} \Theta_{\vp}(W_{F,\vp} \mid \pi_{\vp}) \frac{L(1, \pi_{\vp} \times \widetilde{\pi}_{\vp})}{L(1/2, \Pi_{\vp} \times \widetilde{\pi}_{\vp})}.
\end{equation}

\noindent Similarly, we have

\begin{multline} \label{eq: WtGL2EisD}
	\Theta(s_0, F \mid \chi, \omega^{-1}\chi^{-1}; i\tau) = \prod_v \Theta_v(s_0,W_{F,v} \mid \chi_v, \omega_v^{-1}\chi_v^{-1}; i\tau) \\
	= \frac{L((1-s_0)/2-i\tau, \Pi \times \chi^{-1}) L((1-s_0)/2+i\tau, \Pi \times \omega\chi)}{\zeta_{\F}(2) \extnorm{L(1+2i\tau,\omega\chi^2)}^2} \cdot \prod_{v \mid \infty} \Theta_v(s_0,W_{F,v} \mid \chi_v, \omega_v^{-1}\chi_v^{-1}; i\tau) \cdot \\
	\prod_{\vp < \infty} \Theta_{\vp}(s_0,W_{F,\vp} \mid \chi_{\vp}, \omega_{\vp}^{-1}\chi_{\vp}^{-1}; i\tau) \frac{\zeta_v(2) \extnorm{L(1+2i\tau, \omega_{\vp}\chi_{\vp}^2)}^2}{L((1-s_0)/2-i\tau, \Pi_{\vp} \times \chi_{\vp}^{-1}) L((1-s_0)/2+i\tau, \Pi_{\vp} \times \omega_{\vp}\chi_{\vp})},
\end{multline}
whose specialization to $s_0=0$ is written as
\begin{multline} \label{eq: WtGL2EisDBis}
	\Theta(F \mid \chi, \omega^{-1}\chi^{-1}; i\tau) = \frac{L(1/2-i\tau, \Pi \times \chi^{-1}) L(1/2+i\tau, \Pi \times \omega\chi)}{\zeta_{\F}(2) \extnorm{L(1+2i\tau,\omega\chi^2)}^2} \cdot \\
	\sideset{}{_{v \mid \infty}} \prod \Theta_v(W_{F,v} \mid \chi_v, \omega_v^{-1}\chi_v^{-1}; i\tau) \cdot \\
	\sideset{}{_{\vp < \infty}} \prod \Theta_{\vp}(W_{F,\vp} \mid \chi_{\vp}, \omega_{\vp}^{-1}\chi_{\vp}^{-1}; i\tau) \frac{\zeta_{\vp}(2) \extnorm{L(1+2i\tau, \omega_{\vp}\chi_{\vp}^2)}^2}{L(1/2-i\tau, \Pi_{\vp} \times \chi_{\vp}^{-1}) L(1/2+i\tau, \Pi_{\vp} \times \omega_{\vp}\chi_{\vp})}.
\end{multline}

\begin{remark}
	Just as \cite[(2.5)]{BFW21+}, for $\pi_v = \pi(\chi_v,\omega_v^{-1}\chi_v^{-1})$ we have the relation
	$$ \Theta_v(W_{F,v} \mid \chi_v, \omega_v^{-1}\chi_v^{-1}; i\tau) = \frac{\zeta_v(1)^2}{\zeta_v(2)} \Theta_v(W_{F,v} \mid \pi_v). $$
\end{remark}

	For the dual side, by the global functional equation in Proposition \ref{DZIntGlobalProp} (3), we have a global equality
	$$ \DBZ{1/2+i\tau}{s_0+1/2-i\tau}{\chi}{(\chi \omega \omega_{\Pi})^{-1}}{W_F} = \DBZ{1/2-i\tau}{1/2+i\tau-s_0}{\chi^{-1}}{\chi \omega \omega_{\Pi}}{\widetilde{\Pi}(w_{3,1}).\widetilde{W_F}}. $$
	The following decomposition follows readily from the computation at the unramified places:
\begin{multline} \label{eq: DualWtGL1D}
	\DBZ{1/2-i\tau}{1/2+i\tau-s_0}{\chi^{-1}}{\chi \omega \omega_{\Pi}}{\widetilde{\Pi}(w_{3,1}).\widetilde{W_F}} = L(1/2-i\tau, \widetilde{\Pi} \times \chi^{-1}) L(1/2+i\tau-s_0, \chi\omega\omega_{\Pi}) \cdot \\
	\prod_{v \mid \infty} \DBZ{1/2-i\tau}{1/2+i\tau-s_0}{\chi_v^{-1}}{\chi_v \omega_v \omega_{\Pi,v}}{\widetilde{\Pi}_v(w_{3,1}).\widetilde{W_{F,v}}} \cdot \\
	\prod_{\vp < \infty} \DBZ{1/2-i\tau}{1/2+i\tau-s_0}{\chi_{\vp}^{-1}}{\chi_{\vp} \omega_{\vp} \omega_{\Pi,\vp}}{\widetilde{\Pi}_{\vp}(w_{3,1}).\widetilde{W_{F,\vp}}} \frac{1}{L(1/2-i\tau, \widetilde{\Pi}_{\vp} \times \chi_{\vp}^{-1}) L(1/2+i\tau-s_0, \chi_{\vp}\omega_{\vp}\omega_{\Pi,\vp})},
\end{multline}
whose specialization to $s_0=0$ is written as
\begin{multline} \label{eq: DualWtGL1DBis}
	\DBZ{1/2-i\tau}{1/2+i\tau}{\chi^{-1}}{\chi \omega \omega_{\Pi}}{\widetilde{\Pi}(w_{3,1}).\widetilde{W_F}} = L(1/2-i\tau, \widetilde{\Pi} \times \chi^{-1}) L(1/2+i\tau, \chi\omega\omega_{\Pi}) \cdot \\
	\prod_{v \mid \infty} \DBZ{1/2-i\tau}{1/2+i\tau}{\chi_v^{-1}}{\chi_v \omega_v \omega_{\Pi,v}}{\widetilde{\Pi}_v(w_{3,1}).\widetilde{W_{F,v}}} \cdot \\
	\prod_{\vp < \infty} \DBZ{1/2-i\tau}{1/2+i\tau}{\chi_{\vp}^{-1}}{\chi_{\vp} \omega_{\vp} \omega_{\Pi,\vp}}{\widetilde{\Pi}_{\vp}(w_{3,1}).\widetilde{W_{F,\vp}}} \frac{1}{L(1/2-i\tau, \widetilde{\Pi}_{\vp} \times \chi_{\vp}^{-1}) L(1/2+i\tau, \chi_{\vp}\omega_{\vp}\omega_{\Pi,\vp})}.
\end{multline}

	For the convenience of the follow-up papers, we introduce the \emph{weight functions} as follows (below $\vp$ denotes a finite place while $v$ denotes a general place):
\begin{equation} \label{eq: WtFDef}
	h_v(\pi_v) := \Theta_v(W_{F,v} \mid \pi_v), \quad H_{\vp}(\pi_{\vp}) := h_{\vp}(\pi_{\vp}) \frac{L(1, \pi_{\vp} \times \widetilde{\pi}_{\vp})}{L(1/2, \Pi_{\vp} \times \widetilde{\pi}_{\vp})};
\end{equation}
\begin{align} \label{eq: DWtFDef}
	\widetilde{h}_v(\chi_v) &:= \DBZ{1/2}{1/2}{\chi_v^{-1}}{\chi_v \omega_v \omega_{\Pi,v}}{\widetilde{\Pi}_v(w_{3,1}).\widetilde{W_{F,v}}}, \\
	\widetilde{H}_{\vp}(\chi_{\vp}) &:= \widetilde{h}_{\vp}(\chi_{\vp}) L(1/2, \widetilde{\Pi}_{\vp} \times \chi_{\vp}^{-1})^{-1} L(1/2, \chi_{\vp}\omega_{\vp}\omega_{\Pi,\vp})^{-1}. \nonumber
\end{align}
	We also introduce the abbreviation $\pi(\chi_v,s) := \pi(\chi_v \norm_v^s, \omega_v^{-1} \chi_v^{-1} \norm_v^{-s})$ for simplicity of notation. Then the main equality in Theorem \ref{thm: Main1} can be rewritten as (note that $\zeta_v(1)=1$ at $v \mid \infty$)
\begin{multline} \label{eq: Main1Bis}
	\sum_{\pi: \omega_{\pi}=\omega^{-1}} \frac{L(1/2, \Pi \times \widetilde{\pi})}{2 \Lambda_{\F}(2) L(1,\pi,\mathrm{Ad})} \cdot \prod_{v \mid \infty} h_v(\pi_v) \cdot \prod_{\vp < \infty} H_{\vp}(\pi_{\vp}) + \\ 
	\sum_{\chi \in \widehat{\R_+ \F^{\times} \backslash \A^{\times}}} \int_{-\infty}^{\infty} \frac{L(1/2-i\tau, \Pi \times \chi^{-1}) L(1/2+i\tau, \Pi \times \omega\chi)}{2 \Lambda_{\F}(2) \extnorm{L(1+2i\tau,\omega\chi^2)}^2}
	\cdot \prod_{v \mid \infty} h_v(\pi(\chi_v,i\tau)) \cdot \prod_{\vp < \infty} H_{\vp}(\pi(\chi_{\vp},i\tau)) \frac{\ud \tau}{2\pi} \\
	= \frac{1}{\zeta_{\F}^*} \sum_{\chi \in \widehat{\R_+ \F^{\times} \backslash \A^{\times}}} \int_{-\infty}^{\infty} L(1/2-i\tau, \widetilde{\Pi} \times \chi^{-1}) L(1/2+i\tau, \chi\omega\omega_{\Pi}) \cdot \prod_{v \mid \infty} \widetilde{h}_v(\chi_v \norm_v^{i\tau}) \cdot \prod_{\vp < \infty} \widetilde{H}_{\vp}(\chi_{\vp} \norm_{\vp}^{i\tau}) \frac{\ud \tau}{2\pi} + \\
	\frac{1}{\zeta_{\F}^*} \sum_{\pm} \Res_{s_1 = \pm \frac{1}{2}} L(1/2-s_1, \widetilde{\Pi} \times \omega\omega_{\Pi}) L(1/2+s_1, \id) \cdot \prod_{v \mid \infty} \widetilde{h}_v(\omega_v^{-1}\omega_{\Pi,v}^{-1} \norm_v^{s_1}) \cdot \prod_{\vp < \infty} \widetilde{H}_{\vp}(\omega_{\vp}^{-1}\omega_{\Pi,\vp}^{-1} \norm_{\vp}^{s_1}).
\end{multline}
	Note that in the above the infinite products over $\vp < \infty$ are finite for any chosen $F$.

\section{Local Weight Transforms}

	\subsection{Miller--Schmid Type Theory: Non-Archimedean Case}
	\label{sec: MSTheory}

	We begin with a Paley--Wiener theory.
	
\begin{definition} \label{def: SSchNA}
	(1) A \emph{finite function} on a locally compact group is a continuous function whose translates span a finite dimensional vector space.
	
\noindent (2) We write $\SSch(\F)$ for the space spanned by $\eta \cdot \Sch(\F)$ as $\eta$ runs through finite functions on $\F^{\times}$.
\end{definition}
\begin{remark}
	The space of finite functions is spanned by functions of the shape $\chi \cdot v_{\F}^k$, where $\chi$ is a quasi-character of $\F^{\times}$ and $k \in \Z_{\geq 0}$. It is clear that $\SSch(\F) \subset \Cont^{\infty}(\F^{\times})$, since $\chi \cdot v_{\F}^k \cdot \id_{\vO_{\F}} \in \Cont^{\infty}(\F^{\times})$.
\end{remark}

\begin{definition} \label{eq: SMellinNA}
	Write $\SMel(\F) = \Cont_c(\widehat{\vO_{\F}^{\times}}, \C(X))$, where $\C(X)$ is the fractional field of $\C[X]$.
\end{definition}

\begin{proposition} \label{prop: PaleyWienerNA}
	The Mellin transform for $f \in \Cont^{\infty}(\F^{\times})$ defined for $\xi \in \widehat{\vO_{\F}^{\times}}$ and $s \in \C$ by
	$$ \Mellin{f}(\xi,s) := \int_{\F^{\times}} f(t) \xi(t) \norm[t]^s \ud^{\times} t $$
whenever the integral is absolutely convergent, induces a bijection between $\SSch(\F)$ and $\SMel(\F)$ up to the change of variables $X := q^{-s}$.
\end{proposition}
\begin{proof}
	Clearly the Mellin transform is injective on $\SSch(\F)$. It suffices to identify its image as $\SMel(\F)$. Note that $\SSch(\F)$ is a smooth $\vO_{\F}^{\times}$-module. For every $\xi \in \widehat{\vO_{\F}^{\times}}$ we introduce
	$$ \Cont_c^{\infty}(\F^{\times}; \xi) := \left\{ f \in \Cont_c^{\infty}(\F^{\times}) \ \middle| \ f(tx) = \xi(t)f(x), \ \forall t \in \vO_{\F}^{\times} \right\}, $$
	$$ \SSch(\F; \xi) := \left\{ f \in \SSch(\F) \ \middle| \ f(tx) = \xi(t)f(x), \ \forall t \in \vO_{\F}^{\times} \right\}, $$
	then a (smooth/profinite) version of Peter--Weyl's theorem for $\vO_{\F}^{\times}$ implies
	$$ \SSch(\F) = \sideset{}{_{\xi}} \bigoplus \SSch(\F; \xi). $$
	It suffices to show that the Mellin transform maps $\SSch(\F; \xi)$ onto $\C(X)$ viewed as the subspace of functions in $\SMel(\F)$ supported in the singleton $\{ \xi^{-1} \}$. Note that the Mellin transform identifies $\Cont_c^{\infty}(\F^{\times}; \xi)$ with $\C[X,X^{-1}]$, and takes convolution on $\F^{\times}$ to multiplication. Introduce
	$$ \Sch^*(\F; \xi) := \left\{ f \in \xi \cdot\Sch(\F) \ \middle| \ f(tx) = \xi(t)f(x), \ \forall t \in \vO_{\F}^{\times} \right\} = \Cont_c^{\infty}(\F^{\times};\xi) \bigoplus \C \xi \id_{\vO_{\F}}, $$
	$$ [X]_0 := 1, \quad [X]_k := (X+1)\cdots(X+k), \ \forall k \in \Z_{\geq 1}. $$
	Note that $\Cont_c^{\infty}(\F^{\times};\xi)$ is a commutative algebra with identity $\xi \id_{\vO_{\F}^{\times}}$, and we have the equalities
	$$ \left( \xi \id_{\varpi_{\F} \vO_{\F}^{\times}} - \xi \id_{\vO_{\F}^{\times}} \right) * \left( \xi \id_{\vO_{\F}^{\times}} \right) = \xi \id_{\vO_{\F}^{\times}} \quad \Rightarrow \quad \Sch^*(\F; \xi) = \Cont_c^{\infty}(\F^{\times};\xi) *  \left( \xi \id_{\vO_{\F}} \right); $$
	$$ \norm^{s_0} (\phi * f) = \left( \norm^{s_0} \phi \right) * \left( \norm^{s_0} f \right), \quad v_{\F} (\phi * f) = (v_{\F} \phi)*f + \phi * (v_{\F} f), \quad \forall \phi \in \Cont_c^{\infty}(\F^{\times}), f \in \SSch(\F). $$
	We can therefore write
\begin{multline*} 
	\SSch(\F;\xi) = \sum_{s_0 \in \C, k \in \Z_{\geq 0}} v_{\F}^k \norm^{s_0} \Sch^*(\F;\xi) = \sum_{s_0 \in \C, k \in \Z_{\geq 0}} [v_{\F}]_k \norm^{s_0} \Cont_c^{\infty}(\F^{\times};\xi)*(\xi \id_{\vO_{\F}}) \\
	= \sum_{s_0 \in \C, k \in \Z_{\geq 0}} \Cont_c^{\infty}(\F^{\times};\xi)* \left( [v_{\F}]_k \norm^{s_0} \xi \id_{\vO_{\F}} \right).
\end{multline*}
	Writing $\beta = q^{-s_0}$, a simple computation shows (for $\Re(s) \gg 1$)
	$$ \Mellin{[v_{\F}]_k \norm^{s_0} \xi \id_{\vO_{\F}}}(\xi^{-1},s) = \left[ X \frac{\ud}{\ud X} \right]_k \left( \frac{1}{1-\beta X} \right) = \frac{k!}{(1-\beta X)^k}. $$
	Therefore the image under the Mellin transform of $\SSch(\F;\xi)$ is identified with
	$$ \sum_{\beta \in \C^*, k \in \Z_{\geq 0}} \C[X,X^{-1}] \frac{k!}{(1-\beta X)^k} = \C(X) $$
	by the partial fractional expansion for the \textrm{PID} $\C[X,X^{-1}]$.
\end{proof}
	
	Consider a generic admissible irreducible $\gp{G}_r(\F)$-representation $\pi$. Note that $\VorH(\pi)$ given in \eqref{eq: BasicVHS} is a $\Cont_c^{\infty}(\F^{\times})$-submodule of $\SSch(\F)$ (see \cite[Theorem (2.7)]{JPS83} and \cite[Proposition (2.2)]{JPS79}). The \emph{Voronoi--Hankel transform} in Definition \ref{def: VoronoiGLn}
\begin{equation} \label{eq: VorHInitial}
	\VorH_{\pi, \psi}: \VorH(\pi) \to \norm \VorH(\widetilde{\pi})
\end{equation}
	is an isomorphism determined by the local functional equation
\begin{equation} \label{eq: VorHLocFE}
	\Mellin{\VorH_{\pi, \psi}(h)}(\xi^{-1},-s) = \gamma(s, \pi \times \xi, \psi) \cdot \Mellin{h}(\xi,s), \quad \forall \ h \in \VorH(\pi), \xi \in \widehat{\vO_{\F}^{\times}}.
\end{equation}
	
\noindent We propose a Miller--Schmid type extension of $\VorH_{\pi,\psi}$ to functions with simple singularities at the infinity.
	
\begin{proposition} \label{prop: MillerSchmidtNA}
	(1) For any $h \in \SSch(\F)$ there is a unique $\MS_{\pi,\psi} \circ \Inv(h) \in \SSch(\F)$ such that
	$$ \Mellin{\MS_{\pi, \psi}(\Inv(h))}(\xi^{-1},-s) = \gamma(s, \pi \times \xi, \psi) \cdot \Mellin{\Inv(h)}(\xi,s), \quad \forall \ \xi \in \widehat{\vO_{\F}^{\times}}. $$
The Mellin transforms on both sides are absolutely convergent for $\Re(s) \ll -1$. 

\noindent (2) The two transforms $\VorH_{\pi,\psi}$ and $\MS_{\pi,\psi}$ coincide on $\VorH(\pi) \cap \Inv(\SSch(\F)) = \Cont_c^{\infty}(\F^{\times})$.
\end{proposition}
\begin{proof}
	(1) The existence and uniqueness of $\MS_{\pi,\psi} \circ \Inv(h) \in \SSch(\F)$ is a direct consequence of Proposition \ref{prop: PaleyWienerNA}: the right hand side is absolutely convergent for $\Re(s) \ll -1$ and defines a rational function in $X$. 

\noindent (2) It suffices to note that for $h \in \Cont_c^{\infty}(\F^{\times})$ the Mellin transform $\Mellin{h}(\xi,s)$ is absolutely convergent for all $s \in \C$.
\end{proof}

\begin{remark} \label{rmk: MSArch}
	For archimedean $\F \in \{ \R, \C \}$, we have analogues of $\SMel(\F)$ (see Definition \ref{eq: SMellinNA}). The precise definitions and the corresponding Paley--Wiener theories can be found in \cite[Definition 6.21 \& Corollary 6.39]{MS04} for the real case and \cite[Lemma 2.8]{Qi20} for the complex case. The analogues of Proposition \ref{prop: MillerSchmidtNA} can be stated in the same way, or stated in some different but equivalent way in \cite[Lemma 6.19]{MS04} for the real case and in \cite[Theorem 3.12]{Qi20} for the complex case.
\end{remark}

	\subsection{Extended Voronoi--Hankel Transforms}
	
	We first recall the Godement--Jacquet theory in the local setting. We follow \cite{GoJ11} with a slight modification. Namely, we present the theory with $\beta^{\iota}(g) = \beta(g^{\iota})$ instead of $\check{\beta}(g) = \beta(g^{-1})$.
	
	Let $\F$ be a local field. Let $\pi$ be an irreducible \emph{smooth} and \emph{generic} representation of $\GL_n(\F)$, with smooth dual representation $\widetilde{\pi}$ (also called \emph{contra-gredient representation}). Let $\Pairing{\cdot}{\cdot}$ be the natural pairing on $V_{\pi} \times V_{\widetilde{\pi}}$. A \emph{matrix coefficient} $\beta$ of $\pi$ is a function on $\GL_n(\F)$ so that for some $v \in V_{\pi}$ and $\widetilde{v} \in V_{\widetilde{\pi}}$ we have $\beta(g) = \Pairing{\pi(g).v}{\widetilde{v}}$. The set of matrix coefficients is denoted by $C(\pi)$.
\begin{theorem} \label{thm: GJLoc}
	For any Schwartz--Bruhat function $\Phi \in \Sch(\Mat_n(\F))$ and any $\beta \in C(\pi)$, the Godement--Jacquet zeta function is defined by
	$$ \Zeta(s, \Phi, \beta) := \int_{\GL_n(\F)} \Phi(g) \beta(g) \norm[\det g]^{s+\frac{n-1}{2}} \ud g. $$
\begin{itemize}
	\item[(1)] The above integral defining $\Zeta(s, \Phi, \beta)$ is absolutely convergent in $\Re s > s_0$ for some $s_0 \in \R$.
	\item[(2)] The function $s \mapsto \Zeta(s, \Phi, \beta)$ has a meromorphic continuation to $s \in \C$, and the quotient function $s \mapsto \Zeta(s, \Phi, \beta)/L(s,\pi)$ is entire. If $\F$ is archimedean, then the function $s \mapsto \Zeta(s, \Phi, \beta)$ is rapidly decreasing in any region $a \leq \Re(s) \leq b$ away from the possible poles.
	\item[(3)] We have the local functional equation
	$$ \frac{\Zeta(1-s, \widehat{\Phi}, \beta^{\iota})}{L(1-s, \widetilde{\pi})} = \varepsilon(s, \pi, \psi) \frac{\Zeta(s, \Phi, \beta)}{L(s, \pi)}, \quad \text{or} \quad \Zeta(1-s, \widehat{\Phi}, \beta^{\iota}) = \gamma(s, \pi, \psi) \Zeta(s, \Phi, \beta). $$
\end{itemize}
\end{theorem}
\begin{proof}
	See \cite[Theorem 15.4.4]{GoJ11} for the non-archimedean case. As for the archimedean case, \cite[Theorem 15.9.1]{GoJ11} is a weaker version where $\Sch(\Mat_n(\F))$ is replaced by the subspace of \emph{standard} Schwartz functions $\Sch_0(\Mat_n(\F))$, and $C(\pi)$ is replaced with its $\gp{K}_n$-finite subset. The proof of \cite[Proposition 4.4]{J09} shows the equivalence and implies our version: any relevant $\pi$ is a subrepresentation of a principal one. Note that the denseness of $\Sch_0(\Mat_n(\F))$ can be found in \cite[Exercise \Rmnum{3}.6 \& \Rmnum{3}.7]{HT92}.
\end{proof}
\begin{remark} \label{rmk: FourTransConv}
	We take the convention of Fourier transforms in \cite[(0.2)]{JPS79} as
	$$ \invOFour(\Phi)(X) = \widehat{\Phi}(X) := \int_{\Mat_n(\F)} \Phi(Y) \psi \left( \Tr (X Y^T) \right) \ud Y, $$
	since it is with this convention that the consistency of the Godement--Jacquet theory and the $\GL_n \times \GL_1$ Rankin--Selberg theory has been checked, i.e., they give the same gamma factors.
\end{remark}

\begin{proposition} \label{prop: GJZAbsCTemp}
	If $\pi = \Pi^{\infty}$ for a unitary and \emph{tempered} representation $\Pi$, then the integral $\Zeta(s, \Phi, \beta)$ is absolutely convergent in $\Re(s) > 0$ for any $\beta \in C(\pi)$.
\end{proposition}
\begin{proof}
	This is essentially an easy consequence of the decay of matrix coefficient
\begin{equation} \label{eq: DecayMCSmooth}
	\norm[\beta(g)] \ll \Xi_n(g)
\end{equation}
 	where $\Xi_n$ is the Harish-Chandra's Xi-function for $\GL_n(\F)$. The estimation \eqref{eq: DecayMCSmooth} can be found in \cite[Theorem 1.2]{Sun09} (see also \cite[Theorem 1.1]{He09}), extending the $\gp{K}$-finite case \cite[Theorem 2]{CHH88}. Note that 
	$$ \Xi_n(g) = \Pairing{\Pi_0(g).e_0}{e_0} $$
	is the matrix coefficient of a unitary spherical vector $e_0$ in the representation $\Pi_0 = \Ind_{\gp{B}_n(\F)}^{\GL_n(\F)} \id$ parabolically induced from the trivial character of the Borel subgroup $\gp{B}_n(\F)$.  To conclude, an iterated application of \cite[(15.7.13)]{GoJ11} together with Lemma \ref{lem: WeilSchwartzEnv} shows for any $\sigma > 0$
	$$ \int_{\GL_n(\F)} \extnorm{\Phi(g)} \Xi_n(g) \norm[\det g]^{\sigma+\frac{n-1}{2}} \ud g \ll \int_{(\F^{\times})^{\oplus n}} \phi(t_1,\dots,t_n) \sideset{}{_{j=1}^n} \prod \norm[t_j]^{\sigma} \ud^{\times} t_j < \infty $$
	for some positive Schwartz-Bruhat function $\phi \in \Sch(\F^n)$.
\end{proof}

\begin{proposition} \label{prop: BdWhiTemp}
	Let $\pi = \Pi^{\infty}$ for a unitary and \emph{tempered} representation $\Pi$. For any $W \in \Whi(\pi, \psi)$, there is a Schwartz--Bruhat function $0< \phi \in \Sch(\F^n)$ so that
	$$ \extnorm{W \left( n \begin{pmatrix} a_1 & & \\ & \ddots & \\ & & a_n \end{pmatrix} \kappa \right)} \leq \prod_{j=1}^{n-1} \norm[t_j]^{\frac{j(n-j)}{2}} \left( 1+ (\log \norm[t_j])^2 \right)^d \cdot \phi(t_1,\dots,t_{n-1}) $$
for any $a_j = t_1 \cdots t_j \in \F^{\times}$, $n \in \gp{N}_n(\F)$, $\kappa \in \gp{K}_n$, and some $d \in \Z_{\geq 0}$ depending only on $\pi$.
\end{proposition}
\begin{proof}
	In the archimedean case, the desired bound is a consequence of \cite[Proposition 3.5]{J09} and Lemma \ref{lem: WeilSchwartzEnv}. Alternatively we may derive it from \cite[Theorem 2]{J04} and \cite[Proposition 3.5]{J09}. In the non-archimedean case, we may take $a_n=1$ by the unitarity of the central character and $n=\kappa=\id$ by the $\gp{K}_n$-finiteness of $W$. By \cite[Proposition (2.2)]{JPS79} we have
	$$ W \left( \begin{pmatrix} a_1 & & \\ & \ddots & \\ & & a_n \end{pmatrix} \right) = \omega_{\pi}(a_n) \sum_{\lambda} \lambda(t_1,\dots,t_{n-1}) \phi_{\lambda}(t_1,\dots,t_{n-1}) $$
for a finite number of \emph{finite functions} $\lambda$ and some $\phi_{\lambda} \in \Sch(\F^{n-1})$. We may assume that each $\lambda$ is \emph{decomposable} in the sense that $\lambda(t_1,\dots,t_{n-1}) = \lambda_1(t_1) \cdots \lambda_{n-1}(t_{n-1})$ for finite functions $\lambda_j$ on $\F^{\times}$. The bound \cite[Proposition (2.5)]{JS83} implies that the exponent of $\lambda_j$ is $\leq j(n-j)/2$, namely 
	$$ \norm[\lambda_j(t_j)] \ll \norm[t_j]^{\frac{j(n-j)}{2}-\epsilon}, \quad \forall \epsilon > 0. $$
	We conclude since $\lambda_j$ is a sum of products of a quasi-character and a power of the additive valuation.
\end{proof}

\begin{corollary} \label{cor: RSAbsCTemp}
	Let $\pi = \Pi^{\infty}$ for a unitary and \emph{tempered} representation $\Pi$. For any $h \in \VorH(\pi)$ and $\xi \in \widehat{\F^{\times}}$ the Mellin transform $\Mellin{h}(\xi,s)$ is absolutely convergent in $\Re(s) > 0$.
\end{corollary}

\begin{remark} \label{rmk: NonTempBd}
	It should be possible to extend the estimation \eqref{eq: DecayMCSmooth} to the non-tempered case, based on the work of Oh \cite{Oh02}, and a strategy in the rank one case of Venkatesh \cite[Lemma 9.1]{Ve10} which he attributes to Shalom \cite{Shal00}. Similarly, the bound of the Whittaker functions in Proposition \ref{prop: BdWhiTemp} should also be extensible. We believe that the absolute convergence region should be $\Re(s) > \RamCst$ in both Proposition \ref{prop: GJZAbsCTemp} and Corollary \ref{cor: RSAbsCTemp} for a unitary and $\RamCst$-tempered representation.
\end{remark}

\begin{proof}[Proof of Theorem \ref{thm: ExtVorH}]
	First consider the case of a general irreducible smooth and generic $\pi$. Take any $\Phi \in \Cont_c^{\infty}(\GL_n(\F))$, $\beta \in C(\pi)$ and $h \in \SSch(\F)$. Then we have by the Plancherel for Mellin transform, and for $\sigma \ll -1$
\begin{align*}
	&\quad \int_{\GL_n(\F)} \widehat{\Phi}(g) \beta^{\iota}(g) \norm[\det g]^{\frac{n+1}{2}} \cdot \Inv(h)(\det g) \ud g \\
	&= \sideset{}{_{\xi \in \widehat{\F^1}}} \sum \int_{\sigma+i\R(\F)} \left( \int_{\GL_n(\F)} \widehat{\Phi}(g) \beta^{\iota}(g) \xi(\det g)^{-1} \norm[\det g]^{1-s+\frac{n-1}{2}} \ud g \right) \left( \int_{\F^{\times}} h(t^{-1}) \xi(t) \norm[t]^s \ud^{\times} t \right) \ud_{\F} s \\
	&= \sideset{}{_{\xi \in \widehat{\F^1}}} \sum \int_{\sigma+i\R(\F)} \Zeta(1-s, \widehat{\Phi}, (\beta \otimes \xi)^{\iota}) \cdot \Mellin{\Inv(h)}(\xi,s) \ud_{\F} s \\
	&= \sideset{}{_{\xi \in \widehat{\F^1}}} \sum \int_{\sigma+i\R(\F)} \Zeta(s, \Phi, \beta \otimes \xi) \cdot \Mellin{\MS_{\pi,\psi}(\Inv(h))}(\xi^{-1},-s) \ud_{\F} s,
\end{align*}
	where we have applied the local functional equations in Proposition \ref{prop: MillerSchmidtNA} and Theorem \ref{thm: GJLoc}. Since $\Phi \in \Cont_c^{\infty}(\GL_n(\F))$ the zeta integral $\Zeta(s, \Phi, \beta \otimes \xi)$ is absolutely convergent for any $s \in \C$. Applying the Plancherel for Mellin transform again we get
\begin{equation} \label{eq: VH-MS-Comp}
	\int_{\GL_n(\F)} \widehat{\Phi}(g) \beta^{\iota}(g) \norm[\det g]^{\frac{n+1}{2}} \cdot \Inv(h)(\det g) \ud g = \int_{\GL_n(\F)} \Phi(g) \beta(g) \norm[\det g]^{\frac{n-1}{2}} \cdot \MS_{\pi,\psi}(\Inv(h))(\det g) \ud g. 
\end{equation}
	To see that \eqref{eq: VH-MS-Comp} characterizes $\MS_{\pi,\psi}(\Inv(h))$, take any $t_0 \in \F^{\times}$ and $g_0 \in \GL_n(\F)$ with $t_0 = \det g_0$. There is always a $\beta \in C(\pi)$ such that $\beta(g_0) \neq 0$, since $C(\pi)$ is stable by translation by $\GL_n(\F)$. Then $g \mapsto \Phi(g) \beta(g) \norm[\det g]^{\frac{n-1}{2}}$ can be any smooth function of compact support in any neighborhood of $g_0$ on which $\beta \neq 0$. Therefore \eqref{eq: VH-MS-Comp} characterizes $\MS_{\pi,\psi}(\Inv(h))$ in small neighborhoods of any $t_0 \in \F^{\times}$, since $\GL_n(\F) \simeq \F^{\times} \times \SL_n(\F)$ as topological spaces, hence it uniquely determines $\MS_{\pi,\psi}(\Inv(h))$.
	
	Now consider the special case of unitary and tempered $\pi = \Pi^{\infty}$. Replacing $\Inv(h)$ with any $h \in \VorH(\pi)$, $\Phi \in \Cont_c^{\infty}(\GL_n(\F))$ with any $\Phi \in \Sch(\Mat_n(\F))$, and re-taking the above argument, we see that
\begin{align*}
	&\quad \int_{\GL_n(\F)} \widehat{\Phi}(g) \beta^{\iota}(g) \norm[\det g]^{\frac{n+1}{2}} \cdot h(\det g) \ud g \\
	&= \sideset{}{_{\xi \in \widehat{\F^1}}} \sum \int_{\sigma+i\R(\F)} \Zeta(1-s, \widehat{\Phi}, (\beta \otimes \xi)^{\iota}) \cdot \Mellin{h}(\xi,s) \ud_{\F} s \\
	&= \sideset{}{_{\xi \in \widehat{\F^1}}} \sum \int_{\sigma+i\R(\F)} \Zeta(s, \Phi, \beta \otimes \xi) \cdot \Mellin{\VorH_{\pi,\psi}(h)}(\xi^{-1},-s) \ud_{\F} s
\end{align*}
	holds for $0 < \sigma < 1$ by Proposition \ref{prop: GJZAbsCTemp} and Corollary \ref{cor: RSAbsCTemp}. Therefore we obtain
\begin{equation} \label{eq: ExtVHComp}
	\int_{\GL_n(\F)} \widehat{\Phi}(g) \beta^{\iota}(g) \norm[\det g]^{\frac{n+1}{2}} \cdot h(\det g) \ud g = \int_{\GL_n(\F)} \Phi(g) \beta(g) \norm[\det g]^{\frac{n-1}{2}} \cdot \VorH_{\pi,\psi}(h)(\det g) \ud g
\end{equation}
	and conclude the proof.
\end{proof}

	\subsection{Local Weight Transforms: Tempered Case}

	Let $\Pi$ be $\RamCst$-tempered for some $\RamCst < 1/2$. The goal of this section is to give a formula of
\begin{equation} 
	\widetilde{h}(\chi) := \DBZ{1/2}{1/2}{\chi^{-1}}{\chi \omega \omega_{\Pi}}{\widetilde{\Pi}(w_{3,1}).\widetilde{W_F}} 
\label{SimpleDualWt}
\end{equation}
in terms of the local component of $\Theta(F \mid \pi)$, namely
\begin{equation} 
	h(\pi) := \Theta(W_F \mid \pi) = \sideset{}{_{e \in \Bas(\pi)}} \sum \Psi \left( \tfrac{1}{2}, \Pi(w).W_F, W_{e}^{\vee} \right) W_e(\mathbbm{1}). 
\label{SimpleWt}
\end{equation}

	If we denote $W := \Pi(w).W_F$, then the weight functions $h(\pi)$ given by (\ref{SimpleWt}) depends on the restriction of $W$ to the usual embedding of $\GL_2(\F)$ in $\GL_3(\F)$, i.e., the associated function in the Kirillov model. So the question of weight transformation formula is translated into the question: 
\begin{center}
	How does the Kirillov model determine the Whittaker model?
\end{center}
	If the ``determination'' process is required to go via the local functional equations, then this question is intimately related to Jacquet's conjecture on the local converse theorems, which is now a theorem. In the present simple case of $\GL_3$, we only need the \emph{height} theory associated with the Bruhat decomposition of $\GL_3$ due to Chen \cite{Ch06}. We shall not recall this theory, but only present the ``shortest parth'' we have found with this theory as the following matrix equation. It is responsible for the relevant weight transformation formula/process:
\begin{equation} \label{EssMatId}
	\begin{pmatrix} t_2t_1^{-1} & & \\ t_1^{-1} & 1 & \\ & & 1 \end{pmatrix} w_3 = \begin{pmatrix} 1 & t_2 & \\ & 1 & \\ & & 1 \end{pmatrix} \begin{pmatrix} t_2 & & \\ & 1 & \\ & & 1 \end{pmatrix} w_{3,1} \begin{pmatrix} & -1 & t_1 \\ 1 & & \\ & & t_1 \end{pmatrix}^{\iota}.
\end{equation}
	By the absolute convergence established in Proposition \ref{prop: DZIntLocProp}, we have
\begin{multline} \label{InvRed1}
	\widetilde{h}(\chi) = \int_{(\F^{\times})^2} \widetilde{\Pi}(w_{3,1}) \widetilde{W_F} \begin{pmatrix} t_1 & & \\ t_2 & 1 & \\ & & 1 \end{pmatrix} \omega\omega_{\Pi}(t_2) \chi \left( \frac{t_2}{t_1} \right) \extnorm{\frac{t_2}{t_1}}^{\frac{1}{2}} \ud^{\times}t_1 \ud^{\times}t_2 \\
	= \int_{(\F^{\times})^2} \widetilde{\Pi}(w_3) \widetilde{W} \begin{pmatrix} t_2t_1^{-1} & & \\ t_1^{-1} & 1 & \\ & & 1 \end{pmatrix} \omega\omega_{\Pi}(t_1)^{-1} \chi^{-1}(t_2) \norm[t_2]^{-\frac{1}{2}} \ud^{\times}t_1 \ud^{\times}t_2.
\end{multline}
	The equation of matrices (\ref{EssMatId}) implies the equation
\begin{equation} \label{EssWhitId}
	\widetilde{\Pi}(w_3) \widetilde{W} \begin{pmatrix} t_2t_1^{-1} & & \\ t_1^{-1} & 1 & \\ & & 1 \end{pmatrix} = \psi(-t_2) \cdot \widetilde{\Pi\begin{pmatrix} & -1 & t_1 \\ 1 & & \\ & & t_1 \end{pmatrix}W} \left( \begin{pmatrix} t_2 & & \\ & 1 & \\ & & 1 \end{pmatrix} w_{3,1} \right). 
\end{equation}
	Hence it is reasonable to introduce
\begin{equation} \label{InvDis1}
	h^*(t_1,t_2) = \widetilde{\Pi\begin{pmatrix} & -1 & t_1 \\ 1 & & \\ & & t_1 \end{pmatrix}W} \left( \begin{pmatrix} t_2 & & \\ & 1 & \\ & & 1 \end{pmatrix} w_{3,1} \right),
\end{equation}
	whose relation to (\ref{InvRed1}) is given by
\begin{equation} \label{InvDisRel1}
	\widetilde{h}(\chi) = \int_{(\F^{\times})^2} h^*(t_1,t_2) \omega\omega_{\Pi}(t_1)^{-1} \psi(-t_2) \chi^{-1}(t_2) \norm[t_2]^{-\frac{1}{2}} \ud^{\times}t_1 \ud^{\times}t_2.
\end{equation}

\noindent If we introduce
\begin{equation} \label{eq: InvDis2}
	h(t_1,y) := \int_{\F} \Pi\begin{pmatrix} & -1 & t_1 \\ 1 & & \\ & & t_1 \end{pmatrix}W \begin{pmatrix} y & & \\ x & 1 & \\ & & 1 \end{pmatrix} \ud x,
\end{equation}
	then by Definition \ref{def: VoronoiGLn} we get the relation 
\begin{equation} \label{eq: InvDisRel2}
	h^*(t_1,t_2) = \Vor_{\Pi}(h(t_1,\cdot))(t_2),
\end{equation} 

\noindent We may summarize the equations (\ref{InvDisRel1}) and (\ref{eq: InvDisRel2}) in a single formula as
\begin{equation} \label{eq: DWtORel}
	\widetilde{h}(\chi) = \int_{(\F^{\times})^2} \omega\omega_{\Pi}(t_1)^{-1} \psi(-t_2) \chi^{-1}(t_2) \norm[t_2]^{-\frac{1}{2}} \Vor_{\Pi}(h(t_1,\cdot))(t_2) \ud^{\times}t_1 \ud^{\times}t_2.
\end{equation}

	The functions $\widetilde{h}(\chi; s_1,s_2)$, $h^*(t_1,t_2)$, $h(t_1,y)$ and $h(\pi)$ are distributions/functionals which satisfy certain invariance properties. In fact, by the general theory of Kirillov models, the $\psi$-Kirillov model $\Kir(\Pi^{\infty},\psi)$ contains $\Cont_c^{\infty}(\gp{N}_2(\F) \backslash \GL_2(\F), \psi)$. We can assume that for some $f \in \Cont_c^{\infty}(\GL_2(\F))$ we have
\begin{equation} \label{eq: InvDisConst}
	W \begin{pmatrix} g & \\ & 1 \end{pmatrix} = \int_{\F} \psi(-x) f \left( \begin{pmatrix} 1 & x \\ & 1 \end{pmatrix} g \right) \ud x.
\end{equation}
	The weight function
\begin{align}
	h(\pi) &= \sideset{}{_{e \in \Bas(\pi)}} \sum \int_{\gp{N}_2(\F) \backslash \GL_2(\F)} W\begin{pmatrix} g & \\ & 1 \end{pmatrix} W_{e^{\vee}}(g) \ud g \cdot W_e(\mathbbm{1}) \label{eq: WtIsBesselT} \\
	&= \sideset{}{_{e \in \Bas(\pi)}} \sum \int_{\GL_2(\F)} f(g) W_{e^{\vee}}(g) \ud g \cdot W_e(\mathbbm{1}) = \BesselD_{\widetilde{\pi},\psi^{-1}}(f) \nonumber \\
	&= \int_{(\F^{\times})^2} h(t_1,y) \omega^{-1}(t_1) \BesselF_{\widetilde{\pi},\psi^{-1}} \begin{pmatrix} & -y \\ 1 & \end{pmatrix} \frac{\ud^{\times}y}{\norm[y]} \ud^{\times}t_1 \nonumber
\end{align}
	becomes the Bessel distribution for the contragredient of $\pi$ applied to $f$: $\BesselF_{\pi,\psi}$ is the locally integrable function representing the Bessel distribution $\BesselD_{\pi,\psi}$; and the function
\begin{align} \label{eq: RelOI}
	h(t_1,y) &= h(t_1,y;f) = \omega_{\Pi}(t_1) \int_{\F} W \left( \begin{pmatrix} y & & \\ x & 1 & \\ & & 1 \end{pmatrix} \begin{pmatrix} & -1 & 1 \\ 1 & & \\ & & 1 \end{pmatrix} \begin{pmatrix} t_1^{-1} & & \\ & t_1^{-1} & \\ & & 1 \end{pmatrix} \right) \ud x \\
	&:= \omega_{\Pi}(t_1) \int_{\F^2} f \left( \begin{pmatrix} 1 & x_1 \\ & 1 \end{pmatrix} \begin{pmatrix} & -y \\ 1 & \end{pmatrix} \begin{pmatrix} 1 & x_2 \\ & 1 \end{pmatrix} \begin{pmatrix} t_1^{-1} & \\ & t_1^{-1} \end{pmatrix} \right) \psi(-x_1-x_2) \ud x_1 \ud x_2 \nonumber
\end{align}
	is simply the relative orbital integral for the Bessel distributions. Since $\widetilde{h}(\chi)$ and $h^*(t_1,t_2)$ are integral transforms of $h(t_1,y)$, all these functions are (extensions of) distributions $\Theta$ on $f \in \Cont_c^{\infty}(\GL_2(\F))$ satisfying
\begin{equation} \label{eq: InvDisModel}
	\Theta(\rpL_{n(u_1)} \rpR_{n(u_2)} f) = \psi(-u_1+u_2) \Theta(f), \quad \forall u_1, u_2 \in \F.
\end{equation}

	Moreover note that only the function
\begin{equation} \label{eq: InvDis2IntCenter}
	H(y) = \int_{\F^{\times}} h(t_1,y) \omega\omega_{\Pi}(t_1)^{-1} \ud^{\times}t_1,
\end{equation}
	not the function $h(t_1,y)$, can be recovered from the weight functions $h(\pi)$ via a suitable \emph{Bessel inversion transform}. However, the function $H$ lies beyond the applicability of the usual $\Vor_{\Pi}$, not even $\MS_{\Pi,\psi}$. We shall establish the following version in the case of tempered $\Pi$
\begin{equation} \label{eq: DWtORelBis}
	\widetilde{h}(\chi) = \int_{\F^{\times}} \psi(-t_2) \chi^{-1}(t_2) \norm[t_2]^{-\frac{1}{2}} \widetilde{\Vor}_{\Pi}(H)(t_2) \ud^{\times}t_2.
\end{equation}

\begin{lemma} \label{lem: IntWtsEst}
	Suppose $\Pi$ is unitary and tempered. We have the bounds for any $\epsilon > 0$
	$$ \int_{\F^{\times}} \extnorm{h(t,y)} \ud^{\times} t \ll_{\epsilon} \min \left( \norm[y]_{\F}^{\frac{1}{2}+\epsilon}, \norm[y]_{\F}^{1-\epsilon} \right), \quad \int_{\F^{\times}} \extnorm{h^*(t,y)} \ud^{\times} t \ll_{\epsilon} \min \left( \norm[y]_{\F}^{\epsilon}, \norm[y]_{\F}^{1-\epsilon} \right). $$
\end{lemma}
\begin{proof}
	Proposition \ref{prop: BdWhiTemp} implies the bounds for some $0 \leq \phi_j \in \Sch(\F^2)$ and $d \in \Z_{\geq 0}$
	$$ \extnorm{W \left( n z \begin{pmatrix} t_1t_2 & & \\ & t_2 & \\ & & 1 \end{pmatrix} \kappa \right)} \leq \norm[t_1] \norm[t_2] \left( 1+ (\log \norm[t_1])^2 \right)^d \left( 1+ (\log \norm[t_2])^2 \right)^d \phi_1(t_1,t_2), $$
	$$ \extnorm{\widetilde{W} \left( n z \begin{pmatrix} t_1t_2 & & \\ & t_2 & \\ & & 1 \end{pmatrix} \kappa \right)} \leq \norm[t_1] \norm[t_2] \left( 1+ (\log \norm[t_1])^2 \right)^d \left( 1+ (\log \norm[t_2])^2 \right)^d \phi_2(t_1,t_2), $$
	valid for any $n \in \gp{N}_3(\F), \kappa \in \gp{K}_3, z, t_1, t_2 \in \F^{\times}$. We abbreviate $\norm = \norm_{\F}$ for simplicity of notation.
	
\noindent (1) We first consider the non-archimedean case. We have the Iwasawa decomposition of
	$$ \begin{pmatrix} y & & \\ x & 1 & \\ & & 1 \end{pmatrix} \begin{pmatrix} & -1 & t \\ 1 & & \\ & & t \end{pmatrix} = \begin{cases}
		\begin{pmatrix} -y & & yt \\ & -1 & xt \\ & & t \end{pmatrix} \begin{pmatrix} & 1 & \\ -1 & x & \\ & & 1 \end{pmatrix} & \text{if } x \in \vO_{\F} \\
		\begin{pmatrix} -yx^{-1} & -y & yt \\ & -x & xt \\ & & t \end{pmatrix} \begin{pmatrix} 1 & & \\ -x^{-1} & 1 & \\ & & 1 \end{pmatrix} & \text{if } x \notin \vO_{\F}
	\end{cases}. $$
	From the integral representation \eqref{eq: InvDis2} of $h(t,y)$ we get
\begin{multline*} 
	\int_{\F^{\times}} \extnorm{h(t,y)} \ud^{\times}t \leq \Vol(\vO_{\F}) \cdot \norm[y] \left( 1 + (\log \norm[y])^2 \right)^d \int_{\F^{\times}} \norm[t] \left( 1 + (\log \norm[t])^2 \right)^d \phi_1(y, -t) \ud^{\times}t \\
	+ \int_{\F^{\times}} \int_{\vP_{\F}} \norm[yxt^{-1}] \left( 1 + (\log \norm[yx^2])^2 \right)^d \left( 1 + (\log \norm[xt])^2 \right)^d \phi_1(yx^2, -x^{-1}t^{-1}) \norm[x]^{-2} \ud x \ud^{\times} t \\
	\ll_{\epsilon} \norm[y] \left( 1 + (\log \norm[y])^2 \right)^d \phi_1(y) + \norm[y] \int_{\vP_{\F}} \min \left( \left( \tfrac{1}{\norm[yx^2]} \right)^{\epsilon}, \left( \tfrac{1}{\norm[yx^2]} \right)^{\frac{1}{2}-\epsilon} \right) \ud x \ll_{\epsilon} \min \left( \norm[y]^{\frac{1}{2}+\epsilon}, \norm[y]^{1-\epsilon} \right)
\end{multline*}
	for some positive $\phi_1 \in \Sch(\F)$ by Lemma \ref{lem: WeilSchwartzEnv}. Similarly we have the Iwasawa decomposition of
	$$ \begin{pmatrix} y & \\ & w_2 \end{pmatrix} \begin{pmatrix} & -1 & t \\ 1 & & \\ & & t \end{pmatrix}^{\iota} = \begin{cases}
		\begin{pmatrix} -y & & \\ & t^{-1} & \\ & & 1 \end{pmatrix} \begin{pmatrix} & 1 & \\ & t & 1 \\ 1 & & \end{pmatrix} & \text{if } t \in \vO_{\F} \\
		\begin{pmatrix} -yt^{-1} & y & \\ & 1 & \\ & & 1 \end{pmatrix} \begin{pmatrix} & & 1 \\ & 1 & t^{-1} \\ 1 & & \end{pmatrix} & \text{if } t \notin \vO_{\F}
	\end{cases}. $$
	From the definition \eqref{InvDis1} of $h^*(t,y)$ we get
\begin{multline*}
	\int_{\F^{\times}} \extnorm{h^*(t,y)} \ud^{\times}t \leq \int_{\vO_{\F}} \norm[y] \left( 1 + (\log \norm[yt])^2 \right)^d \left( 1 + (\log \norm[t])^2 \right)^d \phi_2(-yt, t^{-1}) \ud^{\times} t \\
	+ \int_{\vP_{\F}} \norm[yt] \left( 1 + (\log \norm[yt])^2 \right)^d \phi_2(-yt, 1) \ud^{\times} t \\
	\ll_{\epsilon} \norm[y] \int_{\vO_{\F}} \min \left( \left( \tfrac{1}{\norm[yt]} \right)^{1-\epsilon} \norm[t], \left( \tfrac{1}{\norm[yt]} \right)^{\epsilon} \norm[t]^{2\epsilon} \right) \ud^{\times} t + \\
	\int_{\vP_{\F}} \norm[yt] \min \left( \left( \tfrac{1}{\norm[yt]} \right)^{1-\epsilon}, \left( \tfrac{1}{\norm[yt]} \right)^{\epsilon} \right) \ud^{\times} t  \ll_{\epsilon} \min \left( \norm[y]^{\epsilon}, \norm[y]^{1-\epsilon} \right).
\end{multline*}

\noindent (2) We then consider the real case. We have the Iwasawa decomposition of
	$$ \begin{pmatrix} y & & \\ x & 1 & \\ & & 1 \end{pmatrix} \begin{pmatrix} & -1 & t \\ 1 & & \\ & & t \end{pmatrix} = \begin{pmatrix} -\tfrac{y}{\sqrt{1+x^2}} & -\tfrac{yx}{\sqrt{1+x^2}} & yt \\ & -\sqrt{1+x^2} & xt \\ & & t \end{pmatrix} \begin{pmatrix} \tfrac{x}{\sqrt{1+x^2}} & \tfrac{1}{\sqrt{1+x^2}} & \\ -\tfrac{1}{\sqrt{1+x^2}} & \tfrac{x}{\sqrt{1+x^2}} & \\ & & 1 \end{pmatrix}. $$
	From the integral representation \eqref{eq: InvDis2} of $h(t,y)$ we get
\begin{align*} 
	\int_{\R^{\times}} \extnorm{h(t,y)} \ud^{\times}t &\leq \norm[y] \int_{\R^{\times}} \int_{\R} \tfrac{1}{\norm[t]\sqrt{1+x^2}} \left( 1 + \left( \log \tfrac{\norm[y]}{1+x^2} \right)^2 \right)^d \left( 1 + \left( \log \tfrac{\sqrt{1+x^2}}{\norm[t]} \right)^2 \right)^d \cdot \\
	&\qquad \qquad \qquad \phi_1 \left( \tfrac{y}{1+x^2}, - \tfrac{\sqrt{1+x^2}}{t} \right) \ud x \ud^{\times} t \\
	&\leq \norm[y] \int_{\R} \tfrac{1}{1+x^2} \left( 1 + \left( \log \tfrac{\norm[y]}{1+x^2} \right)^2 \right)^d \phi_1 \left( \tfrac{y}{1+x^2} \right) \ud x \\
	&\ll_{\epsilon} \norm[y] \int_{\R} \tfrac{1}{1+x^2} \min \left( \left( \tfrac{1+x^2}{\norm[y]} \right)^{\frac{1}{2}-\epsilon}, \left( \tfrac{1+x^2}{\norm[y]} \right)^{\epsilon} \right) \ud x \ll_{\epsilon} \min(\norm[y]^{\frac{1}{2}+\epsilon}, \norm[y]^{1-\epsilon}).
\end{align*}
	for some positive $\phi_1 \in \Sch(\R)$ by Lemma \ref{lem: WeilSchwartzEnv}. Similarly we have the Iwasawa decomposition of
	$$ \begin{pmatrix} y & \\ & w_2 \end{pmatrix} \begin{pmatrix} & -1 & t \\ 1 & & \\ & & t \end{pmatrix}^{\iota} = \begin{pmatrix} -\tfrac{y}{\sqrt{1+t^2}} & -\tfrac{yt}{\sqrt{1+t^2}} & \\ & \tfrac{\sqrt{1+t^2}}{t} & \\ & & 1 \end{pmatrix} \begin{pmatrix} & \tfrac{1}{\sqrt{1+t^2}} & -\tfrac{t}{\sqrt{1+t^2}} \\ & \tfrac{t}{\sqrt{1+t^2}} & \tfrac{1}{\sqrt{1+t^2}} \\ 1 & & \end{pmatrix}. $$
	From the definition \eqref{InvDis1} of $h^*(t,y)$ we get for any $\epsilon > 0$
\begin{multline*}
	\int_{\R^{\times}} \extnorm{h^*(t,y)} \ud^{\times}t \leq \norm[y] \int_{\R^{\times}} \tfrac{1}{\sqrt{1+t^2}} \left( 1 + \left( \log \tfrac{\norm[yt]}{1+t^2} \right)^2 \right)^d \left( 1 + \left( \log \tfrac{\sqrt{1+t^2}}{\norm[t]} \right)^2 \right)^d \phi_2 \left( -\tfrac{yt}{1+t^2}, \tfrac{\sqrt{1+t^2}}{t} \right) \ud^{\times} t \\
	\ll_{\epsilon} \norm[y] \int_{\R^{\times}} \min \left( \left( \tfrac{1+t^2}{\norm[yt]} \right)^{1-\epsilon} \left( \tfrac{\norm[t]}{\sqrt{1+t^2}} \right), \left( \tfrac{1+t^2}{\norm[yt]} \right)^{\epsilon} \left( \tfrac{\norm[t]}{\sqrt{1+t^2}} \right)^{2\epsilon} \right) \tfrac{\ud^{\times}t}{\sqrt{1+t^2}} \ll_{\epsilon} \min(\norm[y]^{\epsilon}, \norm[y]^{1-\epsilon}).
\end{multline*}

\noindent (3) Since the complex case is quite similar to the real one, we omit the details.
\end{proof}

\begin{proof}[Proof of Theorem \ref{thm: DWtFTemp}]
	It is equivalent to proving \eqref{eq: DWtORelBis}. By \eqref{eq: InvDisRel2} and the characterizing property of the Voronoi--Hankel transform \eqref{eq: ExtVHComp} we have for any $\Phi \in \Sch(\Mat_3(\F))$ and $\beta \in C(\pi)$ the equality
\begin{equation} \label{eq: IntWtsRel}
	\int_{\GL_3(\F)} \widehat{\Phi}(g) \beta^{\iota}(g) \norm[\det g] \cdot h(t, \det g) \ud g = \int_{\GL_3(\F)} \Phi(g) \beta(g) \norm[\det g]\cdot h^*(t,\det g) \ud g.
\end{equation}
	By Lemma \ref{lem: IntWtsEst} and Proposition \ref{prop: GJZAbsCTemp} we have
	$$ \int_{\F^{\times}} \int_{\GL_3(\F)} \extnorm{\widehat{\Phi}(g) \beta^{\iota}(g)} \norm[\det g] \cdot \extnorm{h(t, \det g)} \ud g \ud^{\times} t \ll_{\epsilon} \int_{\GL_3(\F)} \extnorm{\widehat{\Phi}(g) \beta^{\iota}(g)} \norm[\det g]^{2-\epsilon} \ud g < \infty, $$
	$$ \int_{\F^{\times}} \int_{\GL_3(\F)} \extnorm{\Phi(g) \beta(g)} \norm[\det g] \cdot \extnorm{h^*(t,\det g)} \ud g \ud^{\times} t \ll_{\epsilon} \int_{\GL_3(\F)} \extnorm{\Phi(g) \beta(g)} \norm[\det g]^{2-\epsilon} \ud g < \infty. $$
	Hence we can integrate \eqref{eq: IntWtsRel} to get
\begin{multline*}
	\int_{\GL_3(\F)} \widehat{\Phi}(g) \beta^{\iota}(g) \norm[\det g] \cdot \left( \int_{\F^{\times}} h(t, \det g) \omega\omega_{\Pi}(t)^{-1} \ud^{\times} t \right) \ud g \\
	= \int_{\F^{\times}} \omega\omega_{\Pi}(t)^{-1} \left( \int_{\GL_3(\F)} \widehat{\Phi}(g) \beta^{\iota}(g) \norm[\det g] \cdot h(t, \det g) \ud g \right) \ud^{\times} t \\
	= \int_{\F^{\times}} \omega\omega_{\Pi}(t)^{-1} \left( \int_{\GL_3(\F)} \Phi(g) \beta(g) \norm[\det g] \cdot h^*(t,\det g) \ud g \right) \ud^{\times} t \\
	= \int_{\GL_3(\F)} \Phi(g) \beta(g) \norm[\det g] \cdot \left( \int_{\F^{\times}} h^*(t,\det g) \omega\omega_{\Pi}(t)^{-1} \ud^{\times} t \right) \ud g,
\end{multline*} 
	and conclude the proof of Theorem \ref{thm: DWtFTemp} by Theorem \ref{thm: ExtVorH}.
\end{proof}

\section{Voronoi--Hankel Kernel Functions for $\GL_2$}

	\subsection{Dihedral Case}
	
\begin{definition} \label{def: AdmPExt}
	Let $\F$ be a local field. Let $P_2(\F)$ be the set of $(\E/\F, \eta)$ where
\begin{itemize}
	\item $\E$ a quadratic field extension of $\F$ with the non-trivial Galois action denoted by $v \mapsto \bar{v}$,
	\item $\eta: \E^{\times} \to \C^{\times}$ is a quasi-character which does not factor through $\Nr_{\E/\F}$.
\end{itemize}
	Write $P_2^u(\F)$ to be the subset of $(\E/\F, \eta) \in P_2(\F)$ with unitary $\eta$. Write $P_1(\F)$, resp. $P_1^u(\F)$ to be the set of quasi-, resp. unitary characters of $\F^{\times}$.
\end{definition}

\noindent Let $\psi$ be a non-trivial character of $\F$. The Weil's representation $\pi(\eta,\psi)$ of $\GL_2(\F)^+$, the subgroup of $\GL_2(\F)$ consisting of matrices with determinant in $\Nr(\E^{\times})$, is realized in
	$$ \Sch(\E, \eta^{-1}) := \left\{ \Phi \in \Sch(\E) \ \middle| \ \Phi(yv) = \eta(y)^{-1} \Phi(y), \ \forall y \in \E^1 \right\}, $$
	with the formulas
\begin{equation} \label{eq: DiWeilRepF}
\begin{cases}
	\left( \pi(\eta,\psi) \begin{pmatrix} \Nr(b) & \\ & 1 \end{pmatrix} \Phi \right)(v) &= \extnorm{\Nr(b)}^{\tfrac{1}{2}} \eta(b) \Phi(bv), \\
	\left( \pi(\eta,\psi) \begin{pmatrix} 1 & x \\ & 1 \end{pmatrix} \Phi \right)(v) &= \psi(x \Nr(v)) \Phi(v), \\
	\left( \pi(\eta,\psi) \begin{pmatrix} a & \\ & a^{-1} \end{pmatrix} \Phi \right)(v) &= \norm[a] \eta_{\E}(a) \Phi(av), \\
	\pi(\eta,\psi)(w') \Phi &= \lambda(\E/\F,\psi) \cdot \Phi^{\iota}, \quad w' := \begin{pmatrix} & 1 \\ -1 & \end{pmatrix}
\end{cases}
\end{equation}
	where $\eta_{\E}$ is the non-trivial quadratic character trivial on $\Nr(\E^{\times})$, $\lambda(\E/\F,\psi)$ is the \emph{Weil index} and
\begin{equation} \label{eq: TwistFour}
	\Phi^{\iota}(v) := \int_{\E} \Phi(u) \psi(\Tr(u\bar{v})) \ud u = \invOFour_{\E}(\Psi)(\bar{v}). 
\end{equation}
	Then $(\pi_{\eta}, \Sch)$ is induced from $(\pi(\eta,\psi), \Sch(\E, \eta^{-1}))$ and is independent of a choice of $\psi$.
	$$ \Sch = \Ind_{\GL_2(\F)^+}^{\GL_2(\F)} \Sch(\E, \eta^{-1}) = \left\{ f: \GL_2(\F) \to \Sch(\E,\eta^{-1}) \ \middle| \ f(hg) = \pi(\eta,\psi)(h).f(g), \quad \forall h \in \GL_2(\F)^+ \right\}. $$
	In fact $\Sch$ is the subspace of smooth vectors, on which a non-trivial $\psi$-Whittaker functional is given by
	$$ \ell: \Sch \to \C, \quad f \mapsto \ell(f) := f(\id)(1). $$
\begin{definition} \label{def: NrRepsIndMap}
	Let $\boldsymbol{\epsilon} = \left\{ \epsilon_1,  \epsilon_2 \right\}$ be a set of representatives for $\F^{\times} / \Nr(\E^{\times})$. Let the norm one subgroup be $\E^1 = \left\{ b \in \E^{\times} \ \middle| \ \Nr(b) = 1 \right\}$. Define
	$$ I_{\E} = I_{\E,\boldsymbol{\epsilon}}: \intL^1(\F^{\times}) \to \intL^1(\E^{\times}/\E^1)^{\oplus 2}, \quad I_{\E}(f)(b) := (f(\Nr(b)\epsilon_1), f(\Nr(b)\epsilon_2))^T; $$
	$$ J_{\E} = J_{\E,\boldsymbol{\epsilon}}: \intL^1(\F^{\times}) \to \intL^1(\E^{\times}/\E^1)^{\oplus 2}, \quad J_{\E}(f)(b) := (f(\Nr(b)\epsilon_1^{-1}), f(\Nr(b)\epsilon_2^{-1}))^T. $$
	We also write $I_{\E}(f)_j$, resp. $J_{\E}(f)_j$, for the $j$-th component of $I_{\E}(f)$, resp. $J_{\E}(f)$.
\end{definition}

\begin{lemma} \label{lem: VHFourDDi}
	Let $(\E/\F, \eta) \in P_2(\F)$ be an admissible pair. Then we have
	$$ J_{\E} \circ \VorH_{\pi_{\eta}} = \lambda(\E/\F, \psi) \cdot \VorH_{\eta} \circ I_{\E}. $$
\end{lemma}
\begin{proof}
	Let $f \in \Sch$. Write $\Phi_j = f \begin{pmatrix} \epsilon_j & \\ & 1 \end{pmatrix} \in \Sch(\E, \eta^{-1})$. Consider
	$$ h(t) := \norm[t]^{-\frac{1}{2}} W_f \begin{pmatrix} t & \\ & 1 \end{pmatrix} = \norm[t]^{-\frac{1}{2}} f \begin{pmatrix} t & \\ & 1 \end{pmatrix}(1). $$
	By definition we have for $j \in \{ 1,2 \}$
	$$ h(\Nr(b) \epsilon_j) = \extnorm{\Nr(b)\epsilon_j}^{-\frac{1}{2}} \pi(\eta,\psi) \begin{pmatrix} \Nr(b) & \\ & 1 \end{pmatrix}.\Phi_j(1) = \norm[\epsilon_j]^{-\frac{1}{2}} \eta(b) \Phi_j(b). $$
	Therefore we get
\begin{equation} \label{eq: DiHIE}
	I_{\E}(h)(b) = \begin{pmatrix} \norm[\epsilon_1]^{-\frac{1}{2}} \eta(b) \Phi_1(b) \\ \norm[\epsilon_2]^{-\frac{1}{2}} \eta(b) \Phi_2(b) \end{pmatrix}.
\end{equation}
	Similarly consider
	$$ \widetilde{h}(t) := \norm[t]^{\frac{1}{2}} \widetilde{W}_f \begin{pmatrix} -t & \\ & 1 \end{pmatrix} = \norm[t]^{\frac{1}{2}} W_f \left( w' \begin{pmatrix} t^{-1} & \\ & 1 \end{pmatrix} \right) = \norm[t]^{\frac{1}{2}} f \left( w' \begin{pmatrix} t^{-1} & \\ & 1 \end{pmatrix} \right)(1). $$
	By definition we have for $j \in \{ 1,2 \}$
\begin{multline*} 
	\widetilde{h}(\Nr(b) \epsilon_j^{-1}) = \extnorm{\Nr(b)\epsilon_j^{-1}}^{\frac{1}{2}} f \left( \begin{pmatrix} 1 & \\ & \Nr(b)^{-1} \end{pmatrix} w' \begin{pmatrix} \epsilon_j & \\ & 1 \end{pmatrix} \right)(1) \\
	= \extnorm{\Nr(b)\epsilon_j^{-1}}^{\frac{1}{2}} \cdot \extnorm{\Nr(b)}^{\frac{1}{2}} \eta^{-1}(b) f \left( w' \begin{pmatrix} \epsilon_j & \\ & 1 \end{pmatrix} \right)(\bar{b}) = \norm[\epsilon_j]^{-\frac{1}{2}} \eta^{-1}(b) \norm[b]_{\E} \cdot \invOFour_{\E}(\Phi_j)(b).
\end{multline*}
Therefore we get
\begin{equation} \label{eq: DiHJE}
	J_{\E}(\widetilde{h})(b) = \begin{pmatrix} \norm[\epsilon_1]^{-\frac{1}{2}} \eta^{-1}(b) \norm[b]_{\E} \cdot \invOFour_{\E}(\Phi_1)(b) \\ \norm[\epsilon_2]^{-\frac{1}{2}} \eta^{-1}(b) \norm[b]_{\E} \cdot \invOFour_{\E}(\Phi_2)(b) \end{pmatrix}.
\end{equation}
	Comparing \eqref{eq: DiHIE} and \eqref{eq: DiHJE} we conclude.
\end{proof}

\begin{corollary} \label{cor: VorHKDi}
	The Voronoi--Hankel transform $\VorH_{\pi_{\eta}}$ is of convolution type with kernel defined by
	$$ \VHF_{\pi_{\eta}}(t) := \zeta_{\E}(1)^{-1} \lambda(\E/\F,\psi) \id_{\Nr(\E^{\times})}(t) \cdot \norm[t]_{\F} \int_{\E^1} \psi(x \delta) \eta^{-1}(x \delta) \ud \delta, $$
where $x \in \E$ is any element with $\Nr_{\E/\F}(x)=t$, and the Haar measure $\ud \delta$ on $\E^1$ is chosen so that the quotient measure on $\E^{\times}/\E^1 \simeq \Nr(\E^{\times})$ coincides with the restriction of the Haar measure $\ud^{\times} t$ on $\F^{\times}$.
\end{corollary}
\begin{proof}
	Let $h \in \Cont_c^{\infty}(\F^{\times})$ and write $h^* = \VorH_{\pi_{\eta}}(h)$. By Lemma \ref{lem: VHFourDDi} we have for any $1 \leq j \leq 2$
\begin{align*}
	h^*(\Nr(x) \epsilon_j^{-1}) &= \lambda(\E/\F,\psi) \eta^{-1}(x) \norm[x]_{\E} \int_{\E} \psi(xy) \eta^{-1}(y) h(\Nr(y)\epsilon_j) \ud_{\E} y \\
	&= \zeta_{\E}(1)^{-1} \lambda(\E/\F,\psi) \int_{\E^{\times}} \psi(xy) \eta^{-1}(xy) \norm[xy]_{\E} h(\Nr(y)\epsilon_j) \ud_{\E}^{\times} y \\
	&= \int_{\E^{\times}/\E^1} \VHF_{\pi_{\eta}}(\Nr(xy)) h(\Nr(y)\epsilon_j) \ud_{\E}^{\times} y.
\end{align*}
	Now that $\epsilon_j$ form a system of representatives of $\F^{\times} / \Nr(\E^{\times})$, we deduce
	$$ h^*(t_2) = \int_{\F^{\times}} \VHF_{\pi_{\eta}}(t_2t_1) h(t_1) \ud^{\times} t_1, $$
	proving that $\VorH_{\pi_{\eta}}$ is of convolution type with kernel $\VHF_{\pi_{\eta}}$.
\end{proof}

	\subsection{Induced Case}
	
	Suppose $\pi = \mu_1 \boxplus \mu_2$ is induced from the Borel subgroup, and the two quasi-characters $\mu_i$ of $\F^{\times}$ satisfy $\norm[\mu_i(t)]=\norm[t]_{\F}^{\sigma_i}$ with $\sigma_1+\sigma_2=0$ and $\norm[\sigma_i] \leq \RamCst < 1/2$. The Godement section is based on $\Phi \in \Sch(1 \times 2, \F)$ so that any $W \in \Whi(\pi^{\infty}, \psi)$ has the integral representation
\begin{equation} \label{eq: WhiGoSecGL2}
	W(h) = \mu_2(\det(h)) \norm[\det(h)]^{\tfrac{1}{2}} \int_{\F^{\times}} \OFour_2(h.\Phi.t)(1,1) \mu_1^{-1}\mu_2(t) \norm[t] \ud^{\times}t. 
\end{equation}
	It implies the integral representation of
\begin{multline} \label{eq: DualWhiGoSecGL2}
	\widetilde{W}(h) = \mu_2(-1) \mu_2^{-1}(\det(h)) \norm[\det(h)]^{\tfrac{1}{2}} \int_{\F^{\times}} \invOFour_2(h.\widehat{\Phi}.t)(1,1) \mu_1\mu_2^{-1}(t) \norm[t] \ud^{\times}t \\
	= \mu_2(-1) \mu_2^{-1}(\det(h)) \norm[\det(h)]^{-\tfrac{1}{2}} \int_{\F^{\times}} \OFour_1(h^{\iota}.\Phi.t)(1,1) \mu_1^{-1}\mu_2(t) \norm[t] \ud^{\times}t.
\end{multline}
	Hence we get the following integral representations
\begin{multline} \label{eq: VorWtFIntRGL2}
	h(y) := \norm[y]^{-\frac{1}{2}} W \begin{pmatrix} y & \\ & 1 \end{pmatrix} = \mu_2(y) \int_{\F^{\times}} \OFour_2(\Phi.t)(y,1) \mu_1^{-1}\mu_2(t) \norm[t] \ud^{\times}t \\
	= \mu_1(y) \int_{\F^{\times}} \OFour_2(\Phi.t)(1,y) \mu_1^{-1}\mu_2(t) \norm[t] \ud^{\times}t;
\end{multline}
\begin{equation} \label{eq: VorDualWtFIntRGL2}
	h^*(y) := \norm[y]^{\frac{1}{2}} \widetilde{W} \begin{pmatrix} -y & \\ & 1 \end{pmatrix} = \mu_2^{-1}(y) \norm[y] \int_{\F^{\times}} \invOFour_1(\Phi.t)(y,1) \mu_1^{-1}\mu_2(t) \norm[t] \ud^{\times}t.
\end{equation}
	We introduce an intermediate function
\begin{equation} \label{eq: VorMidWtFIntRGL2}
	h^{\OFour}(y) := \mu_1^{-1}(y) \norm[y] \int_{\F^{\times}} (\Phi.t)(1,y) \mu_1^{-1}\mu_2(t) \norm[t] \ud^{\times}t = \mu_2^{-1}(y) \int_{\F^{\times}} (\Phi.t)(y^{-1},1) \mu_1^{-1}\mu_2(t) \norm[t] \ud^{\times}t.
\end{equation}

\begin{lemma} \label{lem: VorTransDecompGL2}
	Let $\OFour$ (resp. $\invOFour$) denote the Fourier transform (resp. its inverse) in the sense of tempered distributions. Then we have the relations
	$$ h^{\OFour} = \Mult_1(\mu_1^{-1}) \circ \invOFour \circ \Mult_0(\mu_1^{-1})(h), \quad h^* = \Mult_{1}(\mu_2^{-1}) \circ \invOFour \circ \Mult_0(\mu_2^{-1}) \circ \Inv(h^{\OFour}). $$
	Equivalently, we get the following decomposition
	$$ \VorH_{\mu_1 \boxplus \mu_2} = \widetilde{\VorH}_{\mu_2} \circ \Inv \circ \widetilde{\VorH}_{\mu_1}. $$
\end{lemma}
\begin{proof}
	Recall $\norm[\mu_j(t)]=\norm[t]^{\sigma_j}$ with $-\RamCst \leq \sigma_j \leq \RamCst$ and $\sigma_1 + \sigma_2 = 0$.
	
\noindent (1) For any test function $\phi \in \Sch(\F)$, we estimate the dominant integrals
\begin{multline*}
	\int_{\F} \extnorm{\phi(y)} \int_{\F^{\times}} \extnorm{\OFour_2(\Phi.t)(1,y) \mu_1^{-1}\mu_2(t)} \norm[t] \ud^{\times}t \ud y = \int_{\F} \extnorm{\phi(y)} \int_{\F^{\times}} \extnorm{\OFour_2(\Phi)(t,yt^{-1})} \norm[t]^{\sigma_2-\sigma_1} \ud^{\times}t \ud y \\
	= \int_{\F^{\times}} \int_{\F} \extnorm{\OFour_2(\Phi)(t,y)} \extnorm{\phi(yt)} \norm[t]^{1+\sigma_2-\sigma_1} \ud y \ud^{\times}t \ll \int_{\F^{\times}} \int_{\F} \extnorm{\OFour_2(\Phi)(t,y)} \norm[t]^{1+\sigma_2-\sigma_1} \ud y \ud^{\times}t < \infty,
\end{multline*}
\begin{multline*}
	\int_{\F} \extnorm{\invOFour(\phi)(y)} \int_{\F^{\times}} \extnorm{(\Phi.t)(1,y) \mu_1^{-1}\mu_2(t)} \norm[t] \ud^{\times}t \ud y = \int_{\F} \extnorm{\invOFour(\phi)(y)} \int_{\F^{\times}} \extnorm{\Phi(t,yt)} \norm[t]^{1+\sigma_2-\sigma_1} \ud^{\times}t \ud y \\
	\ll \int_{\F} \extnorm{\invOFour(\phi)(y)} \int_{\F^{\times}} \Phi_1(t) \norm[t]^{1+\sigma_2-\sigma_1} \ud^{\times}t \ud y < \infty,
\end{multline*}
	where we have applied Proposition \ref{prop: SchRestriction} to bound $\extnorm{\Phi(t,yt)} \ll \Phi_1(t)$ for some positive $\Phi_1 \in \Sch(\F)$ and all $y \in \F$. Hence we can freely change the order of integrals and get
\begin{multline*}
	\int_{\F} \phi(y) \int_{\F^{\times}} \OFour_2(\Phi.t)(1,y) \mu_1^{-1}\mu_2(t) \norm[t] \ud^{\times}t \ud y = \int_{\F^{\times}} \int_{\F} \phi(y) \OFour_2(\Phi.t)(1,y) \mu_1^{-1}\mu_2(t) \norm[t] \ud y \ud^{\times}t \\
	= \int_{\F^{\times}} \int_{\F} \invOFour(\phi)(y) (\Phi.t)(1,y) \mu_1^{-1}\mu_2(t) \norm[t] \ud y \ud^{\times}t = \int_{\F} \invOFour(\phi)(y) \int_{\F^{\times}} (\Phi.t)(1,y) \mu_1^{-1}\mu_2(t) \norm[t] \ud^{\times}t \ud y,
\end{multline*}
	proving the first relation $h^{\OFour} = \Mult_1(\mu_1^{-1}) \circ \invOFour \circ \Mult_0(\mu_1^{-1})(h)$.
	
\noindent (2) For any test function $\phi \in \Sch(\F)$, we estimate the dominant integrals
\begin{multline*}
	\int_{\F} \extnorm{\phi(y)} \int_{\F^{\times}} \extnorm{(\Phi.t)(y,1) \mu_1^{-1}\mu_2(t)} \norm[t] \ud^{\times}t \ud y = \int_{\F} \extnorm{\phi(y)} \int_{\F^{\times}} \extnorm{\Phi(ty,t)} \norm[t]^{1+\sigma_2-\sigma_1} \ud^{\times}t \ud y \\
	\ll \int_{\F} \extnorm{\phi(y)} \int_{\F^{\times}} \Phi_2(t) \norm[t]^{1+\sigma_2-\sigma_1} \ud^{\times}t \ud y < \infty
\end{multline*}
	for some positive $\Phi_2 \in \Sch(\F)$ satisfying $\extnorm{\Phi(yt,t)} \leq \Phi_2(t)$,
\begin{multline*}
	\int_{\F} \extnorm{\OFour(\phi)(y)} \int_{\F^{\times}} \extnorm{\invOFour_1(\Phi.t)(y,1) \mu_1^{-1}\mu_2(t)} \norm[t] \ud^{\times}t \ud y = \int_{\F} \extnorm{\OFour(\phi)(y)} \int_{\F^{\times}} \extnorm{\invOFour_1(\Phi)(yt^{-1},t)} \norm[t]^{\sigma_2-\sigma_1} \ud^{\times}t \ud y \\
	= \int_{\F^{\times}} \int_{\F} \extnorm{\OFour_1(\Phi)(y,t)} \extnorm{\OFour(\phi)(yt)} \norm[t]^{1+\sigma_2-\sigma_1} \ud y \ud^{\times}t \ll \int_{\F^{\times}} \int_{\F} \extnorm{\OFour_1(\Phi)(y,t)} \norm[t]^{1+\sigma_2-\sigma_1} \ud y \ud^{\times}t < \infty.
\end{multline*}
	We deduce the second relation similarly as above and conclude the proof.
\end{proof}

	Consider the special case $h \in \Cont_c^{\infty}(\F^{\times})$. By Lemma \ref{lem: VorTransDecompGL2} we have the formula with absolute convergence
\begin{equation} \label{eq: IntWtFourInt}
	h^{\OFour}(y) = \mu_1^{-1}(y) \norm[y] \int_{\F} \psi(xy) h(x) \mu_1(x)^{-1} \ud x.
\end{equation}
	For any test function $\phi \in \Sch(\F^{\times}) \subset \Sch(\F)$, we insert (\ref{eq: IntWtFourInt}), apply Lemma \ref{lem: VorTransDecompGL2} and Fubini to get
\begin{multline} \label{eq: GL2VorKernDisSplit1}
	\int_{\F} \phi(z) h^*(z) \mu_2(z) \norm[z]^{-1} \ud z = \int_{\F} \invOFour(\phi)(y) h^{\OFour}(y^{-1}) \mu_2^{-1}(y) \ud y \\
	= \lim_{C \to +\infty} \int_{C^{-1} \leq \norm[y] \leq C} \invOFour(\phi)(y) h^{\OFour}(y^{-1}) \mu_2^{-1}(y) \ud y \\
	= \lim_{C \to +\infty} \int_{\F} \phi(z) \int_{\F} \left( \int_{C^{-1} \leq \norm[y] \leq C} \psi(xy^{-1}+yz) \mu_1\mu_2^{-1}(y) \norm[y]^{-1} \ud y \right)  h(x) \mu_1(x)^{-1} \ud x \ud z.
\end{multline}

\begin{lemma} \label{lem: LaplacianCal}
	We calculate some technical derivatives as follows.
\begin{itemize}
	\item[(1)] Let $y > 0$ be a real variable. Let $x \in \R^{\times}$ and $s \in \C$ be parameter, and $g \in \Cont^{\infty}(\R_{>0})$. Consider the smooth function on $\R^{\times}$
	$$ a(y) = a(y; z, s) := g(y^{-1}) e^{2\pi i xy^{-1}} y^{s-1}. $$
	Then for any $n \in \Z_{\geq 0}$ there is a polynomial $P_n \in \Z[i,s][X;u_0,\cdots,u_n;v]$ of $n+3$ variables with coefficients in $\Z[i,s]$ such that
	$$ \partial_y^n a(y) = y^{-n} P_n(y^{-1}; g(y^{-1}), \cdots, g^{(n)}(y^{-1}); 2\pi i x) e^{-2\pi i xy^{-1}} y^{s-1}. $$
	Moreover, the polynomial $P_n$ is homogeneous of degree $1$ in $u_0, \cdots, u_n$, and of degree $n$ in $v$.
	\item[(2)] Let $y = \rho e^{i\theta}$ be a complex variable in the polar coordinates. Let $r>0$, $\alpha \in \R$, $m \in \Z$ and $s \in \C$ be parameters, and $g \in \Cont^{\infty}(\R_{>0})$. Consider the smooth function on $\C^{\times}$
	$$ a(y) = a(y; r, \alpha, m, s) := g(\rho^{-1}) e^{4\pi i r\rho^{-1} \cos(\alpha - \theta)} \rho^{2(s-1)} e^{i m \theta}. $$
	Let $\Delta$ be the Laplacian on $\R^2 \simeq \C$. Then for any $n \in \Z_{\geq 0}$ there is a polynomial $P_n \in \Z[i,s][X;u_0,\cdots,u_{2n};v_1,v_2]$ of $2n+4$ variables with coefficients in $\Z[i,s]$ such that
\begin{multline*} 
	\Delta^n a(y) = \rho^{-2n} \cdot P_n(\rho^{-1}; g(\rho^{-1}), \cdots, g^{(2n)}(\rho^{-1}); 4\pi i r \cos(\alpha-\theta), 4\pi i r \sin(\alpha-\theta)) \\
	\cdot e^{4\pi i r\rho^{-1} \cos(\alpha - \theta)} \rho^{2(s-1)} e^{i m \theta}. 
\end{multline*}
	Moreover, the polynomial $P_n$ is homogeneous of degree $1$ in $u_0, \cdots, u_{2n}$, and of total degree $2n$ in $v_1,v_2$.
\end{itemize}
\end{lemma}
\begin{proof}
	The proof is a simple induction on $n$. We leave the details to the reader.
%
\end{proof}

\begin{lemma} \label{lem: GL2VorKernSplit}
	For any $C > 1$, define
	$$ K_C(x,z) := \mu_2(z)^{-1} \norm[z] \mu_1(x)^{-1} \norm[x] \int_{C^{-1} \leq \norm[y] \leq C} \psi(xy^{-1}+yz) \mu_1\mu_2^{-1}(y) \norm[y]^{-1} \ud y. $$
\begin{itemize}
	\item[(1)] We have the bound uniform in $C \geq 2$
	$$ \extnorm{K_C(x,z)} \ll \norm[z]^{1-\sigma_2} \norm[x]^{1-\sigma_1} \cdot
	\begin{cases}
	 \left( \frac{1+\norm[x]}{\norm[z]} + \frac{1+\norm[z]}{\norm[x]} \right) & \text{if } \F \text{ is archimedean} \\
	 \left( \norm[x] + \norm[x]^{-1} + \norm[z] + \norm[z]^{-1} \right) & \text{if } \F \text{ is non-archimedean}
	\end{cases}. 
	$$
	\item[(2)] The limit $C \to +\infty$ exists and defines a function smooth in $(x,z) \in \F^{\times} \times \F^{\times}$
	$$ K(x,z) := \lim_{C \to +\infty} K_C(x,z). $$
\end{itemize}
\end{lemma}
\begin{proof}
	We prove both (1) and (2) at once, distinguishing different cases of $\F$.

\noindent (\Rmnum{1}) $\F=\R$. We only treat the integral for $y>0$, the one for $y<0$ being similar. Take a smooth partition of unity, namely $f(y)+g(y^{-1})=1$ with $f,g \in \Cont^{\infty}(\R_{>0})$, and
	$$ f(y)=g(y) =  
	\begin{cases}
		1 & \text{if } 0 < y \leq 1/2 \\
		0 & \text{if } y \geq 2
	\end{cases}.
	$$	
	We can break the integral into two parts, which are similar to each other by a change of variables $y \mapsto y^{-1}$. Write $\mu_1\mu_2^{-1}(y)=y^s$ for $y>0$. Then $1\pm \Re(s) > 0$. It suffices to treat the following integral
	$$ I_C(x,z; g) := \int_0^C g(y^{-1}) \psi(xy^{-1}-yz) y^{s-1} \ud y, $$
	Note that $a(y) := g(y^{-1}) \psi(xy^{-1}) y^{s-1}$ is precisely the function considered in Lemma \ref{lem: LaplacianCal} (1). Hence the corresponding $P_n$ are smooth in $y$ with supported contained in $(1/2,+\infty)$, just like $g(y^{-1})$. By integration by parts, we get (for $C \geq 2$)
\begin{multline*} 
	I_C(x,z;g) = \int_0^C a(y) \frac{-1}{2\pi i z} \ud \left(e^{-2\pi i yz}\right) = -\frac{e^{2\pi i(C^{-1}x-Cz)} C^{s-1}}{2\pi i z} \\
	+ \frac{1}{2\pi i z} \int_0^C P_1(y^{-1}; g(y^{-1}), g'(y^{-1}); 2\pi i x) \cdot \psi(xy^{-1}-yz) y^{s-2} \ud y \ll \frac{1+\norm[x]}{\norm[z]}.
\end{multline*}
	This proves (1). To get (2), we apply integration by parts $n$ times
\begin{equation} \label{eq: VorKernGL2ACReal}
	\lim_{C \to +\infty} I_C(x,z) = \frac{1}{(2\pi i z)^n} \int_0^{+\infty} P_n(y^{-1}; g(y^{-1}), \cdots, g^{(n)}(y^{-1}); 2\pi i x) \cdot \psi(xy^{-1}-yz) y^{s-1-n} \ud y.
\end{equation} 
	The right hand side is absolutely convergent and smooth in $x$ and $n$ times differentiable in $z$. Hence the limit is smooth in $(x,z) \in \R^{\times} \times \R^{\times}$.
	
\noindent (\Rmnum{2}) $\F=\C$. Write $\mu_1\mu_2^{-1}(y) = \norm[y]_{\C}^{s-1} [y]^{m}$ for some $s \in \C$ with $\norm[\Re(s)] < 1$ and $[y]=y/\norm[y]$, $m \in \Z$. With similar smooth partition of unity we are reduced to treating the integral ($D_C := \left\{ y \in \C \ \middle| \ \norm[y] \leq C \right\}$)
	$$ I_C(x,z; g) := \int_{D_C} g(\norm[y]^{-1}) \psi(xy^{-1}-yz) \norm[y]_{\C}^{s-1} [y]^{m} \ud y. $$
	If we let $x=r e^{i\alpha}$ and $y=\rho e^{i\theta}$ in the polar coordinates, then $a(y) = g(\norm[y]^{-1}) \psi(xy^{-1}) \norm[y]_{\C}^{s-1} [y]^{m}$ is precisely the function considered in Lemma \ref{lem: LaplacianCal} (2). Hence the corresponding $P_n$ are smooth in $y$ with supported contained in $\left\{ y \in \C \ \middle| \ \norm[y] \geq 1/2 \right\}$, just like $g(\norm[y]^{-1})$. By Green's identity, we get (for $C \geq 2$)
\begin{multline*}
	I_C(x,z;g) = \int_{D_C} a(y) \frac{\Delta \psi(-yz)}{-16\pi^2 \norm[z]_{\C}} \ud y = - \frac{1}{16 \pi^2 \norm[z]_{\C}} \int_{D_C} \widetilde{P}_1(y) \psi(xy^{-1}-yz) \norm[y]_{\C}^{s-2} [y]^{m} \ud y \\
	- \frac{C^{2s-1}}{16 \pi^2 \norm[z]_{\C}} \int_0^{2\pi} \left\{ 4\pi i r \cos(\alpha - \theta) C^{-2} + 2(1-s) C^{-1} - 4\pi i \Re(ze^{i\theta}) \right\} \cdot \psi \left( z C e^{i\theta} - x^{-1} C^{-1} e^{-i\theta} \right) e^{im\theta} \ud \theta \\
	\ll \frac{1+\norm[x]_{\C}}{\norm[z]_{\C}}
\end{multline*}
	where we have written $\widetilde{P}_n(y) := P_n(\norm[y]^{-1}; g(\norm[y]^{-1}), \cdots, g^{(2n)}(\norm[y]^{-1}); 4\pi i r \cos(\alpha-\theta), 4\pi i r\sin(\alpha-\theta))$ for simplicity of notation and have used $\Re(s) \leq 2 \RamCst < 1/2$. To get (2), we apply Green's identity $n$ times
\begin{equation} \label{eq: VorKernGL2ACComp}
	\lim_{C \to +\infty} I_C(x,z;g) = \frac{1}{(-16 \pi^2 \norm[z]_{\C})^n} \int_{\C} \widetilde{P}_n(y) \psi(xy^{-1}-yz) \norm[y]_{\C}^{s-1-n} [y]^{m} \ud y.
\end{equation}
	The right hand side is absolutely convergent and smooth in $x$ and $n$ times differentiable in $z$. Hence the limit is smooth in $(x,z) \in \C^{\times} \times \C^{\times}$.
	
\noindent (\Rmnum{3}) $\F$ is non-archimedean. Let $y=\varpi^n y_0$ with $y_0 \in \vo^{\times}$. If $\norm[y] = q^{-n} \geq \norm[x]$, then $\psi(xy^{-1})=1$. The following Gauss integral is non-vanishing only if $\norm[y] \ll \norm[z]^{-1}$, for otherwise we may take $m \in \Z_{\geq 1}$ large so that $\mu_1\mu_2^{-1}(1+\varpi^m \vo)=1$ and average over $u \in \vo$ for $y \mapsto y(1+\varpi^m u)$ to see its vanishing
	$$ \int_{\varpi^n \vo^{\times}} \psi(-yz) \mu_1\mu_2^{-1}(y) \ud^{\times}y. $$
	Hence the integral defining $K_C(x,z)$ is in fact over $\norm[y] \ll \max(\norm[x], \norm[z]^{-1})$, and also over $\norm[y]^{-1} \ll \max(\norm[z], \norm[x]^{-1})$ by symmetry. Therefore the integral is over
	$$ \min(\norm[x], \norm[z]^{-1}) \ll \norm[y] \ll \max(\norm[x], \norm[z]^{-1}). $$
	The stated bound in (1) follows readily. The integral stabilizes in terms of $\norm[x]$ and $\norm[z]$, hence (2) follows readily.
\end{proof}

\noindent By the dominated convergence theorem and Lemma \ref{lem: GL2VorKernSplit} (1), we get for all $\phi \in \Sch(\F^{\times}), h \in \Cont_c^{\infty}(\F^{\times})$
	$$ \int_{\F} \phi(z) h^*(z) \mu_2(z) \norm[z]^{-1} \ud z = \int_{\F} \phi(z) \mu_2(z) \norm[z]^{-1} \int_{\F} K(x,z)  h(x) \norm[x]^{-1} \ud x \ud z $$
from (\ref{eq: GL2VorKernDisSplit1}). Here $K(x,z)$ is defined in Lemma \ref{lem: GL2VorKernSplit} (2). By the smoothness in $z \in \F^{\times}$, we get
\begin{equation} \label{eq: VHInt}
	h^*(z) = \int_{\F} K(x,z)  h(x) \norm[x]^{-1} \ud x, \quad \forall h \in \Cont_c^{\infty}(\F^{\times}).
\end{equation}

\begin{lemma} \label{lem: VorHKInd}
	(1) We have $K(x,z) = K(1,xz)$. 
	
\noindent (2) The Voronoi--Hankel transform $\VorH_{\mu_1 \boxplus \mu_2}$ is of convolution type with kernel given by
	$$ \VHF_{\mu_1 \boxplus \mu_2}(t) = \zeta_{\F}(1)^{-1} K(1,t) = \zeta_{\F}(1)^{-2} \mu_2(t)^{-1} \norm[t] \lim_{C \to +\infty} \int_{C^{-1} \leq \norm[y] \leq C} \psi(y^{-1}+yt) \mu_1\mu_2^{-1}(y) \ud^{\times} y. $$
\end{lemma}
\begin{proof}
	(1) It is clear from the proof that $K(x,z)$ can be defined with integrals over $a \leq \norm[y] \leq b$ with $a \to 0^+$ and $b \to +\infty$. Hence a change of variable $y \mapsto yx$ gives the desired equality. (2) is a direct consequence of (1) and \eqref{eq: VHInt}.
\end{proof}

\bibliographystyle{acm}
	
\bibliography{mathbib}
	
\end{document}